\documentclass[11pt,reqno,a4paper]{article}

\usepackage{anysize}
\marginsize{3.5cm}{3.5cm}{2.2cm}{2.2cm}

\usepackage[english]{babel}
\usepackage[centertags]{amsmath}
\usepackage{amsfonts,amssymb,amsthm}
\usepackage{hyperref}  % [hypertex]
\usepackage{mathrsfs}

\numberwithin{equation}{section}

\theoremstyle{plain}
\newtheorem{theorem}{Theorem}[section]

\newtheorem{lemma}[theorem]{Lemma}

\newtheorem{corollary}[theorem]{Corollary}
\newtheorem{definition}[theorem]{Definition}

\theoremstyle{definition}
\newtheorem{remarkbase}[theorem]{Remark}
\newenvironment{remark}{\pushQED{\qed} \remarkbase}{\popQED\endremarkbase}

\newtheorem{assu}[theorem]{Assumption}

\DeclareMathOperator{\Op}{Op}

\DeclareMathOperator{\sign}{sign}

\newcommand{\N}{{\mathbb N}}
\newcommand{\R}{{\mathbb R}}
\newcommand{\C}{{\mathbb C}}
\newcommand{\Z}{\mathbb Z}
\newcommand{\Q}{\mathbb Q}
\newcommand{\T}{{\mathbb T}}

\newcommand{\mA}{\mathcal{A}}
\newcommand{\mB}{\mathcal{B}}

\newcommand{\mD}{\mathcal{D}}

\newcommand{\mF}{\mathcal{F}}
\newcommand{\mG}{\mathcal{G}}
\newcommand{\mH}{\mathcal{H}}
\newcommand{\mK}{\mathcal{K}}
\newcommand{\mL}{\mathcal{L}}

\newcommand{\mS}{\mathcal{S}}

\newcommand{\mR}{\mathcal{R}}

\newcommand{\mN}{\mathcal{N}}
\newcommand{\mP}{\mathcal{P}}

\newcommand{\mV}{\mathcal{V}}
\newcommand{\mZ}{\mathcal{Z}}

\newcommand{\pp}{\mathbb{P}}
\newcommand{\ff}{\mathbb{F}}

\renewcommand{\a}{\alpha}
\renewcommand{\b}{\beta}
\newcommand{\g}{\gamma}
\renewcommand{\d}{\delta}

\newcommand{\e}{\varepsilon}

\newcommand{\lm}{\lambda}
\newcommand{\Lm}{\Lambda}

\newcommand{\om}{\omega}
\newcommand{\p}{\pi}
\newcommand{\s}{\sigma}
\renewcommand{\t}{\tau}
\renewcommand{\th}{\vartheta}

\newcommand{\Dy}{|D_y|}

\renewcommand{\Im}{\mathrm{Im}\,}
\renewcommand{\Re}{\mathrm{Re}\,}
\newcommand{\sgn}{\mathrm{sgn}}

\newcommand{\intp}{\int_{0}^{2\p}}

\newcommand{\la}{\langle}
\newcommand{\ra}{\rangle}
\newcommand{\inv}{^{-1}}
\newcommand{\pa}{\partial}

\newcommand{\odd}{\text{odd}}
\newcommand{\even}{\text{even}}
\newcommand{\mult}{m} 

\newcommand{\lip}{ {\mathrm{lip}} } 
 
\newcommand{\Lipe}{ {\mathrm{Lip}(\e)} } 

\def\ba{\begin{aligned}}
\def\ea{\end{aligned}}
\def\beginm{\begin{multline}}
\def\endm{\end{multline}}

\def\B{B}
\def\defn{\mathrel{:=}}

\def\Dx{\lvert D_x\rvert}
\def\eps{\varepsilon}
\def\lA{\left\lVert}

\def\rA{\right\rVert}

\def\xR{\R} %{\mathbf{R}}
\def\xZ{\Z} %{\mathbf{Z}}

\newcommand{\meanT}{\frac{1}{2\p} \int_0^{2\p}}

\title{\Large{\textbf{Gravity capillary standing water waves}}} 
\author{\small{Thomas Alazard and Pietro Baldi}}
\date{} 
\begin{document}
\maketitle

\begin{small}
\textbf{Abstract.}
The paper deals with the 2D gravity-capillary water waves equations in their Hamiltonian formulation, addressing the question of the nonlinear interaction of a plane wave with its reflection off a vertical wall. 
The main result is the construction of small amplitude, \emph{standing} (namely periodic in time and space, and not travelling) solutions of Sobolev regularity, 
for almost all values of the surface tension coefficient, and for a large set of time-frequencies. 
This is an existence result for a quasi-linear, Hamiltonian, reversible system of two autonomous pseudo-PDEs with small divisors. 
The proof is a combination of different techniques, such as a Nash-Moser scheme, 
microlocal analysis, and bifurcation analysis. 
\emph{MSC2010:} 76B15, 76D45, 35B10 (37K50, 37K55, 35S05).
\end{small}

\begin{footnotesize}
\tableofcontents
\end{footnotesize}

\section{Introduction}

This paper deals with the 2D gravity-capillary water waves equations in their Hamiltonian formulation (see equation \eqref{WW}). 
The main result (Theorem \ref{thm:main}) is the construction of small amplitude, \emph{standing} (namely periodic in time and space, and not travelling) solutions of Sobolev regularity, 
for almost all values of the surface tension coefficient, and for a large set of time-frequencies. 
This is an existence result for a quasi-linear system of two autonomous pseudo-PDEs with small divisors.

Before stating precisely the result and describing 
the strategy of the proof, we introduce the problem within a more general framework. 

\medskip

A classical topic in the mathematical theory of hydrodynamics 
concerns the Euler equations for the irrotational flow of an incompressible fluid
in a domain which, at time $t$, is of the form
$$
\Omega(t) = \{ (x,y) \in \xR^{2}\times\xR \, \arrowvert\, y<\eta(t,x) \},
$$
whose boundary is a free surface, which means that $\eta$ is an unknown. 
The simplest type of nontrivial solution for the problem is a \emph{progressive wave}, 
which is a profile of the form $\eta(t,x) = \sigma(k\cdot x-\omega t)$ 
for some periodic function $\sigma\colon \xR\rightarrow \xR$, 
together with a similar property for the velocity field. 
Despite intensive researches on this old subject,  
many natural questions are far from being fully resolved. 
Among these, most questions about the boundary behavior of water waves are not understood. 
Also, the question of the nonlinear interactions of several progressive waves is mostly opened. 

This paper is concerned with these two problems.
We shall study the reflection of a progressive wave off a wall. 
More precisely, we shall study the nonlinear interaction of a 2D gravity-capillary plane wave with 
its reflection off a vertical wall. 
To clarify matters, recall that 2D waves are waves such that the motion is the same in every
vertical section, so that one can consider the two-dimensional motion in one such section (the free surface is then a 1D curve).

\medbreak

\noindent\textbf{Interaction of two gravity-capillary waves.} 
The problem consists in seeking solutions of the water waves equations, periodic in space and time, and such that
\begin{equation}\label{n2}
\eta(t,x) =\eps \cos (k_1\cdot x-\omega(k_1) t)+\eps \cos(k_2\cdot x-\omega(k_2)t)+O(\eps^2),
\end{equation}
together with a similar property for the velocity field, where 
\begin{itemize}
\item $\eps$ is a small parameter which measures the amplitude of the waves;
\item 
$k_1$ and $k_2$ belong to $\xR^2$ and are mirror images 
such that either $k_1 = - k_2$ or $k_1 = (1,\tau)$, $k_2=(1,-\tau)$ for some $\tau$. 
This is the assumption that there is one incident plane wave, say $\eps \cos (k_1\cdot x-\omega(k_1) t)$,  and one reflected wave, $\eps \cos(k_2\cdot x-\omega(k_2)t)$;
\item there holds
$$
\omega(k)\defn \sqrt{g| k|+\kappa | k|^3}
$$
where  
$g>0$ is the acceleration of gravity 
and $\kappa\in [0,1]$ is the surface tension coefficient. Thus 
$k\rightarrow \omega(k)$ is the dispersion relation of the 
water waves equation linearized at the rest position, 
which (after transforming the system into a single equation)
can be written in the form
\begin{equation}\label{n4}
L u \defn \partial_t^2 u + g|D_x| u + \kappa |D_x|^3 u = 0.
\end{equation}
Here $|D_x|$ is the Fourier multiplier defined by $|D_x|e^{ik\cdot x}=|k|e^{ik\cdot x }$, 
where $|k|$ is the Euclidean norm of $k$. 
Note that $\omega(k_1)=\omega(k_2)$.
\end{itemize}

This problem was initiated by Reeder and Shinbrot~\cite{ReSh} and further developed by 
Craig and Nicholls~\cite{CN,CN2} and 
Groves and Haragus~\cite{GrHa2}. They considered the superposition 
of two oblique 2D travelling 
waves, such that $k_1=(1,\tau)$ and $k_2=(1,-\tau)$ for some $\tau$. This 
produces 3D short crested waves. 
Indeed, setting $\omega=\omega(k_1)=\omega(k_2)$ and writing $x=(x_1,x_2)$, one has 
$$
\eps \cos (x_1+\tau x_2-\omega t)+\eps \cos (x_1-\tau x_2-\omega t)=2\eps \cos (x_1-\omega t)\cos (\tau x_2).
$$
Since these waves are propagating in the direction $(1,0)$ (the $x_1$ axis), 
one has to study solutions of the equation that is obtained by replacing 
$\partial_t$ with $-\omega \partial_{x_1}$ in \eqref{n4}, which is 
\begin{equation}\label{n6}
Ku\defn (\omega \partial_{x_1})^2 u + g| D_x|u+\kappa| D_x|^3 u=0.
\end{equation}
For $\kappa>0$, $K$ is an elliptic operator.  
Consequently, in this context, the existence of solutions for the nonlinear equation is a problem in 
bifurcation theory (without small divisors). 

\medbreak

\noindent\textbf{Standing waves.} 
In this paper, we consider the case where the crest of the incident waves are parallel to the wall. This implies that $k_1=-k_2$ and therefore
$$
\eps \cos (k_1\cdot x-\omega(k_1) t)+\eps \cos(k_2\cdot x-\omega(k_2)t)=2\eps \cos (\omega(k_1) t)\cos(k_1\cdot x).
$$
Hence the waves obtained by superimposing the incident and the reflected waves are \emph{standing waves} (namely periodic in time and space, and not travelling). Since standing waves are not travelling, one cannot replace the time derivative by a space derivative. 
This changes dramatically the nature of the problem, as {\em small divisors} appear. 
In this paper we are interested in the case with surface tension $\kappa>0$, 
while in the case without surface tension, namely $\kappa=0$, 
a similar small divisors problem was studied in a series of papers of Iooss, Plotnikov and Toland  
\cite{Iooss-JDDS, IP-SW2, IP-SW1, IP-Mem-2009, IP2, IPT, PlTo}, 
which are described below.

The standing waves we are interested in are 2D waves. 
Without loss of generality, we can assume that $k_1=(\lambda,0)=-k_2$ for some $\lambda \neq 0$. 
Thus we shall consider functions that are independent of $x_2$. In the rest of the paper, $x\in \xR$. 

\medbreak

\noindent\textbf{Small divisors.} 
Let us explain why small divisors enter into the analysis. 
Looking for solutions that are $2\pi$-periodic in space and $(2\pi / \om)$-periodic in time (where the time frequency $\om$ is an unknown of the problem), \eqref{n4} gives 
$$
L e^{i(jx + \omega \ell t)} = p(\omega\ell,j) e^{i(jx + \omega \ell t)}, \quad 
p(\omega\ell,j) := - \om^2 \ell^2 + g |j| + \kappa |j|^3.
$$
In general, namely for almost all values of $\om,\kappa$, 
the eigenvalues $\{ p(\omega \ell,j) \}_{\ell,j\in \xZ}$ of $L$ accumulate to zero. 
To invert $L$ in the orthogonal of its kernel, one finds these small eigenvalues as denominators 
(in fact, small divisors), so that the inverse of $L$ is not a bounded operator, in the sense that it does not map any function space (Sobolev or analytic or H\"older or others) into itself. 
This makes it impossible to apply the standard implicit function theory 
to solve the orthogonal component of the nonlinear problem 
(i.e.\ the range equation, in the language of bifurcation theory). 

\medbreak

\noindent\textbf{Main result.} Our main result is stated in the next section, see Theorem~\ref{thm:main}. 
It asserts that, for almost all values $\kappa$ of the surface tension coefficient, 
for $\eps_0$ small enough there exists a set $\mathcal{G}\subset [0,\eps_0]$ 
of Lebesgue measure greater than $\eps_0(1-C\eps_0^{1/18})$, such that for $\eps$ in $\mathcal{G}$ 
there exists a standing wave whose free surface is of the form \eqref{n2}, or, more precisely, $\eta(t,x) = \e \cos(\om t) \cos(x) + O(\e^2)$, with time-frequency 
$\om = \sqrt{g + \kappa} + O(\e^2)$.
(In Theorem \ref{thm:main} the result is stated precisely for the problem %that 
one obtains after normalizing the gravity constant $g$, rescaling time, and introducing an 
additional amplitude parameter $\xi$, see Section \ref{subsec:main result}.) 

\medskip

Our proof is based on Nash-Moser methods for quasi-linear PDEs on the one side, 
and on techniques of microlocal analysis on the other side. 

Regarding Nash-Moser and KAM theory for quasi-linear PDEs, we remark that in general, as it was proved in the works of Lax, Klainerman and Majda on the formation of singularities 
(see e.g.\ \cite{Klainerman-Majda}), the presence of \emph{nonlinear unbounded} operators --- as it is in our water waves problem --- can compromise the existence of invariant structures of the dynamics like periodic or quasi-periodic solutions.
In fact, the wide existing literature on KAM and Nash-Moser theory for PDEs mainly deals with problems where the perturbation is bounded (see e.g.\ Kuksin \cite{Kuksin-1987} and Wayne \cite{Wayne-1990}; see \cite{Kuksin-Oxford}, \cite{CraigPanorama} for a survey).
For unbounded perturbations where the nonlinear term contains \emph{less} derivatives than the linear one, time-periodic solutions have been obtained by Craig \cite{CraigPanorama} 
and Bourgain \cite{Bourgain-Chicago-99}, while quasi-periodic solutions for PDEs of that type have been constructed via Nash-Moser or KAM methods by Bourgain \cite{Bourgain-94}, 
Kuksin \cite{Kuksin-Oxford}, Kappeler-P\"oschel \cite{Kappeler-Poeschel} for KdV, 
and, more recently, by Liu-Yuan \cite{Liu-Yuan} and Zhang-Gao-Yuan \cite{Zhang-Gao-Yuan} for NLS and Benjamin-Ono, 
and Berti-Biasco-Procesi \cite{Berti-Biasco-Procesi-Ham-DNLW, Berti-Biasco-Procesi-rev-DNLW} for NLW. 

For \emph{quasi-linear} PDEs, namely for equations 
where there are as many derivatives in the nonlinearity as in the linear part (sometimes called ``strongly nonlinear'' PDEs, e.g.\ in \cite{Kuksin-Oxford}), 
the extension of KAM and Nash-Moser theory is a very recent subject, 
which counts very few results. 
Time-periodic solutions for this class of equations have been constructed by Iooss, Plotnikov and Toland for gravity water waves \cite{PlTo,IPT} (which, even more, is a \emph{fully nonlinear} system),
and by Baldi for forced Kirchhoff \cite{Baldi-Kirchhoff} and autonomous Benjamin-Ono equation \cite{Baldi-Benj-Ono}, all using Nash-Moser methods. 
We also mention the pioneering Nash-Moser results of Rabinowitz \cite{RaCPAM} for periodic solutions of fully nonlinear wave equations (where, however, small divisors are avoided by a dissipative term).
The existence (and linear stability) of quasi-periodic solutions for a quasi-linear PDE has been only proved last year by Baldi, Berti and Montalto for forced Airy \cite{BBM-Airy} and autonomous KdV 
\cite{BBM-auto} equations, by Nash-Moser, linear KAM reducibility, and Birkhoff normal forms.  

Regarding the water waves problem, 
in~\cite{PlTo} Plotnikov and Toland proved the existence of pure gravity (i.e.\ $\kappa = 0$) standing waves, periodic in time and space. 
This work has been extended by Iooss, Plotnikov and Toland \cite{IPT} and then by 
Iooss and Plotnikov~\cite{IP-SW1,IP-SW2} who proved the existence of 
unimodal \cite{IPT} and multimodal \cite{IP-SW1,IP-SW2} solutions in Sobolev class via Nash-Moser theory, overcoming the difficulty of a complete resonance of the linearized operator at the origin. 
On the contrary, the gravity-capillary case has an additional parameter, the surface tension coefficient $\kappa$, whose arithmetic properties determine the bifurcation analysis of the linear theory. 
In particular, for all irrational values of $\kappa$ (and therefore for almost all $\kappa$) 
the linearized operator at the origin has a one-dimensional kernel (see section \ref{sec:kernel}).

We also mention the recent proof by Iooss and Plotnikov~\cite{IP-Mem-2009,IP2} of the existence of {\em three-dimensional} periodic progressive gravity waves, obtained with Nash-Moser techniques related to \cite{IP-SW1,IP-SW2}.
The question of the existence of such waves was a well known problem in the theory of surface waves --- we refer the reader to~\cite{Barber,BDM,CN,DiKh,Groves,Iooss-JDDS,IP-Mem-2009} for references and an historical survey of the background of this problem. 

Apart from the dimension of the kernel of the linear problem,  
there are other important differences between the gravity water waves problem, 
as studied by Iooss-Plotnikov-Toland, 
and the gravity-capillary water waves problem studied here. 
The difference is clear at the level of the dispersion relation $\omega=\sqrt{g|k|+\kappa|k|^3}$. 
One could think that the dispersion is stronger for $\kappa>0$ than for $\kappa=0$   
(and this is certainly true for the \emph{linear} part of the problem),
but this does not help the study of the nonlinear problem,
because higher order derivatives also appear in the nonlinearity. 
In fact, this requires the introduction of new techniques. 
A more detailed explanation about which are the new problems emerging in presence of surface tension 
and why the techniques of \cite{IPT} do not work for $\kappa > 0$ is given in the lines below \eqref{model}.

\medskip

To conclude this introduction, let us discuss the main ingredients in our proof.
Applying a Nash-Moser scheme, the main difficulty regards the invertibility of the operator linearized at a nonzero point.  
Like in \cite{IPT,Baldi-Benj-Ono,BBM-Airy,BBM-auto}, 
we seek a sufficiently accurate asymptotic expansion of the eigenvalues in terms of powers of $\eps$ and in terms of inverse powers of the spatial wavelength. 
To do so, 
we conjugate the linearized operator to a constant coefficient operator 
plus a smoothing remainder which can be handled as a perturbation term. 
(We remark that a similar eigenvalues expansion was obtained in \cite{IPT} using inverse powers $\pa_t^{-1}$ of the time-derivative, instead of space-derivative $\pa_x^{-1}$, destroying the structure of dynamical system. 
Such a structure is preserved, instead, by the transformations performed in this paper, as well as those in \cite{Baldi-Benj-Ono,BBM-Airy,BBM-auto}).   
To obtain such a precise knowledge of the asymptotic behavior of the eigenvalues requires 
to find the 
dispersion relationship associated to a variable coefficient equation, which in turn requires microlocal analysis. 
In this direction, we shall follow a now well developed approach in the analysis of water waves equations, 
which consists in 
working with the Craig-Sulem-Zakharov formulation of the equations, introducing the Dirichlet-Neumann operator. 
In particular, we shall use in a crucial way two facts proved by Lannes in \cite{LannesJAMS}. 
Firstly, by introducing what is known as the good unknown of Alinhac (see \cite{Ali,AM}), 
one can overcome an apparent loss of derivative in the analysis of the linearized equation.  
(Another advantage is that, working with the Craig-Sulem-Zakharov formulation \eqref{WW} of the problem as a dynamical system in two unknowns, and thanks to the good unknown of Alinhac, 
here we do not need to introduce any ``approximate inverse'' for the linearized operator, as it was done in \cite{IPT}). 
Secondly, one can use pseudo-differential analysis to study the Dirichlet-Neumann operator 
in domains with limited regularity. In \cite{AM}, this analysis is improved 
by showing that one can paralinearize the Dirichlet-Neumann operator, introducing the paradifferential version 
of the good unknown of Alinhac (see \cite{AliPara}). 
Notice that the analysis of the Cauchy problem for capillary waves requires 
an analysis of the sub-principal terms (see~\cite{MZ,ABZ1}). 
In this paper, we shall use in an essential way the fact that it is also possible 
to symmetrize \emph{sub-sub-principal} terms (the method used below 
can be extended at any order). 
We underline, in particular, the use of a pseudo-differential operator with symbol 
in H\"ormander class $S^0_{\rho,\d}$, $\rho = \d = 1/2$ 
(except for the fact that the symbol here has finite regularity), see \eqref{exotic symbol} and also the discussion about the related model problem \eqref{model}.
Moreover, to apply a Nash-Moser scheme one needs tame estimates (these are 
estimates which are linear with respect to the highest-order norms). We shall use the estimates proved in \cite{AlDe} together with several estimates 
proved in Section \ref{sec:semi-FIO} of this paper using a paradifferential decomposition of the frequencies. 

Eventually, we refer the 
reader to \cite{BGSW,MRT} for recent results establishing the existence 
of progressive waves localized in space. 

\bigskip

The content of the next sections is the following. 
In Section \ref{sec:main result} the water waves equations are written, 
some symmetries are shown (such symmetries play a role in the bifurcation analysis of Section \ref{sec:kernel}, especially to deal with the space-average terms), 
the functional setting is introduced, and the main theorem is stated.
Section \ref{sec:prelim} collects preliminary facts about the Dirichlet-Neumann operator, 
and fix some notations.

In Section \ref{sec:bif} we perform a bifurcation analysis. 
In particular, in Section \ref{sec:approx sol} we construct an ``approximate solution'' $\bar u_\e$, which corresponds to the need of any quadratic Newton scheme of having a sufficiently good ``initial guess''. 
In Section \ref{sec:lin syst} we exploit the (nonlinear) construction of Section \ref{sec:approx sol} to deduce a restriction for the (linear) problem of the inversion of the linearized operator $F'(u)$ at a nonzero point $u$ which is close to $\bar u_\e$. This restriction is given by a linearized version of the Lyapunov-Schmidt approach first (Lemma \ref{lemma:inv bif}), and then by a further restriction with respect to the space frequencies only, closer to a dynamical system point of view 
(Lemma \ref{lemma:inv bif 2}).

In Sections \ref{sec:linearized operator}-\ref{sec:semi-FIO conj} we conjugate the 
linearized operator $F'(u)$ to the sum of a constant coefficient part $\tilde \mD$ and a regularizing, small remainder $\tilde \mR_6$, see \eqref{def mD}. 
In Section \ref{sec:linearized operator} we use (a linear version of) the good unknown of Alinhac. 
In Section \ref{sec:changes} we perform a time-dependent change of the space variable, 
a (space-independent) re-parametrization of the time-variable, and a matrix multiplication, 
to obtain constant coefficients in front of the highest order terms, see \eqref{mL3}. 
In Section \ref{sec:symm top} we symmetrize the highest order terms, 
keeping altogether the few terms that are not small in $\e$ (this is visible in \eqref{def T}).
Section \ref{sec:symm lower} completes the symmetrization procedure, obtaining a symmetric part (see the operator matrix $\mL_5$ in \eqref{symm mL5}) plus a remainder of order $O(\e |D_x|^{-3/2})$,
after solving a block-triangular system of 8 equations in 8 unknowns. 
At this point the $2 \times 2$ real linear system can be written as a single equation for a single complex-valued function $h : \T^2 \to \C$, see \eqref{mL5}. 

Here comes the most interesting part of our conjugation analysis, where it becomes most evident the reason for which the method of \cite{IPT} does not work in presence of surface tension. 
The point can be better explained when reformulated in terms of a modified model problem 
(for the full operator see \eqref{mL5}). 

\smallskip

\textbf{$\bullet$ Model problem:} 
\emph{conjugate the linear operator}
\begin{equation} \label{model}
\om \pa_t + i c |D_x|^{3/2} + a(t,x) \pa_x 
\end{equation}
\emph{to constant coefficients up to order $O(|D_x|^{-3/2})$, knowing that $c$ is a real constant, 
and the variable coefficient $a(t,x)$ is small in size, odd in $t$, and odd in $x$.}
The technique used in \cite{IPT} to eliminate the term $a(t,x) \pa_x$ 
was the one of conjugating the vector field $\om \pa_t + a(t,x) \pa_x$ to $\om \pa_t$ using a suitable change of variable, namely the composition map with a diffeomorphism of the torus 
(this can be done by the method of the characteristics, i.e.\ by solving an ODE). 
However, in \eqref{model}, the change of variables that rectifies $\om \pa_t + a(t,x) \pa_x$ produces a variable coefficient in front of $|D_x|^{3/2}$, 
which is even worst than \eqref{model} for our purposes. A similar effect is obtained by any other Fourier integral operator (FIO) with homogeneous phase function (recall that the changes of variables are special cases of FIOs).
On the other hand, one cannot eliminate 
$a(t,x)\partial_x$ by commuting the equation with 
any multiplication operator, or any other ``standard'' pseudo-differential operator of order zero (with symbol in H\"ormander class $S^0_{\rho,\d}$ with $\rho = 1$, $\d = 0$, see Chapter 7.8 in \cite{Hormander}). 
Indeed, the commutator between such an operator and $|D_x|^{3/2}$ is an operator of order $1/2$, 
which leaves $a(t,x)\partial_x$ unchanged. 
Hence, the algebraic rigidity of the equation forces us 
to commute the equation with a pseudo-differential operator with symbol of type 
$S^0_{\rho,\d}$ with $\rho = \d = 1/2$, see \eqref{exotic symbol}. 
In Section \ref{sec:semi-FIO conj} we calculate the right candidate 
to complete the reduction to constant coefficients up to $O(|D_x|^{-3/2})$.
The study of this operator, namely the proof of its invertibility, its commutators expansion and tame estimates, is developed in Section \ref{sec:semi-FIO}, using pseudo-differential and also para-differential calculus. 

Once the linearized operator has been reduced to constant coefficients up to a domesticated remainder 
(end of Section \ref{sec:semi-FIO conj}),   
its invertibility is straightforward by imposing the first order Melnikov non-resonance conditions (see \eqref{non-res}), and is proved in Section \ref{sec:inv lin op}, where the dependence of the eigenvalues on the parameters is also discussed. 
In Section \ref{sec:NM} we construct a solution of the water waves problem as the limit of a Nash-Moser sequence converging in Sobolev norm, for a large set of parameters, whose Lebesgue measure is estimated in Section \ref{sec:measure}. Finally, Section \ref{sec:utilities} collects some standard technical facts used in the previous sections.

\bigskip

\textbf{Acknowledgements.} We are grateful to G\'erard Iooss, Pavel Plotnikov, Walter Craig, Massimiliano Berti and Jean-Marc Delort for some interesting discussions. 
This research was carried out in the frame of Programme STAR, financially supported by Uni\-NA and Compagnia di San Paolo, and was partially supported by the European Research Council under FP7, 
and PRIN 2012 \emph{Variational and perturbative aspects of nonlinear differential problems}.
Part of the work has been written during some short visits of Baldi to the ENS of Paris. 
This research began at Saint-Etienne de Tin\'ee, and we acknowledge Laurent Stolovitch for the invitation.

\section{The equation and main result} \label{sec:main result}

We recall the Craig-Sulem-Zakharov formulation which allows one to 
reduce the analysis of the Euler equation 
to a problem on the boundary, by 
introducing the Dirichlet--Neumann operator. 

We consider an incompressible liquid occupying a domain 
$\Omega$ with a free surface. Namely, at time $t\ge0$, the fluid domain is
$$
\Omega(t)=\{(x,y)\in \xR\times \xR\,:\, y<\eta(t,x)\}
$$
where $\eta$ is an unknown function. 
We assume that the flow is incompressible and also irrotational, so that 
the velocity field $v$ 
is given by $v=\nabla_{x,y} \phi$ 
for some harmonic velocity potential $\phi\colon \Omega\rightarrow \xR$. 
Following Zakharov~\cite{Zakharov1968}, 
introduce the trace of the potential on the free surface:
$$
\psi(t,x)=\phi(t,x,\eta(t,x)).
$$
Since $\phi$ is harmonic, $\eta$ and $\psi$ fully determines $\phi$. 
Craig and Sulem (see~\cite{CrSu}) observe that one can 
form a system of two evolution equations for $\eta$ and 
$\psi$ by introducing the Dirichlet-Neumann operator $G(\eta)$ 
which relates $\psi$ to the normal derivative 
$\partial_n\phi$.

\begin{definition}\label{defi:DNO}
Given any functions $\eta,\psi\colon \xR\rightarrow \xR$, set 
$\Omega\defn \{\,(x,y)\in\xR\times\xR \,\arrowvert\, y<\eta (x)\,\}$ and define $\phi$ as 
the harmonic extension of $\psi$ in $\Omega$:
\[
\Delta \phi =0 \quad \text{in }\Omega,\quad \phi\arrowvert_{y=\eta}=\psi,\quad 
\nabla\phi \rightarrow 0 \text{ as } y\rightarrow -\infty.
\]
The Dirichlet--Neumann operator is defined by
$$
G(\eta)\psi (x) =
\sqrt{1+\eta_x^2}\,
\partial _n \phi\arrowvert_{y=\eta(x)}
=(\partial_y \phi)(x,\eta(x))-\eta_x (x)\cdot (\partial_x \phi)(x,\eta(x)).
$$
(We denote by $\eta_x$ the derivative $\partial_x \eta$.)
\end{definition}

The water waves equations are a system of two coupled equations: one equation 
describing the deformations of the domain and one equation coming 
from the assumption that the jump of pressure across the free surface is proportional to the mean curvature. 
Using the Dirichlet-Neumann operator, these equations are
\begin{equation}\label{WW}
\left\{
\begin{aligned}
&\partial_t \eta = G(\eta)\psi,\\
&\partial_t\psi +g\eta+  \frac{1}{2} \psi_x^2  
-\frac{1}{2} \frac{\bigl(G(\eta)\psi + \eta_x \psi_x \bigr)^2}{1+ \eta_x^2}
= \kappa H(\eta),
\end{aligned}
\right.
\end{equation}
where $g$ and $\kappa$ are positive constants
and $H(\eta)$ is the mean curvature of the free surface:
\begin{equation*}
H(\eta) := \partial_x \bigg( \frac{\partial_x\eta}{\sqrt{1+(\partial_x\eta)^2}} \bigg)
= \frac{\eta_{xx}}{(1 + \eta_x^2)^{3/2}} \cdot
\end{equation*}
The gravity constant $g$ can be normalized 
by jointly rescaling the time $t$ and the amplitude of $\psi$. 
With no loss of generality, 
in this paper we assume that 
$g = 1$.

\bigskip 
\noindent \textbf{Periodic solutions.} 
We seek solutions $u(t,x) = (\eta(t,x), \psi(t,x))$ of system \eqref{WW} which are periodic, 
with period $2\p$ in space and period $T = 2\p / \om$ in time, where the parameter $\om > 0$ is an unknown of the problem.
Rescaling the time $t \to \om t$, the problem becomes 
\[
F(u,\om) = 0,
\]
where $F = (F_1, F_2)$, $u = (\eta,\psi)$ is $2\p$-periodic both in time and space, and
\begin{align} \label{F1}
F_1(\eta,\psi) & := \om \partial_t \eta - G(\eta)\psi 
\\
F_2(\eta,\psi) & := \om \partial_t\psi + \eta + \dfrac{1}{2} \psi_x^2  
- \dfrac{1}{2} \, \dfrac{1}{1 + \eta_x^2} \, 
\bigl( G(\eta)\psi + \eta_x  \psi_x \bigr)^2 - \kappa \frac{\eta_{xx}}{(1 + \eta_x^2)^{3/2}}.
\label{F2}
\end{align}

\bigskip
\noindent \textbf{Functional setting.}
We use Sobolev spaces of functions with the same regularity both in time and in space:
consider the exponential basis 
$\{e^{i(lt+jx)}: (l,j) \in \Z^2\}$ on $\T^2$, 
and the standard Sobolev space $H^s := H^s(\T^2, \R)$ on $\T^2$ given by
\begin{equation} \label{unified norm}
H^s
= \Big\{ f = \sum_{(l,j) \in \Z^2} \hat f_{l,j} \, e^{i(lt+jx)} : \
\| f \|_s^2 := \sum_{(l,j) \in \Z^2} |\hat f_{l,j}|^2 \langle l,j \rangle^{2s} < \infty \Big\},
\end{equation}
where $\langle l,j \rangle := \max \{ 1, |l|+|j| \}$. 
Also, we set in the natural way $H^s(\T^2, \R^2) $ $:= \{ u = (\eta, \psi) : \eta, \psi \in H^s \}$ 
with norm $\| u \|_s^2 := \| \eta \|_s^2 + \| \psi \|_s^2$. 

\begin{remark}
Regularity in time and in space could be handled separately, 
as it is natural when thinking of the Cauchy problem (see \cite{BG}).
However, we shall consider changes of variables of the form $(t,x) \mapsto (t,x+\b(t,x))$, 
which, in some sense, mix the regularity in time and the one in space. 
To work with regularity in the time-space pair is a convenient choice.
\end{remark}

\noindent \textbf{Symmetries.} 
Because of reversibility in time and symmetry in space, the problem has an invariant subspace, where we look for  solutions: 
$ X \times Y = \{ u = (\eta,\psi)  : \eta \in X, \ \psi \in Y \}$, 
\begin{equation}  \label{def XY}
X := \{ \eta(t,x) : \, \eta \ \text{even}(t), \ \text{even}(x) \}, \quad 
Y := \{ \psi(t,x) : \, \psi \ \text{odd}(t), \ \text{even}(x) \}.
\end{equation}
The restriction to this subspace is used in Section \ref{sec:kernel} to deal with the space and time averages, and in Section \ref{sec:semi-FIO conj} (see \eqref{beta 1}).

\subsection{Main result}  \label{subsec:main result}

We assume three hypotheses on the surface tension coefficient $\kappa > 0$, 
which are discussed in the comments below Theorem \ref{thm:main}.  
First, $\kappa \notin \Q$. 
Second, $\kappa \neq \rho_0$, where $\rho_0$ is the unique real root of the polynomial $p(x) := 136 x^3 + 66 x^2 +3x -8$ 
(by the rational root test, $\rho_0$ is an irrational number, and it is in the interval $0.265 < \rho_0 < 0.266$).
Third, we assume that $\kappa$ satisfies the Diophantine condition
\begin{equation} \label{kappa dioph}
| \sqrt{1 + \kappa}\, l + \sqrt{j + \kappa j^3} | > \frac{ \g_*}{j^{\t_*}} \quad 
\forall l \in \Z, \ j \geq 2,
\end{equation}
for some $\g_* \in (0, \frac12)$, where $\t_* > 1$. 
The next lemma says that \eqref{kappa dioph} is a very mild restriction on $\kappa$. 

\begin{lemma} \label{lemma:kappa dioph}
Let $\kappa_0 > 0$, $\t_* > 1$.
The set 
\begin{equation} \label{def Kappa}
\mK = \{ \kappa \in [0,\kappa_0] : \exists \ \g_* \in (0,1/2) \ \text{such that} \ \kappa  \text{ satisfies } \eqref{kappa dioph} \} 
\end{equation}
has full Lebesgue measure $|\mK| = \kappa_0$.
\end{lemma}

The proof of Lemma \ref{lemma:kappa dioph} is at the end of section \ref{sec:measure}.
Note that almost all positive real numbers $\kappa$ satisfy all these three hypotheses.
The main result of the paper is in the following theorem.

\begin{theorem} \label{thm:main}
Assume that $\kappa > 0$ is an irrational number, $\kappa \neq \rho_0$, and $\kappa$ satisfies \eqref{kappa dioph} for some $\g_* \in (0,1/2)$, with $\t_* = 3/2$.
Then there exist constants $C>0$, $s_0 > 12$, $\e_0 \in (0,1)$ 
such that for every $\e \in (0, \e_0]$ there exists 
a set $\mG_\e \subset [1,2]$ of parameters with the following properties. 

For every $\xi \in \mG_\e$ there exists a solution $u = (\eta, \psi)$ of $F(u,\om) = 0$ with time-frequency 
\begin{equation}  \label{freq-ampl in the thm}
\om = \bar \om + \bar\om_2 \e^2 \xi + \bar\om_3 \e^3 \xi^{3/2},
\end{equation}
where $\bar \om := \sqrt{1 + \kappa}$ and the coefficients $\bar\om_2, \bar\om_3$ depend only on $\kappa$, with $\bar\om_2 \neq 0$. 
The solution has Sobolev regularity $u \in H^{s_0}(\T^2, \R^2)$, 
it has parity $u \in X \times Y$, and small amplitude $u = O(\e)$. 
More precisely, $u$ has $\e$-expansion
\[
\eta = \e \sqrt{\xi} \cos(t) \cos(x) + O(\e^2),
\quad 
\psi = - \e \sqrt{\xi} \sqrt{1 + \kappa}\,  \sin(t) \cos(x) + O(\e^2),
\]
where $O(\e^2)$ denotes a function $f$ such that $\| f \|_{s_0} \leq C \e^2$.

The set $\mG_\e \subset [1,2]$ has positive Lebesgue measure 
$|\mG_\e| \geq 1 - C \e^{1/18}$, asymptotically full 
$| \mG_\e | \to 1$ as $\e \to 0$.
\end{theorem}

Some comment about the role of the three hypotheses on the surface tension coefficient $\kappa$: 

\begin{enumerate}
\item 
The first assumption $\kappa \notin \Q$ implies that the linearized problem at the equilibrium $u=0$
has a nontrivial one-dimensional kernel (see section \ref{sec:kernel}), 
from which the solution of the nonlinear water waves problem bifurcates.
Rational values of $\kappa$ would lead to a different bifurcation analysis.

\item  
The second assumption $\kappa \neq \rho_0$ implies that the coefficient $\bar\om_2$ is 
nonzero (a ``twist'' condition).
As a consequence, the map $\xi \mapsto \om$ in \eqref{freq-ampl in the thm} is a bijection 
for $\e$ sufficiently small. 
Thus Theorem \ref{thm:main} gives the existence of a  solution of the water waves system with time-frequency $\om$ for many values $\om$ in an $\e^2$-neighbourhood of the ``unperturbed'' frequency $\bar\om := \sqrt{1 + \kappa}$. 
Instead, for $\kappa = \rho_0$ one should push forward with the analysis of the frequency-amplitude relation, looking for a higher order twist condition. 

\item 
The third assumption \eqref{kappa dioph} on $\kappa$ gives a Diophantine control on the small divisors of the \emph{unperturbed} problem, and it is used in the measure estimates in section \ref{sec:measure}, see in particular Remark \ref{rem:why imposing kappa dioph}. 
\end{enumerate}

\begin{remark} 
One could rename $\tilde \e := \e \sqrt{\xi}$ and work with one parameter $\tilde \e$
instead of two ($\e, \xi$). 
However, it is convenient to work with the two parameters $\e$ and $\xi$ to split two different roles: $\e \ll 1$ merely gives the smallness, while $\xi \in [1,2]$ allows to control the small divisors by imposing the Melnikov non-resonance conditions.
\end{remark}

\section{Preliminaries} \label{sec:prelim}

\noindent \textbf{Notations.} 
The notation $a \leq_s b$ indicates that $a \leq C(s) b$ for some constant $C(s) > 0$ depending on $s$ and possibly on the data of the problem, namely $\kappa, \g_*, \t_*$ 
($\kappa$ is the surface tension coefficient and $\g_*, \t_*$ are in \eqref{kappa dioph}).

Given $\e > 0$, for functions $u \in H^s(\T^2)$ depending on a parameter $\xi \in \mG \subset [1,2]$, 
we define 
\begin{equation}  \label{def norm Lipe}
\begin{aligned}
\| u \|_s^\Lipe 
&:= \| u \|_s^{\sup} + \e \| u \|_s^\lip \quad\text{with}\\
\| u \|_s^{\sup} &:= \sup_{\xi \in \mG} \| u(\xi) \|_s, 
\qquad
 \| u \|_s^\lip := \sup_{ \begin{subarray}{c} \xi_1, \xi_2 \in \mG \\ \xi_1 \neq \xi_2 \end{subarray}} 
\frac{ \| u(\xi_1) - u(\xi_2) \|_s }{ |\xi_1 - \xi_2| } \,.
\end{aligned}
\end{equation}

\bigskip
\noindent \textbf{Properties of the Dirichlet-Neumann operator.}
We collect some fundamental properties of the Dirichlet-Neumann operator $G$ that are used in the paper, 
referring to \cite{LannesJAMS,IP-Mem-2009,AM,AlDe} for more details. 

The mapping $(\eta,\psi)\rightarrow G(\eta)\psi$ is linear 
with respect to $\psi$ and nonlinear with respect to $\eta$. 
The derivative with respect to $\eta$ is called the ``shape derivative'', and it is given by Lannes' formula (see \cite{LannesJAMS,LannesLivre})
\begin{equation} \label{formula shape der}
G'(\eta) [\tilde{\eta} ] \psi 
= \lim_{\eps\rightarrow 0} \frac{1}{\e} \{ G (\eta+\eps\tilde\eta ) \psi - G(\eta) \psi \}
= - G(\eta) (B \tilde\eta ) -\partial_x (V \tilde\eta )
\end{equation}
where we introduced the notations
\begin{equation} \label{def B V}
\B := \B(\eta,\psi) := \frac{\eta_x \psi_x + G(\eta)\psi }{ 1 + \eta_x^2 }\,, 
\qquad 
V := V(\eta,\psi) := \psi_x - B \eta_x.
\end{equation}
Craig, Schanz and Sulem (see \cite{CSS} and~\cite[Chapter~$11$]{SuSu}) have shown that one can expand the Dirichlet-Neumann operator as a sum of pseudo-differential operators % and they gave 
with precise estimates for the remainders. 
Using \eqref{formula shape der} repeatedly, we get the second order Taylor expansion
\begin{align} \label{Taylor G}
G(\eta)\psi 
& = G(0) \psi + G'(0)[\eta] \psi + \frac12 G''(0)[\eta, \eta] \psi + G_{\geq 4}(\eta)\psi
\notag \\
& = 
|D_x| \psi - |D_x|(\eta |D_x| \psi) - \pa_x (\eta \pa_x \psi) 
+ \frac12 \pa_{xx} (\eta^2 |D_x| \psi) \notag \\
&\quad+ |D_x| ( \eta |D_x| (\eta |D_x| \psi)) 
+ \frac12 |D_x| (\eta^2 \psi_{xx})
+ G_{\geq 4}(\eta)\psi,
\end{align}
where $G(0) = |D_x|$ and $G_{\geq 4}(\eta)\psi$ is of order $4$, such that $G_{\geq 4}(\eta)\psi=O(\eta^3 \psi)$. 
Moreover, it follows from \cite[Section 2.6]{AlDe} 
that it 
satisfies the following estimate: for $s_0\ge 10$ and any 
$s\ge s_0$, if $\| \eta\|_{s_0}$ is small enough, then 
\begin{equation} \label{n35}
\lA G_{\geq 4}(\eta)\psi\rA_{H^{s}(\T)}
\leq_{s} 
\| \eta\|_{H^{s_0}(\T)}^{2} \Bigl\{ \| \psi\|_{H^{s_0}(\T)}
\| \eta \|_{H^{s+4}(\T)}
+\| \eta \|_{H^{s_0}(\T)}
\| \psi\|_{H^{s+5}(\T)}\Bigr\}.
\end{equation}

\bigbreak

A key property is that one can use microlocal analysis to study $G$.
In the present case the matter is easier than in a more general case, 
because the physical problem has space-dimension 2 
(see Definition \ref{defi:DNO}, where the space variables are $(x,y) \in \R^2$), 
it is periodic in the horizontal direction $x \in \T$, with infinite depth, so that 
\begin{equation}  \label{G(eta) = D + rest}
G(\eta) = |D_x| + \mR_{G}(\eta), 
\end{equation}
where the remainder $\mR_{G} = \mR_{G}(\eta)$ is bounded in $t$ and regularizing in $x$ at expense of $\eta$. 
More precisely, there exists a positive constant $\delta$ such that, 
for $\eta(x), h(x)$ functions of $x$ only, independent of time, 
if $\|\eta\|_{H^5(\T)} \le \delta$, then for any $s \geq 1$,
\[ 
\| \mR_G(\eta) \psi \|_{H^s(\T)} 
\leq_s \| \eta \|_{H^{s+4}(\T)} \| \psi \|_{H^{1/2}(\T)}.
\] 
This estimate is proved in \cite{AlDe} 
(see Proposition 2.7.1 in Chapter 2, noticing that the smallness condition on $\lA \eta \rA_{C^\gamma}$ assumed in this proposition is satisfied provided that $\|\eta\|_{H^5(\T)}$ is small enough).
Moreover, if $\| \eta \|_{H^5(\T)} \leq \d$, 
then for all $s \geq 4$ and all $5\le \mu\le s-1$
\begin{equation}\label{n37}
\| G(\eta) \psi \|_{H^\mu(\T)} \leq_s \| \psi \|_{H^{\mu+1}(\T)} 
+ \| \eta \|_{H^{s+1}(\T)} \| \psi \|_{H^5(\T)}.
\end{equation}
Similar estimates can be also proved for functions of $(t,x) \in \T^2$, 
where $t$ plays the role of a parameter, using repeatedly the time-derivative formula 
\[
\pa_t \{ G(\eta)\psi \}
= G(\eta) \pa_t \psi + G'(\eta)[ \pa_t \eta] \psi 
= G(\eta) \psi_t - G(\eta) (B \eta_t ) - \partial_x (V \eta_t)
\]
(see the argument of section \ref{sec:with time dep}, where it is explained in details how to extend estimates for functions of $x \in \T$ to include the dependence on time $t \in \T$).
Thus in $H^s(\T^2)$ equipped with the norm 
$\|\cdot\|_s$ defined by \eqref{unified norm}, we have the following (non-sharp) tame bounds:
if $\| \eta \|_6 \leq \d$, then for all $s \geq 2$, $m \geq 0$, 
\begin{equation}  \label{est mR G}
\| \mR_G(\eta) |D_x|^m \psi \|_s 
\leq_s \| \psi \|_s \| \eta \|_{7+m} 
+ \| \psi \|_2 \| \eta \|_{s+5+m};
\end{equation}
if $\| \eta \|_6 \leq \d$, then, for all $s \geq 6$, 
\begin{equation}  \label{G estimate}
\| G(\eta) \psi \|_s 
\leq_s \| \psi \|_{s+1} 
+ \| \eta \|_{s+1} \| \psi \|_6.
\end{equation}
Finally, regarding parities, we note that, if $\eta \in X$, then $G(\eta)$ preserves the parities,
namely 
$G(\eta)\psi \in X$ for $\psi \in X$, and 
$G(\eta)\psi \in Y$ for $\psi \in Y$.

\section{Bifurcation analysis}  \label{sec:bif}

Let us consider the linearized equations around the equilibrium $(\eta, \psi) = (0,0)$. 
Directly from \eqref{F1}-\eqref{F2}, one finds that the linearized operator is
\begin{equation}  \label{Lom}
L_\om := F' (0, 0) 
= \begin{pmatrix} 
\om \pa_t & - G(0) \\
1 - \kappa \pa_{xx} \  & \  \om \pa_t \\
\end{pmatrix}. 
\end{equation}

\subsection{Kernel} \label{sec:kernel}
 
We study the kernel of $L_\om$. Let $\eta \in X$, $\psi \in Y$, namely
\begin{equation}  \label{eta psi sin cos series}
\eta(t,x) = \sum_{l \geq 0, \, j \geq 0} \eta_{lj} \cos(lt) \cos(jx), \quad 
\psi(t,x) = \sum_{l \geq 1, \, j \geq 0} \psi_{lj} \sin(lt) \cos(jx),
\end{equation}
where $\eta_{lj}, \psi_{lj} \in \R$ and, for convenience, we also define $\psi_{lj} := 0$ for $l = 0$, as it does not change anything in the sum.
Recall that $G(0) = |D_x|$, 
\[
|D_x| \cos(jx) = |j| \cos(jx), \quad 
|D_x| \sin(jx) = |j| \sin(jx) \quad 
\forall j \in \Z,
\]
and $|D_x| = \pa_x \mH$, where $\mH$ is the Hilbert transform, with
\[
\mH \cos(jx) = \sign(j) \sin(jx), \quad 
\mH \sin(jx) = - \sign(j) \cos(jx) \quad 
\forall j \in \Z,
\]
namely 
$|D_x| e^{ijx} = |j| e^{ijx}$, $\mH e^{ijx} = - i \sign(j) e^{ijx}$ for all $j \in \Z$. 
Hence
\begin{equation} \label{formula Lom}
L_\om [\eta, \psi] = \sum_{l,j \geq 0} 
	\begin{pmatrix} 
	(- \om l \eta_{lj} - j \psi_{lj} ) \sin(lt) \cos(jx) \\
	[(1 + \kappa j^2) \eta_{lj} + \om l \psi_{lj} ] \cos(lt) \cos(jx)
	\end{pmatrix}.
\end{equation}
Assume that $L_\om[\eta, \psi] = 0$, so that  
$\om l \eta_{lj} + j \psi_{lj} = 0 = (1 + \kappa j^2) \eta_{lj} + \om l \psi_{lj}$ 
for all $l,j \geq 0$.
Since $1 + \kappa j^2 \geq 1 > 0$, we get
\[
\eta_{lj} = - \frac{\om l}{1 + \kappa j^2} \, \psi_{lj}, 
\quad 
\{ \om^2 l^2 - j (1 + \kappa j^2) \} \psi_{lj} = 0.
\]
At $l = 0$, this implies $\eta_{0j} = \psi_{0j} = 0$ for all $j \geq 0$ 
(recall that $\psi_{0j} = 0$ by definition, see above). 
For $l \geq 1$, choosing $\om \neq 0$, we deduce that at $j = 0$ one has
$\eta_{l0} = \psi_{l0} = 0$ for all $l \geq 1$.

Hence $\eta_{lj}, \psi_{lj}$ can be nonzero only for $l, j \geq 1$. 
If $(\eta_{lj}, \psi_{lj}) \neq (0,0)$, then $\psi_{lj} \neq 0$, and therefore 
(we assume $\om > 0$)
\[
\om = \frac{\sqrt{ j (1 + \kappa j^2)} }{l} \,.
\]
Suppose that there are two pairs $(l_1, j_1)$, $(l_2, j_2)$ that give the same $\om$, namely 
\[
\frac{\sqrt{ j_1 (1 + \kappa j_1^2)} }{l_1} \,
= \frac{\sqrt{ j_2 (1 + \kappa j_2^2)} }{l_2} \,.
\]
Taking the square,
$
\kappa (j_1^3 l_2^2 - j_2^3 l_1^2) + (j_1 l_2^2 - j_2 l_1^2) = 0.
$
Now, if $\kappa \notin \Q$, then both the integers $(j_1^3 l_2^2 - j_2^3 l_1^2)$ and $(j_1 l_2^2 - j_2 l_1^2)$ are zero, whence $(l_2, j_2) = (l_1, j_1)$. Thus irrational values of $\kappa$ give a kernel of dimension one. 
We consider the simplest nontrivial case $(l,j) = (1,1)$, and fix 
\[
\bar \om := \sqrt{1 + \kappa}, \qquad \kappa > 0, \  \kappa \notin \Q. 
\]
The factor 
\begin{equation} \label{factor nonzero}
\{ \bar\om^2 l^2 - j (1 + \kappa j^2) \} \neq 0 \quad \forall l,j \geq 0, \ 
(l,j) \notin \{ (1,1), (0,0) \}. 	
\end{equation}
The case $(l,j) = (0,0)$, as already seen, does not give any contribution to the kernel. Thus the kernel of $L_{\bar\om}$ is 
\begin{equation}  \label{def kernel}
V := \mathrm{Ker}(L_{\bar\om}) 
= \{ \lm v_0 : \lm \in \R \}, \quad 
v_0 := \begin{pmatrix} \cos(t) \cos(x) 
\\ - \bar\om \sin(t) \cos(x) \end{pmatrix}.
\end{equation} 
There is some freedom in fixing another vector $w_0$ to span the subspace $(l,j) = (1,1)$. 
It is convenient to define
\begin{equation}  \label{def W}
\ba
W &:= \{ (\eta,\psi) \in X \times Y : \eqref{eta psi sin cos series} \text{ holds, and } 
(\eta_{11}, \psi_{11}) = \psi_{11} (\bar\om, 1) \} \\
&= W^{(1,1)} \oplus W^{(\neq)},
\ea
\end{equation}
where 
\begin{equation}  \label{def W 11}
W^{(1,1)} := \{ \lm w_0 : \lm \in \R\}, \quad 
w_0 := \begin{pmatrix} \bar\om \cos(t) \cos(x) 
\\ \sin(t) \cos(x) \end{pmatrix}
\end{equation}
and 
\begin{equation}  \label{def W neq}
W^{(\neq)} := \{ (\eta,\psi) \in X \times Y : \eqref{eta psi sin cos series} \text{ holds, and } 
\eta_{11} = \psi_{11} = 0 \}.
\end{equation}
Thus $X \times Y = V \oplus W^{(1,1)} \oplus W^{(\neq)}$, namely every $u \in X \times Y$ can be written in a unique way as $u = a v_0 + b w_0 + w$, where $a,b \in \R$ and $w \in W^{(\neq)}$.

\subsection{Range} \label{sec:Range}
 
Like $F$, also $L_{\bar\om}$ maps $X \times Y \to Y \times X$. 
Let $(f,g) \in Y \times X$, namely 
\begin{equation}  \label{f g sin cos series}
f(t,x) = \sum_{l, j \geq 0} f_{lj} \sin(lt) \cos(jx), \quad 
g(t,x) = \sum_{l, j \geq 0} g_{lj} \cos(lt) \cos(jx),
\end{equation}
with $f_{lj}, g_{lj} \in \R$, $f_{0j} = 0$ for $l=0, j \geq 0$.
By \eqref{formula Lom} (with $\om = \bar\om$), 
the equation $L_{\bar\om}[\eta, \psi] = (f,g)$ is equivalent to 
\begin{equation}  \label{syst L bar om}
- \bar \om l \eta_{lj} - j \psi_{lj} = f_{lj}, \qquad 
(1+\kappa j^2) \eta_{lj} + \bar\om l \psi_{lj} = g_{lj}.
\end{equation}
By \eqref{factor nonzero}, if $(l,j) \notin \{ (1,1), (0,0) \}$, 
then system \eqref{syst L bar om} is invertible, with solution
\begin{equation}  \label{inv Lbom}
\eta_{lj} = \frac{- \bar\om l f_{lj} - j g_{lj}  }{\bar\om^2 l^2 - j(1+\kappa j^2)} \,, 
\qquad 
\psi_{lj} = \frac{(1 + \kappa j^2) f_{lj} + \bar\om l g_{lj} }{\bar\om^2 l^2 - j(1+\kappa j^2)} \,.
\end{equation}
For $(l,j) = (0,0)$, system \eqref{syst L bar om} is also invertible, with solution 
$\eta_{00} = g_{00}$, $\psi_{00} = 0$ (remember that $\psi_{00} = 0 = f_{00}$ by assumption).
For $(l,j) = (1,1)$, system \eqref{syst L bar om} has rank one, and it has solutions if and only if
$g_{11} + \bar\om f_{11} = 0$, in which case the solutions are 
$(\eta_{11}, \psi_{11}) = (0,-f_{11}) + \lm(1, - \bar\om)$, $\lm \in \R$
(clearly $\lm (1, - \bar\om)$ corresponds to $\lm v_0$, namely an element of the kernel $V$). 
Thus 
\begin{align*} 
R &:= \mathrm{Range}(L_{\bar\om}) = \{ (f,g) \in Y \times X : \eqref{f g sin cos series} 
\text{ holds, and } g_{11} + \bar\om f_{11} = 0 \} \\
&= R^{(1,1)} \oplus R^{(\neq)},
\end{align*} 
where
\begin{equation}  \label{def R 11}
R^{(1,1)} := \{ \lm r_0 : \lm \in \R\}, \quad 
r_0 := \begin{pmatrix} - \sin(t) \cos(x) 
\\ \bar\om \cos(t) \cos(x) \end{pmatrix}
\end{equation} 
and 
\begin{equation}  \label{def R neq}
R^{(\neq)} := \{ (f,g) \in Y \times X : \eqref{f g sin cos series} \text{ holds, and } 
f_{11} = g_{11} = 0 \}.
\end{equation} 
There is some freedom in fixing another vector $z_0$ to span the subspace $(l,j) = (1,1)$. 
It is convenient to define
\begin{equation}  \label{def Z}
Z := \{ \lm z_0 : \lm \in \R \} \subset Y \times X, \qquad 
z_0 := \begin{pmatrix} \sin(t) \cos(x) \\ 
\bar\om \cos(t) \cos(x) \end{pmatrix}.
\end{equation}
Note that $L_{\bar\om}$ is an invertible map of $W^{(\neq)} \to R^{(\neq)}$, and 
(using the equality $\bar\om^2 = 1 + \kappa$) 
\begin{equation} \label{L bar om 11}
L_{\bar\om}[w_0] = (\bar\om^2 + 1) r_0 = (2 + \kappa) r_0, 
\qquad 
\pa_t v_0 = - z_0.
\end{equation} 
Thus $Y \times X = Z \oplus R^{(1,1)} \oplus R^{(\neq)}$, namely every $u \in Y \times X$ can be written in a unique way as $u = a z_0 + b r_0 + r$, where $a,b \in \R$ and $r \in R^{(\neq)}$.
The formula for the projection on $r_0, z_0$ is
\begin{equation}  \label{proj 11 rule}
\begin{pmatrix} 
p \sin(t) \cos(x) \\ 
q \cos(t) \cos(x) 
\end{pmatrix}
= \lm_r r_0 + \lm_z z_0, 
\quad 
\lm_r = - \frac{p}{2} + \frac{q}{2 \bar\om} \,, 
\quad
\lm_z = \frac{p}{2} + \frac{q}{2 \bar\om} \,, ~
p,q \in \R.
\end{equation}

\subsection{Construction of an approximate solution}
\label{sec:approx sol}
We look for solutions of \eqref{WW} with frequency $\om$ close to the ``unperturbed'' frequency $\bar\om = \sqrt{1 + \kappa}$. Write 
\[
\om = \bar\om + \e \om_1 + \e^2 \om_2 + \ldots, 
\quad 
u = (\eta,\psi) = \e u_1 + \e^2 u_2 + \ldots, 
\]
\[
F(u) = (\om - \bar\om)\pa_t u + L_{\bar\om} u 
+ \mN_2(u) + \mN_3(u) + \ldots, 
\quad \mN_k(u) = T_k[u,\ldots,u],
\]
where $T_k$ is a symmetric $k$-linear map, so that $\mN_2(u)$ denotes the quadratic part of $F$, 
$\mN_3(u)$ the cubic one, etc.
We get
$$
F(u) = \e \mF_1 + \e^2 \mF_2 + \e^3 \mF_3 + \e^4 \mF_4 + O(\e^5),
$$
where $\mF_1 = L_{\bar\om} u_1$, 
\begin{align*}
\mF_2 & = L_{\bar\om} u_2 + \om_1 \pa_t u_1 + T_2[u_1, u_1],\\
\mF_3 &= L_{\bar\om} u_3 + \om_2 \pa_t u_1 + \om_1 \pa_t u_2 + 2 T_2[u_1, u_2] + T_3[u_1, u_1, u_1], \\ 
\mF_4 & = L_{\bar\om} u_4 + \om_3 \pa_t u_1 + \om_2 \pa_t u_2 + \om_1 \pa_t u_3 
+ 2 T_2[u_1, u_3] + T_2[u_2, u_2] \\
&\quad+ 3 T_3[u_1, u_1, u_2] 
+ T_4[u_1, u_1, u_1, u_1].
\end{align*}
We prove that there exist $u_1, u_2, u_3, u_4, \om_1, \om_2, \om_3$ such that $F(u) = O(\e^5)$.

\emph{Order $\e$. } $\mF_1 = 0$ 
if and only if $u_1 \in V$, namely $u_1 = a_1 v_0$, for some $a_1 \in \R$, where $v_0$ is defined in \eqref{def kernel}. We assume that $a_1 \neq 0$ (otherwise the construction of $u$ becomes trivial).

\emph{Order $\e^2$. } 
The quadratic part of $F$ (see \eqref{F1}-\eqref{F2}) is 
\[
T_2[u,u] 
= \begin{pmatrix}
\pa_x (\eta \psi_x) + G_0 [ \eta G_0 \psi]
\\
\frac{1}{2} [ \psi_x^2  - (G_0 \psi)^2 ]
\end{pmatrix},
\]
where, for brevity, we write $G_0 := G(0) = |D_x|$. 
More generally, if $u' = (\eta', \psi')$ and $u'' = (\eta'', \psi'')$, then 
\begin{equation}  \label{formula T2 general}
T_2[u', u''] 
= \frac12 \begin{pmatrix}
\pa_x ( \eta' \psi''_x + \eta'' \psi'_x ) 
+ G_0 [ \eta' G_0 \psi'' + \eta'' G_0 \psi']
\\
\psi'_x \psi''_x  - (G_0 \psi') (G_0 \psi'') 
\end{pmatrix}.
\end{equation}
In general, for $n,j \geq 0$, 
\begin{multline} \label{DN nj}
\pa_x [ \cos(nx) \pa_x \cos(jx) ] + G_0 [ \cos(nx) G_0 \cos(jx)] \\
= \frac12 j ( |j-n| + n - j ) \, \cos((j-n)x),
\end{multline}
and $|j-n| + n - j = 2(n-j)$ for $j < n$, and it is zero for $j \geq n$. 
In particular, for $n = 1$
\begin{equation}  \label{easy life}
\pa_x [ \cos(x) \pa_x \cos(jx) ] + G_0 [ \cos(x) G_0 \cos(jx)] = 0 \quad \forall j \geq 0.
\end{equation}
For $u_1 = a_1 v_0$ (where $v_0$ is defined in \eqref{def kernel}), we calculate
\begin{equation} \label{T2 v0 v0}
T_2[u_1, u_1] 
= a_1^2 T_2[v_0, v_0], 
\quad 
T_2[v_0, v_0] = \frac{\bar\om^2}{4} 
\begin{pmatrix} 
0 \\
[\cos(2t) - 1] \cos(2x)
\end{pmatrix}.
\end{equation}
In particular, $T_2[u_1,u_1]$ has no component Fourier-supported on $(l,j) = (1,1)$. 
On the contrary, $\pa_t u_1 = a_1 \pa_t v_0 $ is Fourier-supported only on $(l,j) = (1,1)$. 
Split 
\[
u_2 = a_2 v_0 + b_2 w_0 + a_1^2 w_2, \quad 
a_2, b_2 \in \R, \quad w_2 \in W^{(\neq)},
\]
where $w_0$, $W^{(\neq)}$ are defined in \eqref{def W 11}, \eqref{def W neq}.
The equation $\Pi_{R^{(\neq)}} \mF_2 = 0$ (i.e. the projection on the Fourier modes $(l,j) \neq (1,1)$) is 
\begin{equation}  \label{eq:quadratic range}
a_1^2 ( L_{\bar\om} w_2 + T_2[v_0, v_0]) = 0.
\end{equation}
Since $L_{\bar\om} : W^{(\neq)} \to R^{(\neq)}$ is invertible and $a_1 \neq 0$, 
we solve $L_{\bar\om} w_2 + T_2[v_0, v_0] = 0$ and, by \eqref{inv Lbom} and \eqref{T2 v0 v0}, we calculate 
\begin{equation}  \label{u2 out of (1,1)}
w_2 = \begin{pmatrix} \a_{02} \cos(2x) + \a_{22} \cos(2t) \cos(2x) \\ 
\b_{22} \sin(2t) \cos(2x) \end{pmatrix}
\end{equation}
with
\[
\a_{02} := \frac{1 + \kappa}{4 (1 + 4\kappa)}\,, 
\quad 
\a_{22} := \frac{1 + \kappa}{4 (1 - 2\kappa)}\,, 
\quad
\b_{22} := - \bar\om \a_{22}
\]
(the denominators $1 + 4\kappa, 1-2\kappa$ are nonzero because $\kappa \notin \Q$).

It remains to solve the equation $\mF_2 = 0$ on $(l,j) = (1,1)$. 
By \eqref{L bar om 11}, the component of $L_{\bar\om} u_2$ that is Fourier-supported on $(1,1)$ is 
$L_{\bar\om} [a_2 v_0 + b_2 w_0] 
= b_2 L_{\bar\om} [w_0] 
= b_2 (2+\kappa) r_0$.
By \eqref{L bar om 11}, $\pa_t u_1$ $= a_1 \pa_t v_0$ $= - a_1 z_0$, 
while $T_2[u_1, u_1]$ gives no contribution on $(1,1)$ according to \eqref{T2 v0 v0}. 
Thus the equation projected on $(1,1)$ is 
\[
b_2 (2+\kappa) r_0 - \om_1 a_1 z_0 = 0.
\]
$r_0$ and $z_0$ are linearly independent. 
Since $a_1 \neq 0$, we have to choose $\om_1 = 0$ and $b_2 = 0$. 
There is no constraint on $a_2$. It is convenient to fix $a_2 = 0$. 
With this choice we have $u_2 = a_1^2 w_2$.

\medskip

\emph{Order $\e^3$. } Since $\om_1 = 0$, one has 
$\mF_3 = L_{\bar\om} u_3 + \om_2 \pa_t u_1 + 2 T_2[u_1, u_2] + T_3[u_1, u_1, u_1]$. 
Using the Taylor expansion \eqref{Taylor G} of $G(\eta)$ at $\eta = 0$,  
for a general $u = (\eta, \psi)$, the cubic part of $F$ is given by
\[
T_3[u,u,u] 
= \begin{pmatrix}
- \frac12 \pa_{xx} (\eta^2 G_0 \psi) 
- G_0( \eta G_0 (\eta G_0 \psi)) 
- \frac12 G_0(\eta^2 \psi_{xx})
\\
(G_0 \psi) \big( \eta \psi_{xx} + G_0 (\eta G_0 \psi) \big) + \frac12 \kappa \pa_x(\eta_x^3)
\end{pmatrix}.
\]
We calculate $T_3[u_1, u_1, u_1] = a_1^3 \, T_3[ v_0, v_0, v_0]$, where $v_0$ is in \eqref{def kernel}: 
\[
T_3[ v_0, v_0, v_0]
=  \frac{-1}{32} \begin{pmatrix}
2 \bar\om [\sin(t) + \sin(3t)] \cos(x)
\\
\{ (2 + 11 \kappa) \cos(t) + (\kappa - 2) \cos(3t) \}
[\cos(x) - \cos(3x)]
\end{pmatrix} 
\]
(as usual, we have used that $\bar \om^2 = 1 + \kappa$). 
We also calculate $2 T_2[u_1, u_2] = 2 a_1^3 \, T_2[v_0, w_2]$.
By \eqref{formula T2 general}, 
\[
T_2[v_0, w_2] 
= - \frac{\bar\om}{4}
\begin{pmatrix}
\{ (2 \a_{02} - \a_{22}) \sin(t) + \a_{22} \sin(3t) \} \cos(x)
\\
2 \bar\om \a_{22} [\cos(t) - \cos(3t)] \cos(3x)
\end{pmatrix}.
\]
Split $u_3 = a_3 v_0 + b_3 a_1^3 w_0 + a_1^3 w_3$, where $a_3, b_3 \in \R$, $w_3 \in W^{(\neq)}$,
and $w_0$, $W^{(\neq)}$ are defined in \eqref{def W 11}, \eqref{def W neq}.
The projection on $R^{(\neq)}$ of the equation $\mF_3 = 0$ is
\begin{equation} \label{eq mF3 neq}
a_1^3 ( L_{\bar\om} w_3 + \Pi_{R^{(\neq)}} \{ 2 T_2[v_0, w_2] + T_3[v_0, v_0, v_0] \} ) = 0
\end{equation}
because $u_1 = a_1 v_0$ and $u_2 = a_1^2 w_2$. 
Since $a_1 \neq 0$ and $L_{\bar\om} : W^{(\neq)} \to R^{(\neq)}$ is invertible, 
the equation \eqref{eq mF3 neq} 
determines $w_3$, which depends only on $\kappa$. 

Let us study the projection of the equation $\mF_3 = 0$ on $(l,j) = (1,1)$. 
As above, $L_{\bar\om} [a_3 v_0 + b_3 a_1^3 w_0] = b_3 a_1^3 (2 + \kappa) r_0$ and 
$\pa_t u_1 = a_1 \pa_t v_0 = - a_1 z_0$. 
Using \eqref{proj 11 rule}, we calculate the projection $\Pi_{(1,1)}$ on $R^{(1,1)} \oplus Z$ (namely on the Fourier mode $(l,j) = (1,1)$ in $Y \times X$):
\begin{align*}
\Pi_{ (1,1) } ( 2 T_2[u_1, u_2] + T_3[u_1, u_1, u_1] ) 
& = - \frac{a_1^3}{32} 
\begin{pmatrix}
2 \bar\om [ 16 \a_{02} - 8 \a_{22} + 1 ]  \sin(t) \cos(x)
\\ (2 + 11 \kappa) \cos(t) \cos(x)  
\end{pmatrix}
\\ & 
= 
- \frac{a_1^3}{32} \Big( - \bar\om [ 16 \a_{02} - 8 \a_{22} + 1 ] + \frac{2 + 11 \kappa}{2 \bar\om} \Big) r_0 \\
&\quad- \frac{a_1^3}{32} \Big( \bar\om [ 16 \a_{02} - 8 \a_{22} + 1 ] + \frac{2 + 11 \kappa}{2 \bar\om} \Big) z_0 .
\end{align*}
Since $v_0$ and $r_0$ are linearly independent, 
$\Pi_{(1,1)} \mF_3 = 0$ if and only if
\begin{align*}
b_3 a_1^3 (2 + \kappa) 
- \frac{a_1^3}{32} \Big( - \bar\om [ 16 \a_{02} - 8 \a_{22} + 1 ] + \frac{2 + 11 \kappa}{2 \bar\om} \Big) & = 0,
\\
- \om_2 a_1 
- \frac{a_1^3}{32} \Big( \bar\om [ 16 \a_{02} - 8 \a_{22} + 1 ] + \frac{2 + 11 \kappa}{2 \bar\om} \Big) & = 0.
\end{align*}
Since $a_1 \neq 0$, the second equation determines $\om_2$ as
\begin{equation} \label{def om 2}
\om_2 = a_1^2 \, \bar\om_2, \quad
\bar\om_2 := \bar\om_2(\kappa) := - \frac{\bar\om}{32} 
\Big( \frac{ 4(1+\kappa) }{1 + 4 \kappa} - \frac{ 2(1+\kappa) }{1 - 2 \kappa} + 1 + \frac{2 + 11 \kappa }{2 ( 1 + \kappa)} \Big) \neq 0,
\end{equation}
then the first equation determines $b_3$ depending only on $\kappa$. 
Note that $\bar\om_2$ is nonzero for $\kappa \neq \rho_0$, 
where $\rho_0$ is the unique real root of the polynomial $p(x) = 136 x^3 + 66 x^2 +3x -8$
(after writing the common denominator in \eqref{def om 2}, one has 
$ \bar\om_2 = 0$ if and only if $p(\kappa) = 0$). 
There is no constraint on $a_3$. It is convenient to fix $a_3 = 0$. 
With this choice we have $u_3 = a_1^3 (b_3 w_0 + w_3)$.

\medskip

\emph{Order $\e^4$.} 
In the previous steps we have found $\om_1 = 0$, $\om_2 = a_1^2 \bar\om_2$, and 
$u_1 = a_1 v_0$, $u_2 = a_1^2 w_2$, $u_3 = a_1^3 (b_3 w_0 + w_3)$.
Let $u_4 = a_4 v_0 + a_1^4 (b_4 w_0 + w_4)$, with $a_4, b_4 \in \R$ and $w_4 \in W^{(\neq)}$.
The equation $\mF_4 = 0$ becomes
\begin{align}
0 & = \om_3 a_1 \pa_t v_0   
+ a_1^4 \big\{ b_4 L_{\bar\om} [w_0] + L_{\bar\om}[w_4]
+ \bar\om_2 \pa_t w_2 + 2 b_3 T_2[v_0, w_0] + 2 T_2[v_0, w_3] 
\notag\\
& \quad 
+ T_2[ w_2, w_2] 
+ 3 T_3[v_0, v_0, w_2] 
+ T_4[v_0, v_0, v_0, v_0]\big\}.
\label{eq order 4 bis} 
\end{align}
Its projection on $R^{(\neq)}$, after eliminating the factor $a_1^4 \neq 0$, is
\begin{align*}
L_{\bar\om}[w_4] 
+ \Pi_{R^{(\neq)}} \big\{& \bar\om_2 \pa_t w_2
+ 2 T_2[v_0, b_3 w_0 + w_3]
+ T_2[ w_2, w_2]\\
&\quad+ 3 T_3[v_0, v_0, w_2]
+ T_4[v_0, v_0, v_0, v_0] \big\} = 0.
\end{align*}
Since $L_{\bar\om} : W^{(\neq)} \to R^{(\neq)}$ is invertible, this equation determines $w_4$, 
depending only on $\kappa$. 

By \eqref{L bar om 11}, the projection of \eqref{eq order 4 bis} on the Fourier mode $(l,j) = (1,1)$ is
\[
a_1^4 ( b_4 (2+\kappa) + \a ) r_0 + ( - \om_3 a_1 + \b a_1^4 ) z_0 = 0
\]
for some real coefficients $\a, \b $ depending only on $\kappa$.
We choose $\om_3 = \b a_1^3$ and $b_4 = - \a / (2 + \kappa)$, and the equation is satisfied. 
We also fix $a_4 = 0$, so that $u_4 = a_1^4 (b_4 w_0 + w_4)$, 
and rename $a_1^2 := \xi > 0$, $\b := \bar\om_3 = \bar\om_3(\kappa)$.

In conclusion, we have found the \emph{frequency-amplitude relation}
\begin{equation} \label{freq-ampl} 
\om 
= \bar\om + \e^2 \om_2 + \e^3 \om_3
= \bar\om + \e^2 \bar\om_2 \xi + \e^3 \bar\om_3 \xi^{3/2}
\end{equation}
where the coefficient $\bar\om_2$ is nonzero 
and both $\bar\om_2, \bar\om_3$ depend only on $\kappa$, 
and an ``approximate solution''
\begin{equation} \label{approx sol} 
\ba
\bar u_\e 
&= \bar u_\e(\xi)
= \e \bar u_1 + \e^2 \bar u_2 + \e^3 \bar u_3 + \e^4 \bar u_4\\
&= \e \sqrt \xi v_0 + \e^2 \xi \bar w_2 + \e^3 \xi^{3/2} \bar w_3 + \e^4 \xi^2 \bar w_4,
\ea
\end{equation}
where $v_0$ is defined in \eqref{def kernel}, 
$\bar w_2 := w_2 \in W^{(\neq)}$, 
$\bar w_3 := b_3 w_0 + w_3 \in W$, 
$\bar w_4 := b_4 w_0 + w_4 \in W$, 
such that $F(\bar u_\e) = O(\e^5)$. 
All $v_0, \bar w_2, \bar w_3, \bar w_4$ depend only on $\kappa$. 
Moreover $\bar u_\e$ is a trigonometric polynomial, Fourier-supported on $\cos(lt) \cos(jx)$, 
$\sin(lt) \cos(jx)$, with both $l,j \in [0,5]$.

\subsection{Restriction of the linear inversion problem} \label{sec:lin syst}

In view of the Nash-Moser scheme, we have to study the inversion problem for the linearized system: 
given $f$, find $h$ such that $F'(u)[h] = f$. 
In this section we split this (linear) inversion problem in a way that 
takes advantage of the (nonlinear) calculations we have already done in section \ref{sec:approx sol} to construct the approximate solution $\bar u_\e$.

We assume that $u = \bar u_\e + \tilde u$ and $\| \tilde u \|_{s_0 + \s} \leq C \e^{2+\d}$, 
where $s_0 \geq 2$, $\s \geq 4$, $\d > 0$.
The linearized operator is 
\begin{align*}
F'(u)[h] 
& = (\e^2 \om_2 + \e^3 \om_3) \pa_t h + L_{\bar\om} h + 2 T_2[u, h] + 3 T_3[u,u,h] 
+ \mN_{\geq 4}'(u)[h]
\end{align*}
where $\mN_{\geq 4}(u)$ denotes the component of $F$ of order at least quartic. 
In the direction $h = U_\e := \bar u_1 + 2 \e \bar u_2 + 3 \e^2 \bar u_3$ 
(which is $\pa_\e \bar u_\e$ truncated at order $\e^2$) one has
\begin{align} 
F'(u)[U_\e] 
& = F'(u)[\bar u_1 + 2 \e \bar u_2 + 3 \e^2 \bar u_3] \notag 
\\ 
& = 2 \e \{ L_{\bar\om} [ \bar u_2 ] + T_2[\bar u_1, \bar u_1 ] \} \notag \\
&\quad + \e^2 \{ \om_2 \pa_t \bar u_1 + 3 L_{\bar\om} [ \bar u_3 ] + 6 T_2[\bar u_1, \bar u_2 ] 
+ 3 T_3[\bar u_1, \bar u_1, \bar u_1 ] \} + \rho  \notag 
\\ 
& = - 2 \e^2 \om_2 \pa_t \bar u_1 + \rho,
\label{using bif 123}
\end{align}
where 
\begin{align*}
\rho & := \e^3 \Big\{ \om_3 \pa_t \bar u_1  
+ (\om_2 + \e \om_3) \pa_t \{ 2 \bar u_2 + 3 \e \bar u_3 \}
+ 6 T_2[\bar u_1, \bar u_3 ] \\
&\qquad\quad
+ 2 T_2[\bar u_2, 2 \bar u_2 + 3 \e \bar u_3 ] 
+ 2 T_2[ \bar u_3 + \e \bar u_4, U_\e] + 3 T_3[\bar u_1 , \bar u_1, 2 \bar u_2 + 3 \e \bar u_3 ] 
\\ & \qquad\quad \ 
+ 3 T_3[ 2 \bar u_1 + \e \bar u_2 + \e^2 \bar u_3 + \e^3 \bar u_4, \bar u_2 + \e \bar u_3 + \e^2 \bar u_4, U_\e] \Big\}
\\ & \quad 
+ 2 T_2[ \tilde u, U_\e]
+ 6 \e T_3[ \tilde u, \bar u_1 + \e \bar u_2 + \e^2 \bar u_3 + \e^3 \bar u_4, U_\e]
+ 3 T_3[ \tilde u, \tilde u, U_\e] \\
&\quad
+ \mN_{\geq 4}'(\bar u_\e + \tilde u)[U_\e].
\end{align*}
To get \eqref{using bif 123} we have used the equalities $\mF_i = 0$, $i = 1,2,3$, 
namely the equations at order $\e$, $\e^2$, $\e^3$ solved in Section \ref{sec:approx sol}. 
For $s \geq s_0\geq 10$, we claim that 
the function $\rho$ satisfies
\begin{equation} \label{est rho}
\| \rho \|_s \leq_s \e^3 + \| \tilde u \|_{s+4}.
\end{equation}
Indeed, directly from the above definition of $\rho$, 
this estimate is clear for all the terms inside the brackets 
(recalling that $\bar{u}_k$ are trigonometric polynomials). 
To estimate the terms 
$T_2[ \tilde u, U_\e]$, 
$T_3[ \tilde u, \bar u_1 + \e \bar u_2 + \e^2 \bar u_3 + \e^3 \bar u_4, U_\e]$ and $T_3[ \tilde u, \tilde u, U_\e]$, we use 
the usual nonlinear estimates in Sobolev spaces.  
The last term 
$\mN_{\geq 4}'(\bar u_\e + \tilde u)[U_\e]$ 
is estimated starting from 
the linearization formula recalled below 
in~\eqref{linearized vero}, using 
the estimates \eqref{n35}Ê
and \eqref{n37} for 
the Dirichlet-Neumann operator and its 
Taylor expansion, 
together with similar estimates for the 
Taylor expansion for 
the coefficients $B$ and $V$ which appear in \eqref{linearized vero}, see \cite[Section 2.6]{AlDe} for these estimates.

Remember that $- \pa_t v_0 = z_0$. Split the datum $f = b z_0 + \tilde f$, where $b \in \R$ and $\tilde f \in R$. 
We look for $h$ of the form 
\[
h = a U_\e + \tilde h = a(\bar u_1 + 2 \e \bar u_2 + 3 \e^2 \bar u_3) + \tilde h,
\] 
where $a \in \R$, $\tilde h = \tilde h(t,x) \in W$ are unknowns. 
For $h$ of this form, 
\[
F'(u)[h] = a F'(u) [U_\e] + F'(u)[\tilde h] 
= a (2 \e^2 \om_2 \sqrt{\xi} \, z_0 + \rho) + F'(u)[\tilde h].
\]
Projecting onto $Z$ and $R$, one has $F'(u)[h] = f$ if and only if 
\begin{equation} \label{complete system}
\begin{cases}
a (2 \e^2 \om_2 \sqrt{\xi} \, z_0 + \Pi_Z \rho) + \Pi_Z F'(u) [\tilde h] = b z_0 
\\
a \Pi_R \rho + \Pi_R F'(u)[\tilde h] = \tilde f.
\end{cases}
\end{equation}
Assume that the restricted operator
$\mL_{R}^W := \Pi_R F'(u)_{|W} : W \to R$ is invertible, with 
\begin{equation} \label{inv hyp WR}
\| (\mL_{R}^W)^{-1} g \|_s 
\leq_s \g^{-1} \big( \| g \|_{s + 2} + \g^{-1} \| \tilde u \|_{s + \s} \| g \|_{s_0 + 2} \big) 
\end{equation}
for $s \geq s_0$, where $\g := \e^p$, $p := 5/6$.
Then we solve for $\tilde h$ in the second equation in \eqref{complete system} and find 
\begin{equation}  \label{tilde h from eq}
\tilde h = (\mL_{R}^W)^{-1} (\tilde f - a \Pi_R \rho),
\end{equation}
with estimate
\begin{equation} \label{est tilde h with still a}
\| \tilde h \|_s 
\leq_s \g^{-1} \{
\| \tilde f \|_{s+2} + \g^{-1} \| \tilde u \|_{s+\s} \| \tilde f \|_{s_0+2} 
+ |a| (\e^3 + \| \tilde u \|_{s+\s} ) \}, 
\end{equation}
because $\e^{2+\d} \g^{-1} < 1$ and $\s \geq 4$. 
Substituting \eqref{tilde h from eq} in the first equation of \eqref{complete system} gives
\begin{equation}  \label{a from eq}
a \{ 2 \e^2 \om_2 \sqrt{\xi} \, z_0 + \Pi_Z \rho - \Pi_Z F'(u) (\mL_{R}^W)^{-1} \Pi_R \rho  \} 
= b z_0 - \Pi_Z F'(u) (\mL_{R}^W)^{-1} \tilde f.
\end{equation}
Since $\Pi_Z L_{\bar\om} = 0$, the operator $\Pi_Z F'(u)$ starts quadratically, 
and it satisfies 
\[
\ba
&\| \Pi_Z F'(u) g \|_0 
\leq C \| u \|_{s_0 + 2} \| g \|_2
\leq C \e \| g \|_{s_0},\\ 
&\| \Pi_Z F'(u) (\mL_{R}^W)^{-1} g \|_0 
\leq_{s_0} \e \g^{-1} \| g \|_{s_0 + 2} 
\ea
\]
for all $g$. As a consequence, 
\[
\| \Pi_Z F'(u) (\mL_{R}^W)^{-1} \Pi_R \rho \|_0
\leq_{s_0} \e \g^{-1} \| \rho \|_{s_0 + 2}
\leq_{s_0} \e^{2+\d}.
\]
Therefore the coefficient of $a z_0$ in \eqref{a from eq} is $2 \e^2 \om_2 + o(\e^2)$, which is nonzero for $\e$ sufficiently small because $\om_2 \neq 0$. 
Hence from \eqref{a from eq} we find $a$ as a function of $b, \tilde f$, with estimate 
\begin{equation} \label{est a}
|a| \leq_{s_0} \e^{-2} ( |b| + \e \g^{-1} \| \tilde f \|_{s_0 + 2}) 
\leq_{s_0} \e^{-2} ( |b| + \| \tilde f \|_{s_0 + 2}) 
\leq_{s_0} \e^{-2} \| f \|_{s_0 + 2}
\end{equation}
(we have used that $\e \g^{-1} < 1$). 
Then, substituting the value of $a$ in \eqref{tilde h from eq}, we find a formula for $\tilde h$ as a function of $b, \tilde f$. 
We have solved the inversion problem $F'(u)[h] = f$. 
Since $\| h \|_s = \| \tilde h + a U_\e \|_s$ $\leq \| \tilde h \|_s + C(s) |a|$, 
by \eqref{est tilde h with still a},\eqref{est a} we get
\begin{align} 
\| h \|_s = \| F'(u)^{-1} f \|_s 
& \leq_s \g^{-1} \| f \|_{s + 2} + \e^{-2} \{ 1 + \g^{-1} \| \tilde u \|_{s+\s} \} 
\| f \|_{s_0 + 2} 
\notag \\ 
& \leq_s \e^{-2} \big( \| f \|_{s + 2} + \g^{-1} \| \tilde u \|_{s+\s} \| f \|_{s_0 + 2} \big).
\label{total inverse simpler}
\end{align}
We have proved the following inversion result:

\begin{lemma} \label{lemma:inv bif} 
Let $u = \bar u_\e + \tilde u$, with $\| \tilde u \|_{s_0 + \s} \leq C \e^{2+\d} $
and $s_0 \geq 2$, $\s \geq 4$, $\d > 0$.
Assume that the restricted operator $\mL_{R}^W := \Pi_R F'(u)_{|W} : W \to R$ is invertible, and its inverse satisfies \eqref{inv hyp WR}, 
where $\g = \e^{5/6}$. 
Then $F'(u)$ is invertible, and its inverse satisfies \eqref{total inverse simpler} for all $s \geq s_0$. 
\end{lemma}

The ``loss of regularity'' in \eqref{total inverse simpler} is due to the inverse $(\mL_{R}^W)^{-1}$ on the component $W, R$, while the ``loss of smallness'' $\e^{-2}$ ($\gg \g^{-1}$) in \eqref{total inverse simpler} is due to the inverse on the kernel component $V,Z$. 
The starting point $\bar u_\e$ of the Nash-Moser scheme is sufficiently accurate ($F(\bar u_\e) = O(\e^5)$)
to overcome both the loss of derivative and the loss of smallness, and therefore we do not need to distinguish the components on $W,R,V,Z$ in the Nash-Moser iteration of Section \ref{sec:NM}. 

In conclusion, we have reduced the inversion problem for the linearized operator $F'(u)$ 
to the one of inverting $\mL_{R}^W = \Pi_R F'(u)_{|W} : W \to R$.

\medskip

In view of the transformations of the next sections, it is convenient 
(although this is not the only option) 
to split the inversion problem for $\mL_{R}^W$ into its space-Fourier components $\cos(jx)$, 
with $j=0,1$ or $j \geq 2$, because these three cases lead to different situations: 
$j=0$, the space average, is the only space-frequency for which $L_{\bar\om}$ gives a triangular, not symmetrizable system;  $j=1$ is the space-frequency of the kernel $V$; 
all the other $j \geq 2$ can be studied all together using non-trivial infinite-dimensional linear transformations and a symmetrization argument (Sections \ref{sec:linearized operator}-\ref{sec:semi-FIO conj}). 

Thus, to solve the equation $\Pi_R F'(u)_{|W}[h] = f$, 
we split $R = R_0 \oplus R_1 \oplus R_2$, where the elements of $R_0$ depend only on time, 
those of $R_1$ are space-Fourier-supported on $\cos(x)$, and 
those of $R_2$ are space-Fourier-supported on $\cos(jx)$, $j \geq 2$.
Denote $R_{01} := R_0 \oplus R_1$. 
Decompose $W = W_0 \oplus W_1 \oplus W_2$ in the same way, 
and denote $W_{01} := W_0 \oplus W_1$. 
Split $h = h_{01} + h_2 \in W$, $f = f_{01} + f_2 \in R$, 
where $h_{01} \in W_{01}$, $h_2 \in W_2$, and  
$f_{01} \in R_{01}$, $f_2 \in R_2$.  
The problem $\Pi_R F'(u)_{|W}[h] = f$ becomes 
\begin{equation}  \label{system 012}
\begin{cases}
\Pi_{R_{01}} F'(u) [h_{01} + h_2] = f_{01}, 
\\
\Pi_{R_2} F'(u)[h_{01} + h_2] = f_2, 
\end{cases}
\quad \text{i.e.} \quad
\begin{cases}
\mL_{01}^{01} h_{01} 
+ \mL_{01}^2 h_2 = f_{01}, 
\\
\mL_2^{01} h_{01} 
+ \mL_2^2 h_2 = f_2, 
\end{cases}
\end{equation}
where $\mL_{01}^2 := \Pi_{R_{01}} F'(u)_{|W_2} : W_2 \to R_{01}$, etc. 

\begin{lemma} \label{lemma:inv mL 01 01}
$\mL_{01}^{01} : W_{01} \to R_{01}$ is invertible, with
\begin{equation} \label{inv mL 01 01}
\| (\mL_{01}^{01})^{-1} f \|_s \leq_s \| f \|_s + \| \tilde u \|_{s+2} \| f \|_{s_0}, \quad 
\| (\mL_{01}^{01})^{-1} f \|_{s_0} \leq_{s_0} \| f \|_{s_0}.
\end{equation}
\end{lemma}

\begin{proof}
To invert $\mL_{01}^{01}$, we write the linearized operator as $F'(u) = L_\om + \mN'(u)$, 
where $\mN = \mN_2 + \mN_3 + \ldots$ is the nonlinear component of $F$, and $L_\om$ is its linear one, defined in \eqref{Lom}. 
We begin with the invertibility of $L_\om$ as a map of $W_{01} \to R_{01}$.
Since $L_\om$ maps $W_0 \to R_0$ and $W_1 \to R_1$, ``off-diagonal'' one simply has 
$\Pi_{R_0} L_{\om | W_1} = 0$ and $\Pi_{R_1} L_{\om | W_0} = 0$. 

\emph{Step 1.} 
The restricted linear part $\Pi_{R_0} L_{\om|W_0} : W_0 \to R_0$ is invertible, because the equation 
$L_\om h = f$ in the unknown $h = (\eta, \psi) \in W_0$ with datum $ f = (\a, \b) \in R_0$ 
is the triangular system $\om \eta'(t) = \a(t)$, $\eta(t) + \om \psi'(t) = \b(t)$, where $\eta,\b$ are even and $\psi,\a$ are odd, and $\eta, \psi, \a, \b$ are functions of $t$ only 
(this calculation has been done in Section \ref{sec:Range} with $\bar\om$ instead of $\om$). 
Thus $ \| (\Pi_{R_0} L_{\om|W_0})^{-1} f \|_s \leq C \| f \|_{s-1} $ for all $s \geq 1$.

\emph{Step 2.} We prove that the restricted linear part $\Pi_{R_1} L_{\om|W_1} : W_1 \to R_1$ is invertible. Consider 
\[
\eta = \sum_{l \geq 0} \eta_{l} \cos(lt) \cos(x), \quad 
\psi = \sum_{l \geq 0} \psi_l \sin(lt) \cos(x), \quad 
\psi_0 = 0, \quad 
\eta_1 = \bar\om \psi_1,
\]
so that $(\eta,\psi) \in W_1$ (recall \eqref{def W 11}). 
Similarly, let $f = \sum_{l \geq 0} f_l \sin(lt) \cos(x)$, 
$g = \sum_{l \geq 0} g_l \cos(lt) \cos(x)$, with $f_0 = 0$ and $g_1 = - \bar\om f_1$, so that 
$(f,g) \in R_1$ (recall \eqref{def R 11}). 
Using the projection \eqref{proj 11 rule} for $l = 1$, the definition \eqref{def R 11} of $r_0$, 
the assumption $\eta_1 = \bar\om \psi_1$ and the equality $1 + \kappa = \bar\om^2$, we obtain
\[
\Pi_{R_1} L_\om(\eta,\psi) = \begin{pmatrix} 0 \\ \bar\om^2 \eta_0 \cos(x) \end{pmatrix}
+ \a \psi_1 r_0 
+ \sum_{l \geq 2} \begin{pmatrix} -( \om l \eta_l + \psi_l) \sin(lt) \cos(x) 
\\ 
(\bar\om^2 \eta_l + \om l \psi_l) \cos(lt) \cos(x) \end{pmatrix},
\]
where $\a := (\om + \bar\om)(1 + \bar\om^2) / (2 \bar\om)$. 
Thus $\Pi_{R_1} L_\om (\eta, \psi) = (f,g)$ if and only if 
$\eta_0 = g_0 / \bar\om^2$, $\psi_1 = - f_1 / \a$, 
and 
\begin{equation}  \label{inv Lom}
\eta_l = \frac{- \om l f_l - g_l }{\om^2 l^2 - \bar\om^2}, \quad 
\psi_l = \frac{\bar\om^2 f_l + \om l g_l }{\om^2 l^2 - \bar\om^2}, \quad 
l \geq 2.
\end{equation}
For all $l \geq 2$ the denominator is $\om^2 l^2 - \bar\om^2 \geq C l^2$, therefore 
$|\eta_l| + |\psi_l| \leq C(|f_l| + |g_l|) / l$ for all $l \geq 2$. 
Hence $\Pi_{R_1} L_{\om|W_1}$ is invertible, with 
$$
\| (\Pi_{R_1} L_{\om|W_1})^{-1} f \|_s \leq C \| f \|_{s-1}
$$
for all $f \in R_1$, all $s \geq 1$. 

Collecting Steps 1-2 we deduce that $\Pi_{R_{01}} L_{\om|W_{01}} : W_{01} \to R_{01}$ is invertible, with
\[
\| (\Pi_{R_{01}} L_{\om|W_{01}})^{-1} f \|_s \leq C \| f \|_{s-1} \quad \forall f \in R_{01}, 
\ s \geq 1.
\] 
\emph{Step 3.} The linear operator $\mN'(u)$ does not contain derivatives with respect to time, 
and it is a pseudo-differential operator of order 2 with respect to the space variable $x$. 
Denoting $N_{ij}$, $i,j=1,2$, the operator-matrix entries of $\mN'(u)$, 
for $h = (\eta(t) \cos(x), \psi(t) \cos(x)) \in W_1$ one simply has 
\[
\mN'(u)[h] = \begin{pmatrix} N_{11} & N_{12} \\
N_{21} & N_{22} \end{pmatrix}
\begin{pmatrix} \eta(t) \cos(x) \\
\psi(t) \cos(x) \end{pmatrix}
= 
\eta(t) \begin{pmatrix} N_{11}[\cos(x)] \\
N_{21}[\cos(x)] \end{pmatrix} 
+ \psi(t) \begin{pmatrix} N_{12} [\cos(x)] \\
N_{22} [\cos(x)] \end{pmatrix} \! ,
\]
and similarly for $h \in W_0$ (just replace $\cos(x)$ with $1$). 
Therefore, for all $h \in W_{01}$,
\begin{equation}  \label{mN' W01}
\ba
&\| \mN'(u)_{|W_{01}} [h] \|_s 			
\leq C(s_0) \e \| h \|_s + C(s) \| u \|_{s+2} \| h \|_{s_0}, \\
&\| \mN'(u)_{|W_{01}} [h] \|_{s_0}  
\leq C(s_0) \e \| h \|_{s_0},
\ea
\end{equation}  
because $\| u \|_{s_0 + 2} \leq C \e$. 
The conclusion follows by tame Neumann series.
\end{proof}

By Lemma \ref{lemma:inv mL 01 01}, we solve for $h_{01}$ in the first line of \eqref{system 012}, 
and the system becomes
\begin{equation} \label{eq for h2}
(\mL_2^2 + \mR) h_2 
= f_2 - \mL_2^{01} (\mL_{01}^{01})^{-1} f_{01},
\end{equation}
where
$$
\mR := - \mL_2^{01} (\mL_{01}^{01})^{-1} \mL_{01}^2 : W_2 \to R_2.
$$

\begin{lemma}  \label{lemma:bif remainder}
Let $u = \bar u_\e + \tilde u$, with $\| \tilde u \|_{s_0 + 2 + m} \leq C \e^{2+\d}$ 
for some $s_0 \geq 2$, $m \geq 0$, $\d>0$.
Then the operator $\mR$ defined in \eqref{eq for h2} satisfies for all $s \geq s_0$
\begin{equation}  \label{est bif remainder}
\| \mR |D_x|^m h \|_s \leq_s \e^2 \| h \|_s + \e \| \tilde u \|_{s+2+m} \| h \|_{s_0}.
\end{equation}
\end{lemma}

\begin{proof} $(\mL_{01}^{01})^{-1}$ is estimated in Lemma \ref{lemma:inv mL 01 01}, 
and $\mL_2^{01} = \Pi_{R_2} \mN'(u)_{|W_{01}}$ satisfies \eqref{mN' W01} because 
$\| \Pi_{R_2} h \|_s \leq \| h \|_s$. 
It remains to estimate $\mL_{01}^2$. 
Let $h = (\eta,\psi)$. 
By the explicit formula \eqref{linearized vero}, 
using integration by parts and the self-adjointness of the Dirichlet-Neumann operator, 
it follows that both the first and the second component of $\Pi_{01} \mN'(u)[ \pa_x^m h]$ have the form
\[ 
\Big( \int_\T (\eta a_0 + \psi b_0) \, dx \Big) 
+ \Big( \int_\T (\eta a_1 + \psi b_1) \, dx \Big) \cos(x)
\] 
for some coefficients $a_i(t,x), b_i(t,x)$, $i=0,1$, depending on $u$ and of size $O(u)$. 
Note that both the derivatives contained in $\mN'(u)$ and the additional derivatives $|D_x|^m$ go to the coefficients $a_i, b_i$ and do not affect $\eta, \psi$. 
Now, any function $f(t,x)$ of the form $f(t,x) = g(t)$ or $f(t,x) = g(t) \cos(x)$ satisfies
$\| f \|_s \leq_s \| g \|_{H^s_t}$ 
(as usual, $H^s_t$ means $H^s(\T)$ where the variable is $t \in \T$).
Also 
\[
\ba
\bigg\| \int_\T \eta(t,x) a(t,x) \, dx \bigg\|_{H^s_t} 
&\leq_s \| \eta \|_{L^2_x H^s_t} \| a \|_{L^2_x H^1_t} + \| \eta \|_{L^2_x H^1_t} \| a \|_{L^2_x H^s_t}\\
&\leq_s \| \eta \|_s \| a \|_1 + \| \eta \|_1 \| a \|_s
\ea
\]
and similarly for $\psi b$. Therefore, if $\| u \|_{3+m} \leq C$, we get
$
\| \Pi_{01} \mN'(u)[ \pa_x^m h] \|_s 
\leq_s \| h \|_s \| u \|_{3 + m} + \| h \|_1 \| u \|_{s+2+m},
$
and the lemma follows by composition.
\end{proof}

Lemma \ref{lemma:bif remainder} will be used with $m = 3/2$ or $m = 2$.

Using \eqref{system 012}, \eqref{eq for h2} and the estimates for $\mL_{01}^{01}, \mL_{01}^{2}, \mL_{2}^{01}$, we deduce the following inversion result:

\begin{lemma} \label{lemma:inv bif 2} 
Let $u = \bar u_\e + \tilde u$, with $\| \tilde u \|_{s_0 + \s} \leq C \e^{2+\d} $
and $s_0 \geq 2$, $\s \geq 4$, $\d>0$. 
Assume that $\mL_2^2 + \mR : W_2 \to R_2$ is invertible, and its inverse satisfies 
\begin{equation} \label{inv hyp W2 R2}
\| (\mL_2^2 + \mR)^{-1} g \|_s 
\leq_s \g^{-1} \big( \| g \|_{s + 2} + \g^{-1} \| \tilde u \|_{s + \s} \| g \|_{s_0 + 2} \big) 
\end{equation}
for all $s \geq s_0$, where $\g = \e^{5/6}$. 
Then $\mL_{R}^W := \Pi_R F'(u)_{|W} : W \to R$ is invertible, and its inverse satisfies \eqref{inv hyp WR} 
(with the same $\s, \g$). 
\end{lemma}

\begin{remark}  \label{rem:reduced lin to invert}
Collecting Lemmata \ref{lemma:inv bif} and \ref{lemma:inv bif 2}, 
we have reduced the inversion problem for $F'(u)$ 
to the one of inverting $\mL_2^2 + \mR : W_2 \to R_2$, 
where $\mL_2^2 = \Pi_{R_2} F'(u)_{|W_2}$ and 
$\mR$ satisfies \eqref{est bif remainder}.
Therefore our goal now is to invert $\mL_2^2 + \mR$ and to prove \eqref{inv hyp W2 R2}.
\end{remark}

We finish this section with the following lemma, which rests on both the result of section \ref{sec:approx sol} and property \eqref{using bif 123} of the linearized operator. 
Recall the definition \eqref{def norm Lipe} of the norm $\| \  \|_s^\Lipe$.

\begin{lemma}  \label{lemma:F(u0)}
Let $\om, \bar u_\e$ be as defined in \eqref{freq-ampl}-\eqref{approx sol}.
Then
$$
\| F(\bar u_\e, \om) \|_s^\Lipe \leq_s \e^5
$$
for all $s \geq 0$.
\end{lemma}

\begin{proof}
By the construction of section \ref{sec:approx sol} it follows immediately that 
$\| F(\bar u_\e, \om) \|_s \leq_s \e^5$. 
It remains to estimate the derivative 
$\pa_\xi \{ F(\bar u_\e, \om) \} = (\pa_\xi \om) \pa_t \bar u_\e + F'(\bar u_\e) [\pa_\xi \bar u_\e]$.
Recalling \eqref{approx sol} and the definition $U_\e := \bar u_1 + 2 \e \bar u_2 + 3 \e^2 \bar u_3 $, we get
\[
\pa_\xi \bar u_\e = \frac{\e}{2\xi} U_\e + 2 \e^4 \xi \bar w_4.
\]
Hence, using \eqref{using bif 123}, and recalling that $\om_2 = \bar \om_2 \xi$,
for $\xi \in [1,2]$ one has
\begin{align*}
\pa_\xi \{ F(\bar u_\e, \om) \}
& = \Big( \e^2 \bar \om_2 + \e^3 \bar \om_3 \frac32 \xi^{1/2} \Big) \pa_t \bar u_\e
+ \frac{\e}{2\xi} \, F'(\bar u_\e) [ U_\e ]
+ 2 \e^4 \xi F'(\bar u_\e) [\bar w_4]
\\ &
= \e^3 \bar \om_2 \pa_t \bar u_1 + O(\e^4) 
+ \frac{\e}{2\xi} ( -2 \e^2 \om_2 \pa_t \bar u_1 + \rho) + O(\e^4)
= O(\e^4).
\end{align*}
Thus $\| \pa_\xi \{ F(\bar u_\e, \om) \} \|_s \leq_s \e^4$, and the lemma is proved.
\end{proof}

\section{Linearized equation} \label{sec:linearized operator}

The computation of the linearized equations is based 
on formula \eqref{formula shape der} for the ``shape derivative'' of $G(\eta)\psi$.
The derivative of the two components $F_1, F_2$ of $F$ (see \eqref{F1}-\eqref{F2}) at the point $u = (\eta,\psi)$ in the direction $\tilde u = (\tilde\eta, \tilde\psi)$ is 
\begin{align*}
F_1'(u) [\tilde u ] 
& = \omega \partial_t\tilde\eta + \partial_x ( V \tilde\eta )
- G(\eta) (\tilde\psi - B \tilde\eta )
\\
F_2'(u) [ \tilde u ] 
& = \omega \partial_t \tilde\psi + V\partial_x \tilde\psi - B G(\eta) \tilde\psi 
+ ( 1 + B V_x ) \tilde\eta + B G(\eta) ( B \tilde\eta )\\
&\quad - \kappa \partial_x ((1 + \eta_x^2)^{-3/2} \partial_x \tilde\eta ),
\end{align*}
as can be checked by a direct computation, noticing that
$
B\partial_x(V\tilde\eta)-B\psi_x\tilde\eta_x+B^2\eta_x\tilde\eta_x=B V_x\tilde\eta,
$
where $B, V$ are defined in \eqref{def B V}.

\smallbreak

\emph{Notation}: any function $a$ is identified with the corresponding multiplication operators $h \mapsto ah$, and, where there is no parenthesis, composition of operators is understood. For example, 
$\pa_x c \pa_x$ means: $h \mapsto \pa_x (c \pa_x h)$.

\smallbreak

Using this notation, one can represent the linearized operator as a $2 \times 2$ operator matrix
\begin{equation}  \label{linearized vero}
\ba
F'(u)[\tilde u] 
&= F' (\eta, \psi)\begin{bmatrix} \tilde\eta  \\ \tilde\psi \end{bmatrix} 
\\
&= \begin{pmatrix} 
\om \pa_t + \partial_x V + G(\eta) B & - G(\eta) \\
(1 + B V_x) + B G(\eta) B - \kappa \pa_x c \pa_x \  & \  \om \pa_t + V \partial_x - B G(\eta) \\
\end{pmatrix} 
\begin{bmatrix} \tilde\eta  \\  \tilde\psi  \end{bmatrix},
\ea
\end{equation}
where 
\begin{equation} \label{def c}
c := (1 + \eta_x^2)^{-3/2}.
\end{equation}
The linearized operator $F'(u)$ has the following conjugation structure: 
\begin{equation} \label{Con1}
F'(u) = \mZ \mL_0 \mZ^{-1},
\end{equation} 
where 
\begin{equation}  \label{mZ mL0}
\mZ := \begin{pmatrix}  1 & 0 \\ B & 1  \end{pmatrix},
\quad 
\mZ^{-1} = \begin{pmatrix} 1 & 0 \\ - B & 1  \end{pmatrix},
\quad 
\mL_0 := \begin{pmatrix}
\om \pa_t + \partial_x V \  & \ - G(\eta) \\ 
a - \kappa \pa_x c \pa_x \ & \ \om \pa_t + V \partial_x 
\end{pmatrix},
\end{equation} 
and $a$ is the coefficient
\begin{equation}  \label{a}
a := 1 + \om B_t + V B_x.
\end{equation} 
$a,B,V,c$ are periodic functions of $(t,x)$, namely they are variable coefficients.
This conjugation structure is now well-known (see \cite{LannesJAMS}).
Formula \eqref{Con1} is verified by a direct calculus, and it is a consequence of the following two facts:

($i$) the pseudodifferential terms in $F_1'$ and $F_2'$ are equal except than for a factor $B$. Hence they cancel in the sum $ F_2' + B F_1' $;

($ii$) both in the sum $ F_2' + B F_1' $ and in $ F_2' $ the quantity $\tilde \phi := \tilde \psi - B \tilde \eta$ arises naturally and replaces $ \tilde\psi $ completely. 
This $\tilde \phi$ is the ``good unknown of Alinhac'' 
(see \cite{Ali,AM}). 

\medskip

\begin{remark} \emph{Parities.} Since $u = (\eta, \psi) \in X \times Y$, it follows that 
$B \in Y$, $V = \text{odd}(t), \text{odd}(x)$, 
$c,a = \text{even}(t), \text{even}(x)$,
and $\mZ$ maps $X \to X$ and $Y \to Y$. 
\end{remark}

We want to obtain a conjugation similar to \eqref{Con1} for $\mL_2^2 + \mR$, see Remark \ref{rem:reduced lin to invert}, where $\mR : W_2 \to R_2$ satisfies \eqref{est bif remainder} 
and $\mL_2^2 = \Pi_{R_2} F'(u)_{|W_2} : W_2 \to R_2$ 
(remember that $R_2, W_2$ are subspaces of functions that are space-Fourier-supported on $\cos(jx), j \geq 2$). 
Denote, in short, $\pp$, $\ff$ the projection 
\[
\pp h(x) = \sum_{j \geq 2} h_j \cos(jx), \quad 
\ff h(x) = h_0 + h_1 \cos(x), 
\]
where $h(x) = \sum_{j \geq 0} h_j \cos(jx)$. Clearly for all $s \geq 0$, $m \geq 0$, all $h(t,x)$,
\begin{equation} \label{norm pp ff}
\| \pp h \|_s \leq \| h \|_s, \quad 
\| |D_x|^m \ff h \|_s \leq  \| \ff h \|_s \leq \| h_0 \|_{H^s_t} + \| h_1 \|_{H^s_t}
\leq \| h \|_s.
\end{equation}
We want to conjugate $\mL_2^2 + \mR = \pp F'(u) \pp + \mR$. 
Let $\tilde \mZ := \pp \mZ \pp$. 
By the parity of $B$, the operator $\tilde \mZ$ maps $R_2 \to R_2$ and $W_2 \to W_2$. 
Moreover, by Neumann series, $\tilde\mZ$ is invertible
(see Lemma \ref{lemma:remainder mR0} below).
Using the equalities $F'(u) \mZ = \mZ \mL_0$ and $I = \pp + \ff$, we get 
\[
(\mL_2^2 + \mR) \tilde \mZ
= (\pp F'(u) \pp + \mR) \pp \mZ \pp 
= \pp \mZ \pp \mL_0 \pp + \pp \mZ \ff \mL_0 \pp - \pp F'(u) \ff \mZ \pp + \mR \pp \mZ \pp
\]
whence 
\begin{equation} \label{conj mL0 PP}
\ba
&\mL_2^2 + \mR 
= \tilde\mZ ( \tilde\mL_0 + \tilde\mR_0 ) \tilde\mZ^{-1}, \quad 
\tilde\mL_0 := \pp \mL_0 \pp, \\ 
&\tilde\mR_0 := \tilde\mZ^{-1} \{ \pp \mZ \ff \mL_0 \pp - \pp F'(u) \ff \mZ \pp + \mR \tilde \mZ \}.
\ea
\end{equation}
The remainder $\tilde \mR_0$ has size $O(u^2)$ and it is regularizing of any order in $\pa_x$. 
More precisely,

\begin{lemma}  \label{lemma:remainder mR0}
(i) Let $u = \bar u_\e + \tilde u$, with $\| \tilde u \|_{s_0 +1} \leq C \e^{2+\d}$, $s_0 \geq 5$, $\d>0$.
Then for all $s \geq s_0$ the functions $B,V$ satisfy
$\| B \|_s + \| V \|_s 
\leq_s \| u \|_{s+1} 
\leq_s \e + \| \tilde u \|_{s+1}$, and 
\begin{equation}  \label{est Z-Id}
\| (\mZ - I) h \|_s 
\leq C(s_0) \| u \|_{s_0+1} \| h \|_s + C(s) \| u \|_{s+1} \| h \|_{s_0}. 
\end{equation}
Also $\tilde \mZ$ and $\tilde \mZ^{-1}$ satisfy \eqref{est Z-Id} (with possibly larger constants $C(s_0), C(s)$).

(ii) There is $\s \geq 2$ such that, 
if $\| \tilde u \|_{s_0 + \s + m} \leq C \e^{2+\d}$,  
then the operator $\tilde \mR_0 : W_2 \to R_2$ defined in \eqref{conj mL0 PP} satisfies for all $s \geq s_0$
\begin{equation}  \label{est mR0}
\| \tilde \mR_0 |D_x|^m h \|_s \leq_s \e^2 \| h \|_s + \e \| \tilde u \|_{s+\s+m} \| h \|_{s_0}.
\end{equation}
\end{lemma}

\begin{proof}
\eqref{est Z-Id} holds because $\| (\mZ - I) h \|_s \leq \| B h \|_s$. 
For $\tilde \mZ$ use that $\| \pp h \|_s \leq \| h \|_s$, 
for $\tilde \mZ^{-1}$ use tame Neumann series. 
To get the estimate for $\tilde \mR_0$, note that in $\ff \mL_0 \pp$ and in $\pp F'(u) \ff$ there is no derivative $\pa_t$, because $\ff \pa_t \pp = 0 = \pp \pa_t \ff$. 
Also use that $\pp \mZ \ff = \pp (\mZ - I) \ff$ and $\ff \mZ \pp = \ff (\mZ - I) \pp$.
\end{proof}

\section{Changes of variables} 
\label{sec:changes}

We have arrived at the inversion problem for $\tilde\mL_0 + \tilde\mR_0$ defined in \eqref{conj mL0 PP}, where
\[
\mL_0 = \begin{pmatrix}
\om \pa_t + V \partial_x + V_x \  & \ - |D_x| - \mR_G \\ 
a - \kappa c \pa_{xx} - \kappa c_x \pa_x \ & \ \om \pa_t + V \partial_x 
\end{pmatrix},
\]
and $|D_x| + \mR_G = G(\eta)$. 
Our first goal is to obtain a constant coefficient in the term of order $\pa_{xx}$. 
To do that, we use two changes of variables: 
a space-independent change of the time variable (i.e.~a reparametrization of time), and 
a time-dependent change of the space variable. 

\medskip

We start with an elementary observation. 
Given $b_1, b_2, b_3, b_4$ functions of $(t,x)$, the system 
\[
\begin{pmatrix} 
b_1 & b_2 \\ 
b_3 & b_4 \end{pmatrix}
=
\begin{pmatrix} 
f & 0 \\ 
0 & g \end{pmatrix}
\begin{pmatrix} 
\lm_1 & \lm_2 \\ 
\lm_3 & \lm_4 \end{pmatrix}
\begin{pmatrix} 
p & 0 \\ 
0 & q \end{pmatrix}
\]
has solutions $f,g,p,q$, 
\[
f = \frac{b_1}{\lm_1 p}\,, \quad 
g = \frac{b_3}{\lm_3 p}\,, \quad
q = \frac{\lm_1 b_2 p}{b_1 \lm_2}\,, \quad
p = \text{any}\,,
\]
if and only if $b_i, \lm_i$ satisfy
\[
\frac{b_1 b_4}{b_2 b_3} \, = \frac{\lm_1 \lm_4}{\lm_2 \lm_3}\,.
\]
In particular, to have $\lm_i$ constant, it is necessary that $b_1 b_4 / b_2 b_3$ is a constant, and, at the leading order, this is the condition we want to obtain after changing the coefficients of $\mL_0$ by the changes of variables.
 
\medskip

First, consider the change of variable
$y = x + \b(t,x)$ $\Leftrightarrow$ $x = y + \tilde\b(t,y)$, 
where $\b(t,x)$ is a periodic function with $|\b_x| \leq 1/2$, and $\tilde\b(t,y)$ is given by the inverse diffeomorphism.  
Denote 
\[
(\mB h)(t,x) := h(t, x+\b(t,x)).
\]
Conjugation rules for $\mB$ are these:
$\mB^{-1} a \mB = (\mB^{-1} a)$, namely the conjugate of the multiplication operator $h \mapsto ah$ is the multiplication operator $h \mapsto (\mB\inv a)h$, and 
\begin{align*}
\mB\inv \pa_x \mB & = \{ \mB\inv(1 + \b_x) \} \pa_y, 
\\ 
\mB\inv \pa_{xx} \mB & = \{ \mB\inv(1 + \b_x) \}^2 \pa_{yy} + (\mB\inv \b_{xx}) \pa_y, 
\\
\mB\inv \pa_t \mB & = \pa_t + (\mB\inv \b_t) \pa_y, \\
\mB\inv |D_x| \mB & = \{ \mB\inv(1 + \b_x) \} |D_y| + \mR_{\mB},
\end{align*}
where $\mR_{\mB} := \{ \mB\inv(1 + \b_x) \} \pa_y (\mB\inv \mH \mB - \mH)$
is bounded in time, regularizing in space at expense of $\eta$, because 
\[ 
\ba
\mB\inv |D_x| \mB 
&= \mB\inv \pa_x \mH \mB 
= (\mB\inv \pa_x \mB) (\mB\inv \mH \mB)\\
&= \{ \mB\inv(1 + \b_x) \} \pa_y [ \mH + (\mB\inv \mH \mB - \mH) ],
\ea
\]
and $(\mB\inv \mH \mB - \mH)$ is bounded in time, regularizing in space at expense of $\eta$ (see Lemma~\ref{benj} in Section \ref{S:tame} of the Appendix). 
Thus
\[
\mL_1 := \mB\inv \mL_0 \mB 
= \begin{pmatrix}
\om \pa_t + a_1 \pa_y + a_2  \  & \ - a_3 |D_y| + \mR_1 \\ 
- \kappa a_4 \pa_{yy} - \kappa a_5 \pa_y + a_6 \ & \ \om \pa_t + a_1 \pa_y
\end{pmatrix},
\]
where the variable coefficients $a_i = a_i(t,y)$ are 
\begin{alignat*}{2}
a_1 & := \mB\inv [ \om \b_t + V(1 + \b_x)], \qquad &
a_2 & := \mB\inv(V_x),  \\
a_3 & := \mB\inv(1 + \b_x), \qquad &
a_4 & := \mB\inv [c(1 + \b_x)^2], \\
a_5 & := \mB\inv [c \b_{xx} + c_x (1 + \b_x)], \qquad &
a_6 & := \mB\inv a,
\end{alignat*}
and $\mR_1 := - \mR_\mB - \mB\inv \mR_G \mB$.

We want to conjugate $\tilde \mL_0 + \tilde \mR_0 = \pp \mL_0 \pp + \tilde \mR_0$. 
Let $\tilde \mB := \pp \mB \pp$.  
The operator $\tilde \mB$ maps $R_2 \to R_2$ and $W_2 \to W_2$ and it is invertible
(see Lemma \ref{lemma:BAPQ} below).
Using the equalities $\mL_0 \mB = \mB \mL_1$ and $I = \pp + \ff$, we get 
$$
\tilde \mB^{-1} (\tilde \mL_0 + \tilde \mR_0) \tilde \mB 
= \tilde \mL_1 + \tilde \mR_1, 
$$
with
\[
\tilde \mL_1 := \pp \mL_1 \pp, 
\quad
\tilde \mR_1
:= \tilde \mB^{-1} \{ \pp \mB \ff \mL_1 \pp - \pp \mL_0 \ff \mB \pp + \tilde \mR_0 \tilde \mB \}.
\]

Second, consider a reparametrization of time
$\t = t + \a(t)$ $\Leftrightarrow$ $t = \t + \tilde\a(\t)$,
where $\a(t)$ is a periodic function with $|\a'| \leq 1/2$, and $\tilde\a(\t)$ is given by the inverse diffeomorphism.  
Denote 
\[
(\mA h)(t,y) := h(t+\a(t), y).
\]
Conjugation rules for $\mA$ are these:
$\mA^{-1} a \mA = (\mA^{-1} a)$, 
namely the conjugate of the multiplication operator $h \mapsto ah$ is the multiplication operator $h \mapsto (\mA^{-1} a)h$, and 
\[
\mA^{-1} \pa_y \mA = \pa_y, 
\quad 
\mA^{-1} |D_y| \mA = |D_y|, 
\quad
\mA^{-1} \pa_t \mA = \{ \mA^{-1}(1 + \a') \} \pa_\t.
\]
Thus
\[
\mL_2 := 
\mA^{-1} \mL_1 \mA = \begin{pmatrix}
\om a_7 \pa_\t + a_8 \partial_y + a_9 \  & \ - a_{10} |D_y| + \mR_2 \\ 
- \kappa a_{11} \pa_{yy} - \kappa a_{12} \pa_y + a_{13} \ & \ \om a_7 \pa_\t + a_8 \partial_y
\end{pmatrix},
\]
where $\mR_2 = \mA^{-1} \mR_1 \mA$ and where the coefficients $a_i = a_i(\t,y)$ are
\[
a_7 := \mA^{-1}(1 + \a'), 
\qquad 
a_k := \mA^{-1} (a_{k-7}), \quad k=8,\ldots,13.
\]
Note that $a_7(\t)$ does not depend on $y$.

We want to conjugate $\tilde \mL_1 + \tilde \mR_1 
= \pp \mL_1 \pp + \tilde \mR_1$. 
The transformation $\mA$ maps $W_2 \to W_2$ and $R_2 \to R_2$, and commutes with $\pp$. 
Hence $\tilde \mA := \pp \mA \pp = \mA \pp$ is the restriction of $\mA$ to the subspace $W_2$ or $R_2$, and we get
\[
\tilde \mA^{-1} (\tilde \mL_1 + \tilde \mR_1) \tilde \mA 
= \tilde \mL_2 + \tilde \mR_2, 
\qquad
\tilde \mL_2 := \pp \mL_2 \pp, 
\quad
\tilde \mR_2
:= \tilde \mA^{-1} \tilde \mR_1 \tilde \mA.
\]

Following the elementary observation above, we look for $\a, \b$ such that 
\begin{equation}  \label{proportional}
\mult a_7^2 = a_{10} a_{11}
\end{equation} 
for some constant $\mult \in \R$. 
Since $a_7 = \mA^{-1} (1+\a')$, $a_{10} = \mA^{-1} a_3$, $a_{11} = \mA^{-1} a_4$, and $\mA^{-1}$ is bijective, \eqref{proportional} is equivalent to 
\begin{equation}  \label{proportional 2}
\mult (1 + \a')^2 = a_3 a_4.
\end{equation} 
Since $a_3 = \mB\inv (1+\b_x)$, $a_4 = \mB\inv[c(1+\b_x)^2]$, and $\a' = \mB\inv(\a')$, 
\eqref{proportional 2} is equivalent to 
\begin{equation}  \label{proportional 3}
\mult (1 + \a')^2 = c(1+\b_x)^3,
\end{equation} 
namely $\mult^{1/3} \, (1 + \a'(t))^{2/3} \, c(t,x)^{-1/3} = 1 + \b_x(t,x)$.
Integrating this equality in $dx$, the term $\b_x$ disappears because it has zero mean. 
Therefore 
\[
1 + \a'(t) = \mult^{-1/2} \, \Big( \meanT c(t,x)^{-1/3} \, dx \Big)^{-3/2}.
\]
Integrating the last equality in $dt$ determines the constant $\mult$: 
\begin{equation} \label{mu}
\mult = \bigg\{ \meanT \Big( \meanT c(t,x)^{-1/3} \, dx \Big)^{-3/2} \, dt \bigg\}^{2}
\end{equation}
and, by \eqref{def c}, $c^{-1/3} = (1 + \eta_x^2)^{1/2}$. 

By construction, $[\mult^{-1/2} \, ( \frac{1}{2\p} \int c^{-1/3} \, dx )^{-3/2} - 1]$ has zero average in $t$, 
therefore we can fix $\a(t)$ as
\begin{equation} \label{alpha}
\a = \pa_t\inv \Big[ \mult^{-1/2} \, \Big( \meanT c(t,x)^{-1/3} \, dx \Big)^{-3/2} - 1 \Big],
\end{equation}
where $\pa_t\inv$ is the Fourier multiplier 
\[
\pa_t\inv \, e^{ilt} = \frac{1}{il} \, e^{ilt} \quad \forall l \in \Z \setminus \{ 0 \}, 
\qquad 
\pa_t\inv \, 1 = 0,
\]
namely, for any function $f$, $\pa_t\inv f$ is the primitive of $f$ in $t$, with zero average in $t$.

By construction, $[\mult^{1/3} (1 + \a')^{2/3} c^{-1/3} - 1]$ has zero average in $x$, therefore we can fix $\b(t,x)$ as
\begin{equation} \label{beta}
\b = \pa_x\inv \big[ \mult^{1/3}(1 + \a'(t))^{2/3} c(t,x)^{-1/3} - 1 \big],
\end{equation}
where $\pa_x\inv$ is defined in the same way as $\pa_t\inv$.  
With these choices of $\a,\b$, \eqref{proportional} holds, with $\mult$ given in \eqref{mu}.
We have found formulae 
\begin{equation}  \label{formula beta alpha}
\ba
&1 + \b_x = \sqrt{1 + \eta_x^2} \,\, \Big( \frac{1}{2\p} \int_\T \sqrt{1 + \eta_x^2} \, dx \Big)^{-1}, \\
&1 + \a'(t) = \frac{1}{\sqrt \mult} \, \Big( \meanT \sqrt{1 + \eta_x^2} \, dx \Big)^{-3/2}.
\ea
\end{equation}

\begin{remark} 
Since $c \in X$, it follows that $\a = \text{odd}(t), \text{even}(x)$ and $\b = \text{even}(t), \text{odd}(x)$.
As a consequence, both the transformations $\mA$ and $\mB$ preserve parities, namely they map  
$X \to X$ and $Y \to Y$, and $\mA^{-1}, \mB\inv$ do the same. 
Therefore 
\[
\ba
&a_1, a_8 = \text{odd}(t), \odd(x); \quad
a_2, a_9 \in Y; \quad
a_3, a_4, a_6, a_7, a_{10}, a_{11}, a_{13}  \in X; \\
&a_5, a_{12} = \even(t), \odd(x). 
\ea
\]
\end{remark}

We follow the elementary observation above, with $\lm_1 = \lm_2 = \lm_4 = 1$, $\lm_3 = \mult$, 
$p = 1$.
Let
\[
\ba
&P := \begin{pmatrix}
a_7 & 0 \\ 
0 & a_{11} \mult^{-1} \end{pmatrix}, 
\quad  
P^{-1} = \begin{pmatrix}
a_7\inv  & 0 \\ 
0 &  a_{11}\inv \mult \end{pmatrix}, 
\\
&Q := \begin{pmatrix}
1  & 0 \\
0 &  a_{10}\inv a_7 \end{pmatrix},
\quad 
Q^{-1} = \begin{pmatrix}
1  & 0 \\
0 &  a_{10} a_7\inv \end{pmatrix},
\ea
\]
and calculate 
\begin{equation}  \label{mL3}
\mL_3 := P\inv \mL_2 Q
= \begin{pmatrix} 
\om \pa_\t + a_{14} \pa_y + a_{15} 
& 
- |D_y| + a_{16} \mH + \mR_3
\vspace{4pt} \\
- \mult \kappa \pa_{yy} + \mult \kappa a_{17} \pa_y + \mult a_{18} \quad
&
\om \pa_\t + a_{14} \pa_y + a_{19} 
\end{pmatrix},
\end{equation}
where, using \eqref{proportional},
\begin{alignat*}{3}
a_{14} & := \frac{a_8}{a_7}\,, 
\quad \quad & 
a_{15} & := \frac{a_9}{a_7}\,,
\quad \quad &
a_{16} & := - \frac{a_{10}}{a_7} \, \Big( \frac{a_7}{a_{10}} \Big)_y\,,
\vspace{4pt} \\
a_{17} & := - \frac{a_{12}}{a_{11}}\,,
\quad \quad & 
a_{18} & := \frac{a_{13}}{a_{11}}\,,
\quad\quad  &
a_{19} & := \mult \om \frac{a_7}{a_{11}} \, \Big( \frac{a_7}{a_{10}} \Big)_\t 
+ \mult \frac{a_8}{a_{11}} \, \Big( \frac{a_7}{a_{10}} \Big)_y 
\end{alignat*}
and
\begin{equation}  \label{def mR3}
\mR_3 :=  
- \frac{a_{10}}{a_7} \, \pa_y \Big[ \mH , \, \frac{a_7}{a_{10}} \Big] 
+ \frac{1}{a_7} \, \mR_2 \frac{a_7}{a_{10}} \,.
\end{equation}
The commutator $[\mH,f]$ of the Hilbert transform $\mH$ and the multiplication by any function $f$ is bounded in $\t$ and regularizing in $y$ at expense of $f$ 
(see Lemma~\ref{benj} in Section \ref{S:tame} of the Appendix). 
To calculate \eqref{mL3} we have used \eqref{proportional}.

We want to conjugate $\tilde \mL_2 + \tilde \mR_2 = \pp \mL_2 \pp + \tilde \mR_2$. 
Let $\tilde P := \pp P \pp$ and $\tilde Q := \pp Q \pp$. 
Using the equalities $P \mL_3 = \mL_2 Q$ and $I = \pp + \ff$, we get
$$
\tilde P^{-1} (\tilde \mL_2 + \tilde \mR_2) \tilde Q
= \tilde \mL_3 + \tilde \mR_3, 
$$
with
\begin{equation}  \label{def mL3 mR3}
\tilde \mL_3 := \pp \mL_3 \pp, 
\quad
\tilde \mR_3
:= \tilde P^{-1} \{ \pp P \ff \mL_3 \pp - \pp \mL_2 \ff Q \pp + \tilde \mR_2 \tilde Q \}.
\end{equation}
Thus we have conjugate 
\[
\tilde \mL_0 + \tilde \mR_0 
= \tilde \mB \tilde \mA \tilde P (\tilde \mL_3 + \tilde \mR_3) \tilde Q^{-1} 
\tilde \mA^{-1} \tilde \mB^{-1},
\]
and the coefficients of $\pa_\t, \pa_{yy}, |D_y|$ in $\mL_3$ are constants. 

\begin{remark} Using the parities of $a_i$, $i \leq 13$, it follows that  
\[
a_{14} = \odd(\t), \odd(y); \quad 
a_{15}, a_{19} \in Y; \quad
a_{16}, a_{17} = \even(\t), \odd(y); \quad
a_{18} \in X. \qedhere
\]
\end{remark}

\begin{lemma} \label{lemma:BAPQ}
There is $\s \geq 2$ with the following properties.
(i) Let $u = \bar u_\e + \tilde u$, with $\| \tilde u \|_{s_0 + \s} \leq C \e^{2+\d}$, $s_0 \geq 5$, $\d>0$.
Then all the operators $\tilde \mB, \tilde \mA, \tilde P, \tilde Q$ map $W_2 \to W_2$ and $R_2 \to R_2$, and they are all invertible. 
The inverse operators $\tilde \mB^{-1}, \tilde \mA^{-1}, \tilde P^{-1}, \tilde Q^{-1}$ also map $W_2 \to W_2$ and $R_2 \to R_2$. 
All theses operators satisfy, for all $s \geq s_0$, 
\begin{equation}  \label{est BAPQ}
\| A h \|_s 
\leq_s \| h \|_s + \| \tilde u \|_{s+\s} \| h \|_{s_0}, 
\qquad 
A \in \{ \tilde \mB, \tilde \mA, \tilde P, \tilde Q, \tilde \mB^{-1}, \tilde \mA^{-1}, \tilde P^{-1}, \tilde Q^{-1} \}.
\end{equation}
(ii) The functions $a_i(\t,y)$, $i=14, \ldots, 19$ satisfy
\[
\| a_{14} \|_s + \| a_{15} \|_s + \| a_{16} \|_s + \| a_{17} \|_s + 
\| a_{18} - 1 \|_s + \| a_{19} \|_s 
\leq_s \e + \| \tilde u \|_{s + \s}.
\]
(iii) If $\| \tilde u \|_{s_0 + \s + m} \leq C \e^{2+\d}$ 
for some $m \geq 0$, then for all $s \geq s_0$ 
the operator $\tilde \mR_3 : W_2 \to R_2$ defined in \eqref{def mL3 mR3} satisfies the same estimate \eqref{est mR0} as $\tilde \mR_0$, 
and the operator $\mR_3$ defined in \eqref{def mR3} satisfies 
\begin{equation}  \label{est mR3}
\| \mR_3 |D_x|^m h \|_s \leq_s \e \| h \|_s + \| \tilde u \|_{s+\s+m} \| h \|_{s_0}.
\end{equation}
\end{lemma}

\begin{proof} ($i$) 
The proof of the invertibility of $\tilde \mB$ is based on these arguments: 
$\mB$ is invertible of $X \to X$ and $Y \to Y$; 
$\mB - I$ is of order $O(\b) = O(\e^2)$ in size and of order 1 in $\pa_x$, 
therefore $\ff (\mB - I) \ff$ is small and bounded (because $\ff \pa_x$ is bounded). 
As a consequence, $\ff \mB \ff$ is invertible by Neumann series. 
Then $\pp \mB \pp$ is invertible by a standard argument of linear systems.
The invertibility of $\tilde \mA$ is trivial, because $\tilde \mA h = \mA h$ for all $h \in W_2$ or $h \in R_2$, and $\mA$ is invertible.
The invertibility of $P,Q, \tilde P, \tilde Q$ follows by Neumann series. 

($ii$) Composition estimates for all $a_1, \ldots, a_{19}$.
 
($ii$) The estimate for $\tilde \mR_3$ is proved similarly as for $\tilde \mR_0$, see Lemma \ref{lemma:remainder mR0}. 
The estimate for the term $\mR_G$ in $\mR_3$ comes from \eqref{est mR G}.
\end{proof}

\section{Symmetrization of top order}
\label{sec:symm top}

Let $g : \R \to \R$ be a $C^\infty$ function such that $g(\xi) > 0$ for all $\xi \in \R$, and 
\begin{equation}  \label{symbol g}
g(\xi) := \Big( \frac{1 + \kappa \xi^2}{|\xi|} \Big)^{\frac12} 
\quad \forall |\xi| \geq 2/3; \qquad 
g(\xi) := 1 \quad \forall |\xi| \leq 1/3.
\end{equation}
Let $\Lm$ be the Fourier multiplier of symbol $g$. 
Let
\[
S = \begin{pmatrix} 1 & 0 \\ 
0 & \mult^{1/2} \Lm \end{pmatrix}, 
\quad
S\inv = \begin{pmatrix} 1 & 0 \\ 
0 & \mult^{-1/2} \Lm^{-1} \end{pmatrix},
\]
where $\mult$ is the real constant in \eqref{mL3}, so that 
\begin{equation}  \label{conj S}
\ba
&\mL_3^+ := S\inv \mL_3 S
= \begin{pmatrix} 
A_3 & \mult^{1/2} B_3 \Lm
\vspace{4pt} \\
\mult^{-1/2} \Lm^{-1} C_3 & \Lm^{-1} D_3 \Lm
\end{pmatrix}
=: \begin{pmatrix} A_3^+ & B_3^+ \\ 
C_3^+ & D_3^+ \end{pmatrix},
\\
&\mL_3 = \begin{pmatrix} A_3 & B_3 \\ 
C_3 & D_3 \end{pmatrix}, 
\qquad 
\ea
\end{equation}
where, in short, $A_3, B_3, C_3, D_3$ are the entries of $\mL_3$ (see \eqref{mL3}).
Recall the formula for the composition of the Fourier multiplier $\Lm$ and any multiplication operator $h \mapsto a h$:
\begin{equation}  \label{symbol expansion}
\Lm (a u) \sim \sum_{n=0}^\infty \frac{1}{i^n n!} \, (\pa_x^n a)(x) \, \, \mathrm{\Op}(\pa^n_\xi g) u,
\end{equation}
where $\Op(\pa^n_\xi g)$ is the Fourier multiplier with symbol $\pa^n_\xi g(\xi)$. 
Thus $A_3^+ = A_3 = \om \pa_\t + a_{14} \pa_y + a_{15}$, 
\begin{align}
B_3^+ & = - \sqrt{\mult} |D_y|^{1/2} (1 - \kappa \pa_{yy})^{1/2} 
+ \sqrt{\mult} a_{16} \mH |D_y|^{-1/2} (1 - \kappa \pa_{yy})^{1/2} \notag \\
&\quad + \sqrt{\mult} \, \mR_3 \Lambda
\label{B3+}
\\
C_3^+ & = \sqrt{\mult} \la D_y \ra^{1/2} (1-\kappa \pa_{yy})^{1/2} 
	+ \sqrt{\mult} \, \kappa a_{17} \Lambda^{-1} \pa_y 
	+ \sqrt{\mult} \, \kappa (a_{17})_y (\Lambda^{-1})_{1} \pa_y 
\notag \\ & \quad 
	+ \sqrt{\mult} (a_{18} - 1) \Lambda^{-1} 
	+ \mR_{3,C}^+ 
\label{C3+}
\\
D_3^+ & = \om \pa_\t + a_{14} \pa_y + (a_{14})_y (\Lm^{-1})_1 \Lm \pa_y 
	+ (a_{14})_{yy} (\Lm^{-1})_2 \Lm \pa_y 
	+ a_{19} \notag\\
	&\quad+ (a_{19})_y (\Lm^{-1})_1 \Lm + \mR_{3,D}^+ 
\label{D3+}
\end{align}
where 
\begin{itemize}
\item 
$\la D_y \ra$ is the Fourier multiplier with symbol $\la \xi \ra = \max\{ 1, |\xi|\}$; 

\item
$(\Lambda^{-1})_{1} := \Op(-i \pa_\xi(1/g))$ and $(\Lambda^{-1})_{2} := \Op(- \pa^2_\xi(1/g))$
are the terms corresponding to $n=1$ and $n=2$ respectively in the expansion 
\eqref{symbol expansion} of $\Lm^{-1}$;  

\item
the remainder $\mR_{3,C}^+$ is an operator of order $O(|D_y|^{-3/2})$ regarding derivatives, 
and of size $\| (a_{17})_{yy} \| + \| (a_{18})_y \|$ (therefore $O(\e)$) regarding the amplitude;

\item
the remainder $\mR_{3,D}^+$ is an operator of order $O(|D_y|^{-2})$ regarding derivatives, 
and of size $\| (a_{14})_{yy} \| + \| (a_{19})_y \|$ (therefore $O(\e)$) regarding the amplitude.
\end{itemize}

Note that, regarding size, 
\begin{equation} \label{def T}
\pp \mL_3^+ \pp = \begin{pmatrix} 
\om \pa_t & - T \\ 
T & \om \pa_t \end{pmatrix} \pp
+ O(\e), \quad 
T := \sqrt{\mult} |D_y|^{1/2} (1 - \kappa \pa_{yy})^{1/2}.
\end{equation}
For $|j| \geq 1$, let 
\begin{equation}  \label{Taylor capillary}
r_j := \sqrt{ 1+\kappa j^2 } - \sqrt\kappa |j| - \frac{1}{2 \sqrt{\kappa} |j|}\,.
\end{equation}
Then $|r_j| \leq C_\kappa$ for all $|j| \geq 1$, and, in particular, 
$|r_j| \leq C_\kappa |j|^{-3}$ for all $|j| \geq 2 \kappa^{-1/2}$, 
for some constant $C_\kappa > 0$ depending on $\kappa$. 
Therefore, for all $h = \sum_{|j| \geq 1} h_j e^{ijy}$, 
\begin{equation}
(1 - \kappa \pa_{yy})^{1/2} h 
= \sqrt{\kappa} |D_y| h + \frac{1}{2 \sqrt\kappa} |D_y|^{-1} h + \mR_\kappa h
\label{basic decomp}
\end{equation}
where the remainder $\mR_\kappa$ has order $O(|D_y|^{-3})$. 
Similarly, we expand 
\begin{align*}
\Lm^{-1} & = \frac{1}{\sqrt{\kappa}} |D_y|^{-1/2} - \frac{1}{2 \kappa^{3/2}} |D_y|^{-5/2} + O(|D_y|^{-9/2}), 
\\
(\Lm^{-1})_1 & = - \frac{1}{2 \sqrt\kappa} |D_y|^{-3/2} \mH + O(|D_y|^{-7/2}), 
\\
(\Lm^{-1})_2 &= - \frac{3}{4 \sqrt\kappa} |D_y|^{-5/2} + O(|D_y|^{-9/2}). 
\end{align*}
Using the equality $\pa_y = - |D_y| \mH$, we get
\begin{align}
B_3^+ & = - T + \sqrt{\mult \kappa} \, a_{16} |D_y|^{1/2} \mH + O(|D_y|^{-3/2}) 
\label{B3+ bis}
\\
C_3^+ & = T + \pi_0
	- \sqrt{\mult \kappa} \, a_{17} |D_y|^{1/2} \mH 
	+ a_{25} |D_y|^{-1/2} +  O(|D_y|^{-3/2}) 
\label{C3+ bis}
\\
D_3^+ & = \om \pa_\t + a_{14} \pa_y + a_{27} + a_{28} |D_y|^{-1} \mH +  O(|D_y|^{-3/2}) 
\label{D3+ bis}
\end{align}
where $T$ is defined in \eqref{def T}, $\pi_0$ is the space average (i.e. $\pi_0(h) := \frac{1}{2\p} \int_\T h \, dx$), 
\[
\ba
a_{25} &:= \frac{\sqrt\mult}{\sqrt \kappa}\,(a_{18}-1) - \frac{1}{2} \sqrt{\mult \kappa} \, (a_{17})_y,
\qquad
a_{27} := a_{19} - \frac12 (a_{14})_y, 
\\
a_{28} &:= \frac34 (a_{14})_{yy} - \frac12 (a_{19})_y,
\ea
\]
and the remainders are of order $ O(|D_y|^{-3/2})$ regarding derivatives, and of size $O(\e)$. 
Let $\tilde S := \pp S \pp = \pp S = S \pp$, and note that $\tilde S^{-1} = S^{-1} \pp$. 
Define
\begin{equation} \label{def mL4}
\mL_4 := \begin{pmatrix} A_4 & B_4 \\ C_4 & D_4 \end{pmatrix},
\end{equation}
\begin{align*}
A_4 & := \om \pa_\t + a_{14} \pa_y + a_{15}, \\ 
B_4 & := - T + \sqrt{\mult \kappa} \, a_{16} |D_y|^{1/2} \mH, \\
C_4 & := T - \sqrt{\mult \kappa} \, a_{17} |D_y|^{1/2} \mH + a_{25} |D_y|^{1/2}, \\
D_4 & := \om \pa_\t + a_{14} \pa_y + a_{27} + a_{28} |D_y|^{-1} \mH.
\end{align*}
Thus 
\begin{equation}  \label{def mL4 mR4}
\ba
&\tilde S^{-1} ( \tilde \mL_3 + \tilde \mR_3 ) \tilde S
= \tilde \mL_4 + \tilde \mR_4, 
\\
&\tilde \mL_4 := \pp \mL_4 \pp,
\quad
\tilde \mR_4 := \pp (\mL_3^+ - \mL_4) \pp + \tilde S^{-1} \tilde \mR_3 \tilde S.
\ea
\end{equation}
Unlike in $C_3^+$, the average term $\pi_0$ is not present in $C_4$ because $\pi_0 \pp = 0$.
The remainder $\tilde \mR_4 : W_2 \to R_2$ has order $O(|D_y|^{3/2})$ regarding derivatives, and size $O(\e)$ regarding amplitude.

\begin{remark} By the parities of $a_i$, $i \leq 19$, one has
$a_{25} \in X$, $a_{27} \in Y$, $a_{28} = \odd(t), \odd(x)$.
\end{remark}

\begin{lemma} \label{lemma:S}
(i) The Fourier multiplier $S$ is an operator of order $1/2$, with $\| S h \|_s \leq C \| h \|_{s+(1/2)}$, for all $s \in \R$. $S$ is invertible, and $\| S^{-1} h \|_s \leq C \| h \|_{s}$, for all $s$. Moreover 
$\tilde S$, $\tilde S^{-1}$ satisfy the same estimates as $S$, $S^{-1}$ respectively.

(ii) There is $\s \geq 2$ such that, if $u = \bar u_\e + \tilde u$, with $\| \tilde u \|_{s_0 + \s} \leq C \e^{2+\d}$, $s_0 \geq 5$, $\d>0$, then for all $s \geq s_0$ 
\[
\| a_{25} \|_s + \| a_{27} \|_s + \| a_{28} \|_s 
\leq_s \e + \| \tilde u \|_{s + \s}.
\]
The operator $\tilde \mR_4 : W_2 \to R_2$ defined in \eqref{def mL4 mR4} satisfies 
the same estimate \eqref{est mR3} as $\mR_3$, with $m=3/2$. 
\end{lemma}

\section{Symmetrization of lower orders}
\label{sec:symm lower}

Let
\begin{equation}  \label{symm mL5}
\mL_5 := \begin{pmatrix} 
A_5 &  - C_5 \\ 
C_5 & A_5 \end{pmatrix}, 
\qquad 
\begin{array}{l}
A_5 := \om \pa_\t + a_{14} \pa_y + a_{29} + a_{30} \mH |D_y|^{-1},  
\vspace{3pt} \\
C_5 := T + a_{31} \mH |D_y|^{1/2} + a_{32} |D_y|^{-1/2},
\end{array}
\end{equation}
so that $\mL_5$ is a ``symmetrized'' version of $\mL_4$ in \eqref{def mL4}. 
The coefficients $a_{29}, a_{30}, a_{31}, a_{32}$ are unknown real-valued periodic functions of $(\t,y)$. 
We prove that there is a transformation 
\[
M =  \begin{pmatrix} 
1 & g \\ 
0 & v \end{pmatrix},
\qquad 
\begin{array}{l}
v = 1 + v_{2} \mH |D_y|^{-1} + v_{4} |D_y|^{-2}, 
\vspace{3pt} \\
g = g_{3} |D_y|^{-3/2} + g_{5} \mH |D_y|^{-5/2}, 
\end{array}
\]
where $v_2, v_4, g_3, g_5$ are real-valued, periodic functions of $(\t,y)$, 
such that $\mL_4 M - M \mL_5 = O(\Dy^{-3/2})$.

Using formula \eqref{symbol expansion} to commute $\Dy^s$ with multiplication operators, namely
\[
\Dy^s a = a \Dy^s + s a_y \Dy^{s-1} \mH 
- \frac{s(s-1)}{2} a_{yy} \Dy^{s-2} + O(\Dy^{s-3})
\]
and also the fact that $\mH$ commutes with multiplication operators up to a regularizing rest that enters in the remainder of order $O(\Dy^{-3/2})$, we calculate the entries of the matrix 
\[
\mL_4 M - M \mL_5 
= \begin{pmatrix} 
A_4 - A_5 - g C_5 
\quad & \quad 
B_4 v + C_5 + A_4 g - g A_5 
\\ 
C_4 - v C_5 
\quad & \quad  
D_4 v - v A_5 + C_4 g 
\end{pmatrix}.
\]
To make the notation uniform, we write $\pa_y = - \mH \Dy$, 
$\pa_y \mH = \Dy$, and $\mH \mH = - I + \pi_0$ ($\pi_0$ denotes the space-average).
In the following calculations, remember that $\mult, \kappa$ are constants, i.e. they do not depend on $(t,x)$.
Denote, in short, 
\[
\lm := \sqrt{\mult \kappa}.
\]

\emph{Position (1,1)}. Calculate 
$g C_5 = \lm g_{3} + ( g_{3} a_{31} + \lm g_5 ) \mH |D_y|^{-1} + O(\Dy^{-2})$.
As a consequence, $A_4 - A_5 - g C_5 = O(\Dy^{-2})$ if
\begin{align} \label{prov 2}
& a_{15} - a_{29} - \lm g_3 = 0,	
\\
& a_{30} + \lm g_5 + g_3 a_{31} = 0. 
\label{prov 3}
\end{align}

\emph{Position (1,2)}. Calculate 
\begin{gather*}
B_4 v = 
- T + \lm (a_{16} - v_2) \mH |D_y|^{1/2} 
+ \lm \big\{ \tfrac32\, (v_2)_y - v_{4} - a_{16} v_2 \big\} |D_y|^{-1/2} 
+ O(\Dy^{-3/2}),
\\
A_4 g = 
g \om \pa_\t 
- a_{14} g_3 \mH \Dy^{-1/2} 
+ O(\Dy^{-3/2}),
\\
g A_5 = 
g \om \pa_\t 
- g_{3} a_{14} \mH \Dy^{-1/2} 
+ O(\Dy^{-3/2}).
\end{gather*}
Therefore $B_4 v + C_5 + A_4 g - g A_5 = O(\Dy^{-3/2})$ if 
\begin{align} \label{prov 5}
& \lm (a_{16} - v_2) + a_{31} = 0,
\\
& \lm \{ \tfrac32\, (v_2)_y - v_{4} - a_{16} v_2 \} + a_{32} = 0.
\label{prov 6}
\end{align}

\emph{Position (2,1)}. Calculate 
\[
v C_5 =  
T + ( a_{31} + \lm v_2 ) \mH |D_y|^{1/2}
+ ( a_{32} - v_2 a_{31} + \lm v_{4} ) |D_y|^{-1/2} 
+ O(\Dy^{-3/2}).
\]
Therefore $C_4 - v C_5 = O(\Dy^{-3/2})$ if
\begin{align} \label{prov 8}
& \lm a_{17} + a_{31} + \lm v_2 = 0,
\\
& a_{25} - a_{32} + v_2 a_{31} - \lm v_{4} = 0.
\label{prov 9}
\end{align}

\emph{Position (2,2)}. 
Calculate 
\begin{align*} 
D_4 v & = 
v \om \pa_\t 
- a_{14} \mH \Dy 
+ ( a_{27} + a_{14} v_2 ) \\
&\quad
+ \big\{
\om (v_2)_\t
+ a_{14} (v_2)_y 
- a_{14} v_4 
+ a_{27} v_2 
+ a_{28} 
\big\} \mH |D_y|^{-1} 
\\ & \quad 
+ O(\Dy^{-2}),
\\
v A_5 & = 
v \om \pa_\t 
- a_{14} \mH \Dy
+ ( a_{29} + v_2 a_{14} )\\
&\quad + ( a_{30} - v_2 (a_{14})_y + v_2 a_{29} - v_{4} a_{14} ) \mH |D_y|^{-1} 
+ O(\Dy^{-2}),
\\
C_4 g & = 
\lm g_{3} 
+ \lm \{ \tfrac32\, (g_3)_y + g_5 - a_{17} g_{3} \} \mH |D_y|^{-1}
+ O(\Dy^{-2}).
\end{align*}
Therefore $ D_4 v - v A_5 + C_4 g = O(\Dy^{-3/2})$ if
\begin{gather} \label{prov 11}
a_{27} - a_{29} + \lm g_{3} = 0,
\\
\om (v_2 )_\t
+ (a_{14} v_2 )_y 
+ v_2  (a_{27} - a_{29})
+ a_{28} - a_{30} 
+ \lm \tfrac32\, (g_3)_y  
+ \lm g_5 
- \lm a_{17} g_{3}   
= 0.
\label{prov 12}
\end{gather}

\emph{Solution of the symmetrization system.} 
\eqref{prov 2}-\eqref{prov 12} is a system of 8 equations 
in the 8 unknowns $v_2$, $v_4$, $g_3$, $g_5$, 
$a_{29}$, $a_{30}$, $a_{31}$, $a_{32}$. 
First, we solve \eqref{prov 5} and \eqref{prov 8}, which give 
\begin{equation} \label{formula v2}
v_2 := \frac{1}{2} \, (a_{16} - a_{17} ), 
\qquad 
a_{31} := - \frac{\lm}{2} \, (a_{16} + a_{17}). 
\end{equation}
Next, we solve \eqref{prov 2} and \eqref{prov 11}, which give
\[
a_{29} := \frac12 \, (a_{15} + a_{27}), 
\qquad 
g_3 := \frac{1}{2 \lm} \, (a_{15} - a_{27}).
\]
Then we solve \eqref{prov 6} and \eqref{prov 9}, which give  
\[
\ba
v_{4} &:= \frac{1}{2\lm} \,\Big( \frac{3\lm}{2} \, (v_2)_y + v_2 (a_{31} - \lm a_{16}) + a_{25} \Big),
\\
a_{32} &:= \frac12 \Big( (\lm a_{16} + a_{31}) v_2 - \frac{3\lm}{2} \, (v_2)_y + a_{25} \Big),
\ea
\]
and then \eqref{prov 3} and \eqref{prov 12}, which give  
\begin{gather*}
g_5 := - \frac{1}{2\lm} \Big\{ \om (v_2 )_\t + (a_{14} v_2 )_y + v_2  (a_{27} - a_{29})
+ a_{28} + \lm \tfrac32\, (g_3)_y  + g_3 (a_{31} - \lm a_{17}) \Big\},
\\
a_{30} := \frac12 \, \Big\{ \om (v_2 )_\t
+ (a_{14} v_2 )_y 
+ v_2  (a_{27} - a_{29})
+ a_{28} 
+ \lm \tfrac32\, (g_3)_y  
- g_3 (a_{31} + \lm a_{17}) \Big\}.
\end{gather*}
System \eqref{prov 2}-\eqref{prov 12} is solved. 
To be more precise: the system is solved up to a remainder, say $\mR_c$, which is arbitrarily regularizing and is the sum of a fixed, finite number of commutators, all of the type $[a,\mH]$, where $a$ is a multiplication operator $h \mapsto ah$ by a real-valued function $a(\t,y)$. 

We have found $M, \mL_5$ such that 
$\mR_5 := \mL_4 M - M \mL_5 = O(\Dy^{-3/2})$. 
Let $\tilde M := \pp M \pp$, which is invertible by Neumann series.
By the equalities $\mL_4 M = M \mL_5 + \mR_5$ and $I = \pp + \ff$ we get
\begin{equation}  \label{def mL5}
\tilde M^{-1} (\tilde \mL_4 + \tilde \mR_4) \tilde M
= \tilde \mL_5 + \tilde \mR_5,
\qquad
\tilde \mL_5 := \pp \mL_5 \pp,
\end{equation}
where the remainder 
\begin{equation}  \label{def mR5}
\tilde \mR_5 := \tilde M^{-1} \{ \tilde \mR_4 \tilde M
+ \pp \mR_5 \pp + \pp M \ff \mL_5 \pp - \pp \mL_4 \ff M \pp \}
\end{equation}
has order $O(|D_y|^{-3/2})$ and size $O(\e)$.

\medskip

Both $\eta$ and $\psi$ are real-valued. 
Therefore, using the complex representation $h := \eta + i \psi \in \C$ of the pair $(\eta,\psi) \in \R^2$, 
$\eta = \Re(h)$, $\psi = \Im(h)$,
\begin{equation} \label{mL5} 
\mL_5 = \om \pa_\t + i T 
+ a_{14} \pa_y + i a_{31}  \mH  |D_y|^{1/2}
+ a_{29} + i a_{32} |D_y|^{-1/2}
+ a_{30} \mH |D_y|^{-1}.
\end{equation}

\begin{remark} By the parity of $a_i$, $i \leq 28$, it follows that 
\[
a_{29}, g_3 \in Y; \quad 
a_{30},g_5 = \odd(t), \odd(x); \quad 
a_{32}, v_4 \in X; \quad
a_{31}, v_2 = \even(t), \odd(x).
\]
Hence $v$ maps $X \to X$ and $Y \to Y$ (it preserves the parity), 
$g$ maps $X \to Y$ and $Y \to X$ (it changes the parity), 
and  $M$ maps the product space $X \times Y = \{ (\eta,\psi) : \eta \in X, \psi \in Y \}$ into itself. 
In complex notation, $M : (X + i Y) \to (X + i Y)$. 
The operator $\mL_5$ maps $(X + i Y) \to (Y + i X)$ (and this is obvious, because $\mL_5 = M^{-1} \mL_4 M$, and $\mL_4 : X \times Y \to Y \times X$).
\end{remark}

\begin{lemma}
There is $\s \geq 2$ such that, 
if $u = \bar u_\e + \tilde u$, with $\| \tilde u \|_{s_0 + \s} \leq C \e^{2+\d}$, $s_0 \geq 5$, $\d>0$,
then 
\[
\| a_{29} \|_s + \| a_{30} \|_s + \| a_{31} \|_s + \| a_{32} \|_s 
+ \| v_2 \|_s + \| v_4 \|_s + \| g_3 \|_s + \| g_5 \|_s 
\leq_s \e + \| \tilde u \|_{s + \s},
\]
and $\tilde \mR_5 : W_2 \to R_2$ defined in \eqref{def mR5} satisfies 
the same estimate \eqref{est mR3} as $\mR_3$, with $m=3/2$. 
\end{lemma}

\section{Reduction to constant coefficients} 
\label{sec:semi-FIO conj}

We have arrived to $\tilde \mL_5 + \tilde \mR_5$, where $\tilde \mL_5 = \pp \mL_5 \pp$,
$\mL_5$ is defined in \eqref{def mL5} (and $T$ in \eqref{def T}). 
Rename the variables $y = x$, $\t = t$.
Consider a transformation $A$ of the form
\[
h(t,x) = \sum_{j \in \Z} h_j(t) \, e^{ijx} 
\quad \mapsto \quad
Ah(t,x) = \sum_{j \in \Z} h_j(t) \, p(t,x,j) \, e^{i \phi(t,x,j)}, 
\]
where the amplitude $p(t,x,j)$ is a symbol of order zero, periodic in $(t,x)$,
the phase function $\phi(t,x,j)$ is of the form
\begin{equation} \label{phase function phi}
\phi(t,x,j) = j x + |j|^{1/2} \b(t,x),
\end{equation}
and $\b(t,x)$ is a periodic function, with $|\b_x(t,x)| < 1/2$. 
$A$ is a periodically-$t$-dependent $x$-Fourier integral operator with non-homogeneous phase function.
Moreover, $A$ is also the pseudo-differential operator $\Op(a)$ of symbol 
\begin{equation}  \label{exotic symbol}
a(t,x,j) := p(t,x,j) e^{i |j|^{1/2} \b(t,x)}
\end{equation}
in H\"ormander class $S^0_{\frac12, \frac12}$ (except for the fact that $a$ in \eqref{exotic symbol} has finite regularity). 
Let 
\begin{equation} \label{mD}
\mD := \om \pa_t + i T + i \lm_1 \Dx^{1/2} + i \lm_{-1} |D_x|^{-1/2}, 
\end{equation}
with $\lm_{1}, \lm_{-1} \in \R$. 
In this section we prove that there exist real constants $\lm_{1}$, $\lm_{-1}$ 
and functions $p(t,x,j), \b(t,x)$ such that $\mL_5 A - A \mD = O(\Dx^{-3/2})$.

Let $\t : \R \to \R$ be a $C^\infty$ function such that 
\begin{equation}  \label{symbol tau}
\tau(\xi) = \{ \mult |\xi| (1+\kappa \xi^2)\}^{1/2} 
\quad \forall |\xi| \geq 2/3; 
\qquad 
\tau(\xi) = 0 
\quad \forall |\xi| \leq 1/3,
\end{equation}
so that $\Op(\tau) = T$ on the periodic functions. 
The commutator $[T,A]$ is given by 
\[
[T,A] = \sum_{n=1}^5 \frac{1}{n!} \, \Op(\pa_x^n a) \circ \Op( i^{-n} \pa_\xi^n \tau) + O(|D_x|^{-3/2}),
\] 
where $a$ is defined in \eqref{exotic symbol}. 
Here only terms with $n \leq 5$ are relevant for our purpose, 
because $\Op(\pa_x^n a) = O(|D_x|^{n/2})$, $\Op( i^{-n} \pa_\xi^n \tau) = O(|D_x|^{\frac32 - n})$, 
and $\frac{n}{2} + \frac32 - n \leq - \frac32$ for all $n \geq 6$. 
Now, using \eqref{Taylor capillary} and proceeding like in \eqref{basic decomp}, 
\begin{align*}
\Op( i^{-1} \pa_\xi \tau) 
& = \frac{3\lm}{2}\, |D_x|^{1/2} \mH - \frac{\sqrt\mult}{4 \sqrt\kappa}\, |D_x|^{-3/2} \mH 
+ O(|D_x|^{-7/2}), 
\\
\Op( i^{-2} \pa_\xi^2 \tau) 
& = - \frac{3\lm}{4}\, |D_x|^{-1/2} + O(|D_x|^{-5/2}),
\\
\Op( i^{-3} \pa_\xi^3 \tau) &= \frac{3\lm}{8}\, |D_x|^{-3/2} \mH + O(|D_x|^{-7/2}),
\\
\Op( i^{-4} \pa_\xi^4 \tau) 
& = \frac{9\lm}{16}\, |D_x|^{-5/2} + O(|D_x|^{-9/2}),
\\
\Op( i^{-5} \pa_\xi^5 \tau) &= - \frac{45 \lm}{32}\, |D_x|^{-7/2} \mH + O(|D_x|^{-11/2}),
\end{align*}
while
\begin{align*}
\pa_x a 
& = \{ i |j|^{1/2} p \b_x + p_x \} e^{i|j|^{1/2} \b},
\\
\pa_x^2 a 
&= \{ - |j| p \b_x^2 + i |j|^{1/2} ( 2 p_x \b_x + p \b_{xx} ) + p_{xx} \} e^{i|j|^{1/2} \b},
\\
\pa_x^3 a 
& = \{ - i |j|^{3/2} p \b_x^3 
- 3 |j| \b_x (p_x \b_x + p \b_{xx})
\\ & \quad
+ i |j|^{1/2} (3 p_{xx} \b_{x} + 3 p_x \b_{xx} + p \b_{xxx}) 
+ O(|j|^0) \} e^{i|j|^{1/2} \b},
\\
\pa_x^4 a 
& = \{ |j|^{2} p \b_x^4 
- i |j|^{3/2} (4 p_x \b_x^3 + 6 p \b_x^2 \b_{xx}) + O(|j|) \} e^{i|j|^{1/2} \b},
\\
\pa_x^5 a 
& = \{ i |j|^{5/2}  p \b_x^5 + O(|j|^2) \} e^{i|j|^{1/2} \b}.
\end{align*}
The composition $\pa_x A$ can be computed directly: 
\[
\pa_x A h 
= \sum_{j \in \Z} \hat h_j (i j p + i |j|^{1/2} \b_x p + p_x) e^{i \phi(t,x,j)} .
\]
Composition formulae for $\Dx^r A$, $\mH \Dx^r A$ are given in Lemma 
\ref{lemma:composition formula}. In particular, 
the composition $\mH \Dx^{1/2} A$ is 
\begin{align*} 
\mH \Dx^{1/2} A h 
& = \sum_{j \in \Z} \hat h_j c(t,x,j) e^{i \phi(t,x,j)},
\\
c(t,x,j) & 
= - i \, \sgn(j) |j|^{1/2} p 
- i \, \frac12\, \b_x p 
+ |j|^{-1/2} \Big\{ i \,\frac18\, \sgn(j) \b_x^2 p - \frac{1}{2} \,p_x \Big\} 
\\ & \quad 
+ |j|^{-1} \Big\{ 
\frac{1}{8} \, \sgn(j) \, (2 \b_x p_x + \b_{xx} p)
- i \, \frac{1}{16} \, \b_x^3 p \Big\} 
+ O(|j|^{-3/2});
\end{align*}
the composition $\Dx^{-1/2} A$ is 
\[
\ba
\Dx^{-1/2} A h 
& = \sum_{j \in \Z} \hat h_j c(t,x,j) e^{i \phi(t,x,j)},
\\
c(t,x,j) & = 
|j|^{-1/2} p - |j|^{-1} \, \frac12\, \sgn(j)  \b_x p  
+ O( |j|^{-3/2});
\ea
\]
the composition $\mH \Dx^{-1} A$ is 
\[
\mH \Dx^{-1} A h 
= \sum_{j \in \Z} \hat h_j c(t,x,j) e^{i \phi(t,x,j)},
\quad 
c(t,x,j) = - |j|^{-1} \, i \, \sgn(j) \, p + O( |j|^{-3/2});
\]
and the commutator $[\pa_t, A] = \pa_t A - A \pa_t$ is 
\[
[\pa_t, A] h 
= \sum_{j \in \Z} \hat h_j c(t,x,j) e^{i \phi(t,x,j)},
\quad 
c(t,x,j) = p_t(t,x,j) + i |j|^{1/2} \b_t(t,x) p(t,x,j).
\]
Using the expansions above, the difference $\mL_5 A - A \mD$ is 
an operator of phase function $\phi$ as in \eqref{phase function phi}, and amplitude 
\begin{equation} \label{amplitude to kill}
c(t,x,j) =  
i j p \, \Big( \frac32\, \lm \b_x + a_{14} \Big) 
+ \sum_{-2 \leq k \leq 1} |j|^{k/2} \, T^{(k)}[p] + O(|j|^{-3/2}),
\end{equation} 
where $T^{(k)}$ are the linear differential operators 
\begin{align*}
& T^{(1)} := v^{(1)}_1 \pa_x + v^{(1)}_0 - i \lm_1, \\
& T^{(0)} := \om \pa_t + v^{(0)}_1 \pa_x + v^{(0)}_0, 
\\
& T^{(-1)} := v^{(-1)}_2 \pa_{xx} + v^{(-1)}_1 \pa_x + v^{(-1)}_0 - i \lm_{-1}, \\
& T^{(-2)} := v^{(-2)}_2 \pa_{xx} + v^{(-2)}_1 \pa_x + v^{(-2)}_0, 
\end{align*}
with coefficients  
\begin{alignat}{2} 
\label{v(1)1 v(1)0} 
v^{(1)}_1 
& := \sgn(j) \frac{3\lm}{2}\,, 
\quad &
v^{(1)}_0 
& := \sgn(j) a_{31} + i \Big( \om  \b_t + \frac{3 \lm}{8}\, \b_x^2 + a_{14} \b_x \Big), 
\\
\label{v(0)1 v(0)0} 
v^{(0)}_1
& := \frac{3\lm}{4} \, \b_x + a_{14},
\quad & 
v^{(0)}_0
& := \Big( \frac{3\lm}{8} \, \b_{xx} + \frac12\, a_{31} \b_x + a_{29} \Big) 
- i \sgn(j) \frac{\lm}{16} \, \b_x^3,
\\
\label{v(-1)2 v(-1)1} 
v^{(-1)}_2
& := - i \frac{3\lm}{8}\,,
\quad &
v^{(-1)}_1
& := - \sgn(j) \frac{3\lm}{16} \, \b_x^2 - i \frac{1}{2}\, a_{31},
\end{alignat}
\vspace{-20pt}
\begin{gather} 
\label{v(-1)0}
v^{(-1)}_0
:= - \sgn(j) \Big( \frac{3\lm}{16} \, \b_x \b_{xx} + \frac18\, a_{31} \b_x^2 \Big) 
+ i \Big( \frac{3\lm}{2^7} \, \b_x^4 + a_{32} \Big),
\\
\label{v(-2)2 v(-2)1}
v^{(-2)}_2 
:= i \sgn(j) \frac{3\lm}{16} \, \b_x,
\qquad 
v^{(-2)}_1
:= \frac{3\lm}{32} \, \b_x^3 
+ i \sgn(j) \Big( \frac{3\lm}{16} \, \b_{xx} + \frac{1}{4} \, a_{31} \b_x \Big),
\end{gather}
\vspace{-20pt}
\begin{align} \label{v(-2)0}
v^{(-2)}_0
& := \Big( \frac{9\lm}{64} \, \b_x^2 \b_{xx} + \frac{1}{16} \, a_{31} \b_x^3 \Big)  
+ i \sgn(j) \Big( - \frac{3\lm}{2^8} \, \b_x^5 + \frac{\lm}{16} \, \b_{xxx} 
+ \frac{1}{8} \, a_{31} \b_{xx} 
\\ & \qquad \notag
- \frac12\, a_{32} \b_x 
- a_{30} - \frac{\sqrt\mult}{4 \sqrt\kappa} \, \b_x \Big).
\end{align}
Our goal is to choose $\lm_k, \b, p$ such that the amplitude $c(t,x,j)$ in \eqref{amplitude to kill} is of order $O(|j|^{-3/2})$.

\medskip

\emph{Elimination of the order $1$}. --- 
$a_{14}$ is $\odd(t), \odd(x)$, therefore it has zero space-average, and 
$\pa_x\inv a_{14}$, which is the $dx$-primitive of $a_{14}$ with zero average, is well-posed. 
We fix 
\begin{equation} \label{beta 1}
\b(t,x) := \b_0(t) + \b_1(t,x), \quad 
\b_1 := - \frac{2}{3\lm}\, \pa_x^{-1} a_{14},
\end{equation}
where $\b_0(t)$ is a periodic function of $t$ only, which will be determined later (see the next step). We have eliminated the terms of order $O(|j|)$ from \eqref{amplitude to kill}. 
Since $a_{14}$ is odd$(t)$, odd$(x)$, we get $\b_1 \in Y$.

We seek $p$ under the form 
\[
p(t,x,j) = \sum_{-3 \leq m \leq 0} |j|^{m/2} p^{(m)}(t,x,j),
\]
with all $p^{(m)}$ bounded in $j$. 
Then, by linearity,  
\eqref{amplitude to kill} becomes
\begin{align}
c = {} & \sum_{\begin{subarray}{c} -2 \leq k \leq 1 \\ -3 \leq m \leq 0 \end{subarray}} 
|j|^{\frac{k+m}{2}} T^{(k)}[p^{(m)}] + O(|j|^{-3/2}).
\label{c descent}
\end{align}

\medskip

\emph{Elimination of the order $1/2$}. --- 
To eliminate the term of order $1/2$ from \eqref{c descent}, 
we have to solve the equation 
\begin{equation} \label{eq for p0}
T^{(1)}[p^{(0)}] = 0	
\end{equation}
in the unknown $p^{(0)}$. 
Write $p^{(0)}$ as 
\begin{equation}  \label{p0 as exp}
p^{(0)}(t,x,j) = \exp \big( f(t,x,j) \big),
\end{equation}
so that \eqref{eq for p0} becomes the equation 
\begin{equation} \label{eq for f}
v^{(1)}_1 f_x + v^{(1)}_0 - i \lm_1 = 0
\end{equation}
for the unknown $f$.
The coefficients $v^{(1)}_1$, $v^{(1)}_0$ are given in \eqref{v(1)1 v(1)0}. 
Equation \eqref{eq for f} has a solution $f$ if and only if 
\begin{equation} \label{descent average condition}
\int_\T \big( v^{(1)}_0(t,x,j) - i \lm_1 \big) \, dx = 0 \quad \forall t \in \T, \ j \in \Z. 
\end{equation}
We look for $\b_0(t), \lm_1$ such that the (crucial) average condition 
\eqref{descent average condition} holds.
Remember that $\b = \b_0 + \b_1$, where $\b_1(t,x)$ has already been determined in \eqref{beta 1} and $\b_0(t)$ is still free; $\b_x = (\b_1)_x$ because $\b_0$ depends only on $t$. 
Moreover, $\int_\T a_{31} \, dx = 0$ because $a_{31}$ is odd in $x$. 
Therefore \eqref{descent average condition} becomes
\[
\om \pa_t \b_0(t) - \lm_1 + \rho(t) = 0,
\]
where 
\[
\rho(t) 
:= \frac{1}{2\p} \, \int_0^{2\p} \Big( \om (\b_1)_t + \frac{3\lm}{8}\, (\b_1)_x^2 + a_{14} (\b_1)_x \Big) \, dx 
= - \frac{1}{4 \p \lm} \int_\T a_{14}^2 \, dx
\]
($(\b_1)_t$ has zero space-average because $\pa_x^{-1} a_{14}$ has zero space-average).
We fix 
\begin{equation}  \label{formula lm 1}
\lm_1 
:= \frac{1}{2\p} \, \int_0^{2\p} \rho(t) \, dt 
= - \frac{1}{8 \p^2 \lm} \, \int_{\T^2} a_{14}^2 \, dx dt,
\qquad
\b_0 := - \frac{1}{\om} \, \pa_t\inv ( \rho - \lm_1 ), 
\end{equation}
and \eqref{descent average condition} is solved. 
$\lm_1$ is a negative real number, and $\b_0$ is a real-valued function of $t$, independent on $x,j$. Then \eqref{eq for f} has solutions 
\begin{equation} \label{f = f0 + f1}
f(t,x,j) := f_0(t,j) + f_1(t,x,j), \quad 
f_1 := - \frac{2}{3\lm} \, \sgn(j) \, \pa_x\inv (v^{(1)}_0 - i \lm_1)
\end{equation}
where $f_0$ does not depend on $x$ and it will be determined in the next step. 
$p^{(0)} = \exp(f)$ solves \eqref{eq for p0}. 
Since $\b_1 \in Y$, it follows that $\rho \in X$, and therefore $\b_0 \in Y$. Thus 
$\b \in Y$, namely $\b=\b(t,x)$ is a real-valued function, $\odd(t)$, $\even(x)$, independent of $j$. 
A direct calculation gives
\[
f_1 = - \frac{2}{3\lm} \pa_x^{-1} a_{31} + i \sgn(j) \frac{2}{3\lm} \pa_x^{-1} \Big( \frac{1}{2\lm} a_{14}^2 + \rho + \frac{2\om}{3\lm} \pa_x^{-1} \pa_t a_{14} \Big).
\]

\begin{remark} \label{rem:parity coefficients v(k)m}
By the parity of $a_i$, $i \leq 34$, and $\b \in Y$, it follows that the coefficients $v^{(k)}_m$ have the form 
\begin{align*}
v^{(1)}_1, v^{(-1)}_1 & = \sgn(j) a + i b, \quad 
a \in X, \quad 
b = \even(t), \odd(x);
\\
v^{(1)}_0, v^{(-1)}_2, v^{(-1)}_0 & = \sgn(j) a + i b, \quad 
a = \even(t), \odd(x), \quad 
b \in X; 
\\
v^{(0)}_1, v^{(-2)}_1 & = a + i \sgn(j) b, \quad 
a = \odd(t), \odd(x), \quad 
b \in Y; 
\\
v^{(0)}_0, v^{(-2)}_2, v^{(-2)}_0 & = a + i \sgn(j) b, \quad 
a \in Y, \quad 
b = \odd(t), \odd(x),
\end{align*}
where $a,b$ denote (different) real-valued functions of $(t,x)$, independent of $j$. 
\end{remark}

By the parity of $v^{(1)}_0$ (see the previous remark), $f_1$ has the form
\begin{equation} \label{f1 components}
f_1 = a + i \sgn(j) b, \quad 
a \in X, \quad 
b = \even(t), \odd(x),
\end{equation}
with $a,b$ real-valued functions of $(t,x)$, independent of $j$. 
In particular, $f_1$ is $\even(t)$.

\medskip

\emph{Elimination of the order $0$}. --- 
The order zero in \eqref{c descent} vanishes if 
\begin{equation} \label{eq for p-1}
T^{(1)}[p^{(-1)}] + T^{(0)}[p^{(0)}] = 0.
\end{equation}
In general, for any function $g$, one has
\begin{equation} \label{automatic T(1)}
T^{(1)}[ \exp(f) g ] = \exp(f) \, v^{(1)}_1 g_x 
\end{equation}
by \eqref{eq for f}, and 
\begin{equation} \label{automatic T(0)}
T^{(0)}[ \exp(f) g ] = \exp(f) \, \big( b^{(0)} g + \om g_t + v^{(0)}_1 g_x \big), \quad 
b^{(0)} := \om f_t + v^{(0)}_1 f_x + v^{(0)}_0.
\end{equation}
In particular, for $g=1$, we get $T^{(0)}[p^{(0)}] = \exp(f) \, b^{(0)}$. 
By variation of constants, write $p^{(-1)}$ as
\[
p^{(-1)} := p^{(0)} g^{(-1)} = \exp(f) g^{(-1)}.
\]
Equation \eqref{eq for p-1} becomes 
\begin{equation} \label{eq for g-1}
v^{(1)}_1 g^{(-1)}_x + b^{(0)} = 0
\end{equation}
in the unknown $g^{(-1)}$. 
Since $v^{(1)}_1$ is a constant, \eqref{eq for g-1} has a solution $g^{(-1)}$ if and only if 
\begin{equation} \label{average condition b(0)}
\int_\T b^{(0)}(t,x,j) \, dx = 0 \quad \forall t \in \T, \ j \in \Z.
\end{equation}
Remember that $f = f_0 + f_1$, where $f_1$ has already been determined in \eqref{f = f0 + f1}, and $f_0 = f_0(t,j)$ is still at our disposal. Thus
$b^{(0)} = \om (f_0)_t + \om (f_1)_t + v^{(0)}_1 (f_1)_x + v^{(0)}_0$. 
By \eqref{f1 components} and Remark \ref{rem:parity coefficients v(k)m}, 
\[
\om (f_1)_t + v^{(0)}_1 (f_1)_x + v^{(0)}_0 = a + i \sgn(j) b, \quad 
a \in Y, \quad 
b = \odd(t) \odd(x),
\]
for some $a,b$ real-valued functions of $(t,x)$, independent of $j$. Therefore 
\[
\int_\T \Big( \om (f_1)_t + v^{(0)}_1 (f_1)_x + v^{(0)}_0 \Big) dx = \odd(t)
\]
is a real-valued function of $t$ only, independent of $x,j$, with zero mean (because it is odd).
We fix 
\[
f_0 := - \frac{1}{2 \p \om} \, \pa_t\inv \Big\{ \int_\T 
\Big( \om (f_1)_t + v^{(0)}_1 (f_1)_x + v^{(0)}_0 \Big) \, dx \Big\},
\]
and \eqref{average condition b(0)} is satisfied. $f_0$ is a real-valued even function of $t$ only, independent of $x,j$. A direct calculation gives
\[
f_0 = - \frac{1}{2 \p \om} \, \pa_t\inv \Big\{ \int_\T \Big( a_{29} - \frac{2}{3\lm} a_{14} a_{31} \Big) dx \Big\}.
\]
\begin{remark} \label{rem:parity f b(0)}
$f$  and $b^{(0)}$ are of the form 
\begin{align} \label{f components}
f & = a + i \sgn(j) b, \quad 
a \in X, \quad 
b = \even(t), \odd(x),
\\
b^{(0)} & = a + i \sgn(j) b, \quad 
a \in Y, \quad b = \odd(t), \odd(x),
\label{b(0) components}
\end{align}
where $a,b$ denote (different) real-valued functions of $(t,x)$, independent of $j$. 
\end{remark}

We choose 
\begin{equation} \label{g(-1)1}
g^{(-1)}(t,x,j) := g^{(-1)}_0(t,j) +  g^{(-1)}_1(t,x,j), \quad 
g^{(-1)}_1 := - \frac{2}{3\lm} \,\sgn(j) \,  \pa_x\inv (b^{(0)}),
\end{equation}
where $g^{(-1)}_0(t,j)$ will be determined at the next step. 
\eqref{eq for g-1} is satisfied. 
$g^{(-1)}_1$ is of the form
\begin{equation}  \label{decomposition g(-1)1}
g^{(-1)}_1 = \sgn(j) a + i b, \quad 
a = \odd(t),\odd(x), \quad 
b \in Y, 
\end{equation}
for some $a,b$ real-valued functions of $(t,x)$, independent of $j$. 

\medskip

\emph{Elimination of lower orders}. --- Once the first two steps in $p$ are done (i.e. elimination of orders $1/2$ and $0$), the algorithm proceeds in a similar way. 
For the sake of completeness (and to obtain $\lm_{-1}$), we write the calculations for the order $-1/2$ in details, then lower orders will be similar.

\medskip

\emph{Elimination of the order $-1/2$}. --- 
We have to solve 
\begin{equation} \label{eq for p(-2)}
T^{(1)}[p^{(-2)}] + T^{(0)}[p^{(-1)}] + T^{(-1)}[p^{(0)}] = 0
\end{equation}
in the unknown $p^{(-2)}$ (and also $g^{(-1)}_0$ is still free). 
By variation of constants, write $p^{(-2)} = \exp(f) g^{(-2)}$. 
By \eqref{automatic T(1)} and \eqref{automatic T(0)},
\[
\ba
T^{(1)}[p^{(-2)}] &= \exp(f) \, v^{(1)}_1 g^{(-2)}_x , \\
T^{(0)}[p^{(-1)}] &= \exp(f) \, \big\{ b^{(0)} g^{(-1)} + \om g^{(-1)}_t + v^{(0)}_1 g^{(-1)}_x \big\}. 
\ea
\]
Recall that $g^{(-1)} = g^{(-1)}_0 + g^{(-1)}_1$. 
Let 
\begin{equation} \label{b(-1) r(-1)}
b^{(-1)} := \exp(-f) \, \Big( T^{(0)}[p^{(-1)}] + T^{(-1)}[p^{(0)}] \Big) 
\ = b^{(0)} g^{(-1)}_0 + \om (g^{(-1)}_0)_t + r^{(-1)},
\end{equation}
where 
\[ 
r^{(-1)} := b^{(0)} g^{(-1)}_1 + \om (g^{(-1)}_1)_t + v^{(0)}_1 (g^{(-1)}_1)_x 
+ \exp(-f) \, T^{(-1)} [p^{(0)}].
\]
Thus \eqref{eq for p(-2)} becomes 
\begin{equation} \label{eq for g(-2)}
v^{(1)}_1 g^{(-2)}_x + b^{(-1)} = 0.
\end{equation} 
If
\begin{equation} \label{average condition b(-1)}
\int_\T b^{(-1)}(t,x,j) \, dx = 0 \quad \forall t \in \T, \ j \in \Z,
\end{equation}
then \eqref{eq for g(-2)} has a solution $g^{(-2)}$.
By \eqref{b(-1) r(-1)} and \eqref{average condition b(0)}, the average condition \eqref{average condition b(-1)} becomes
\begin{equation} \label{eq for g(-1)0}
2 \p \om (g^{(-1)}_0)_t(t,j)  + \int_\T r^{(-1)}(t,x,j) \, dx = 0.
\end{equation} 
If
\begin{equation} \label{average condition r(-1)}
\int_{\T^2} r^{(-1)}(t,x,j) \, dx dt = 0,
\end{equation} 
then we can choose $g^{(-1)}_0$ such that \eqref{eq for g(-1)0} is satisfied.  
By \eqref{decomposition g(-1)1}, \eqref{b(0) components} and Remark \ref{rem:parity coefficients v(k)m}, one proves that $r^{(-1)}$ has the form 
\begin{equation} \label{decomposition r(-1)}
r^{(-1)} = \sgn(j) a + i b - i \lm_{-1}, \quad 
a = \even(t),\odd(x), \quad 
b \in X,
\end{equation}
for some $a,b$ real-valued functions of $(t,x)$, independent of $j$. 
We fix 
\begin{equation} \label{lambda -1}
\lm_{-1} := \frac{1}{(2\p)^2} \int_{\T^2} b(t,x) \, dx dt,
\end{equation}
where $b$ in \eqref{lambda -1} is the function $b$ in \eqref{decomposition r(-1)}, so that \eqref{average condition r(-1)} is satisfied. Note that $\lm_{-1}$ is a real number. Then we fix 
\[
g^{(-1)}_0 = - \frac{1}{2 \p \om} \pa_t^{-1} \Big( \int_\T r^{(-1)}(t,x,j) \, dx \Big),
\]
and \eqref{eq for g(-1)0} is satisfied. By \eqref{decomposition r(-1)} it follows that $g^{(-1)}_0$ is a purely imaginary, odd function of $t$, independent of $x,j$, and therefore 
\begin{equation} \label{decomposition g(-1)}
g^{(-1)} = \sgn(j) a + i b, \quad 
a = \odd(t),\odd(x), \quad 
b \in Y,
\end{equation}
for some $a,b$ real-valued functions of $(t,x)$, independent of $j$.
Thus \eqref{average condition b(-1)} is satisfied. We choose 
\begin{equation} \label{g(-2)1}
g^{(-2)} = g^{(-2)}_0(t,j) + g^{(-2)}_1(t,x,j), \quad 
g^{(-2)}_1 := - \sgn(j) \frac{2}{3\lm} \, \pa_x^{-1} b^{(-1)},
\end{equation} 
where $g^{(-2)}_0$ is free (it will be fixed in the next step), 
so that \eqref{eq for g(-2)} is satisfied. $b^{(-1)}$ is of the form
\[
b^{(-1)} = \sgn(j) a + i b, \quad 
a = \even(t), \odd(x), \quad 
b \in X,
\]
therefore $g^{(-2)}_1$ is of the form
\[
g^{(-2)}_1 = a + i \sgn(j) b, \quad 
a \in X, \quad 
b = \even(t), \odd(x), 
\]
where $a,b$ denote (different) real-valued functions of $(t,x)$, independent of $j$.

\medskip

\emph{Elimination of the order $-1$}. --- 
We proceed similarly as in the previous step, with $T^{(-1)}[p^{(-1)}]$ $+$ $T^{(-2)}[p^{(0)}]$ instead of $T^{(-1)}[p^{(0)}]$; $g^{(-3)}$ instead of $g^{(-2)}$; etc.  
There is no need of leaving $g^{(-3)}_0(t,j)$ free, as this is the last step: 
so we fix $g^{(-3)}_0(t,j) := 0$, and $g^{(-3)} := g^{(-3)}_1$. 
Regarding parities, we obtain coefficients of the form
\begin{align*}
r^{(-2)}, b^{(-2)} & = a + i \sgn(j) b, \quad 
a \in Y, \quad 
b = \odd(t), \odd(x),
\\
g^{(-2)} & = a + i \sgn(j) b, \quad 
a \in X, \quad 
b = \even(t), \odd(x),
\\
g^{(-3)} & = \sgn(j) a + i b, \quad 
a = \odd(t), \odd(x), \quad 
b \in Y,
\end{align*}
where $a,b$ denote (different) real-valued functions of $(t,x)$, independent of $j$.

We have found $A, \mD$ such that 
$\mR_6 := \mL_5 A - A \mD = O(\Dx^{-3/2})$. 
Let $\tilde A := \pp A \pp$, which is invertible by the same argument as for the first transformation $\mB$. 
By the equalities $\mL_5 A = A \mD + \mR_6$ and $I = \pp + \ff$ we get
\begin{equation} \label{def mD}
\tilde A^{-1} (\tilde \mL_5 + \tilde \mR_5) \tilde A
= \tilde \mD + \tilde \mR_6,
\qquad
\tilde \mD := \pp \mD \pp = \pp \mD = \mD \pp,
\end{equation}
where the remainder 
\begin{equation} \label{def mR6}
\tilde \mR_6 := \tilde A^{-1} \{ \tilde \mR_5 \tilde A
+ \pp \mR_6 \pp - \pp \mL_5 \ff A \pp \}
\end{equation}
has order $O(|D_y|^{-3/2})$ and size $O(\e)$.
More precisely, 

\begin{lemma} \label{lemma:estimate mR6} 
There exist constants $\s, C > 0$ such that, if  
$u = \bar u_\e + \tilde u$, with $\| \tilde u \|_{s_0 + \s} \leq C \e^{2+\d}$, 
$s_0 \geq 5$, $\d>0$, then 
$\tilde \mR_6 : W_2 \to R_2$ defined in \eqref{def mR6} satisfies, for all $s \geq s_0$, 
\begin{equation} \label{est mR6}
\| \tilde \mR_6 |D_x|^{3/2} h \|_s \leq_s \e \| h \|_s + \| \tilde u \|_{s+\s} \| h \|_{s_0}.
\end{equation}
\end{lemma}

\section{Inversion of the restricted linearized operator} 
\label{sec:inv lin op}

Recall that our goal is the restricted inversion problem in Remark \ref{rem:reduced lin to invert}.
The diagonal operator $\tilde \mD := \mD \pp : W_2 \to R_2$, where $\mD$ is defined in \eqref{mD},
has purely imaginary eigenvalues 
\begin{equation} \label{eigenvalues mu j} 
\ba
&\tilde \mD [e^{ilt} \cos(jx)] = i (\om l + \mu_j) \, e^{ilt} \cos(jx), \\
&\mu_{j} := \lm_3  (j + \kappa j^3)^{1/2} + \lm_1 j^{1/2} + \lm_{-1} j^{-1/2} \in \R, 
\ea
\end{equation}
with $l \in \Z$, $j \geq 2$, where we denote $\lm_3 := \sqrt{\mult}$. 
Let $\g = \e^{5/6} > 0$, $\t := 3/2$, and assume that $\om$ satisfies the first-order Melnikov non-resonance condition 
\begin{equation} \label{non-res}
| \om l + \mu_j | \geq \frac{\g}{|j|^\t} \quad \forall l \in \Z, \ j \geq 2,
\end{equation}
where $\t = 3/2$. 
Then $\tilde \mD$ has inverse 
\[
\tilde \mD^{-1} h(t,x) 
:= \sum_{l \in \Z, \, j \geq 2} \frac{h_{lj} e^{ilt} }{i(\om l + \mu_j)} \, \cos(jx), 
\qquad \tilde \mD^{-1} : R_2 \to W_2,
\]
of order $O(|D_x|^{3/2})$ and size $1/\g$, namely 
\begin{equation} \label{estimate Lambda inv mD}
\| |D_x|^{-3/2} \tilde \mD^{-1} h \|_s \leq \g^{-1} \, \| h \|_{s} , 
\end{equation}
because $|\om l + \mu_j| j^{3/2} \geq \g$ for all $l \in \Z, j \geq 2$.

By \eqref{estimate Lambda inv mD} and \eqref{est mR6}, 
writing explicitly the constant $C(s)$, we have
\begin{align*}
\| \tilde \mR_6 \tilde \mD^{-1} h \|_s 
& = \| (\tilde \mR_6 |D_x|^{3/2}) (|D_x|^{-3/2} \tilde \mD^{-1}) h \|_s \\
&\leq C(s) \g^{-1} ( \e \| h \|_s + \| \tilde u \|_{s + \s} \| h \|_{s_0} )
\\ & 
\leq \frac12 \| h \|_s + C(s) \g^{-1} \| \tilde u \|_{s + \s} \| h \|_{s_0} ,
\end{align*} 
where the last inequality holds for $\e$ sufficiently small, namely $C(s) \e^{1/6} \leq 1/2$.
Therefore, by tame Neumann series (see e.g. Lemma B.2 in \cite{Baldi-Benj-Ono}, Appendix B), 
$(I_{R_2} + \tilde \mR_6 \tilde \mD^{-1})$ is invertible on $H^{s} \cap R_2$, 
where $I_{R_2}$ is the identity map of $R_2$, and
\[
\| (I_{R_2} + \tilde \mR_6 \tilde \mD^{-1})^{\pm 1} h \|_s 
\leq 2 \| h \|_s + 4 C(s) \g^{-1} \| \tilde u \|_{s + \s} \| h \|_{s_0}.
\]
As a consequence, $\tilde \mD + \tilde \mR_6 = (I_{R_2} + \tilde \mR_6 \tilde \mD^{-1}) \tilde \mD$ $: W_2 \to R_2$ is invertible, with 
\begin{equation} \label{tame estimate mL6}
\| (\tilde \mD + \tilde \mR_6)^{-1} h \|_s 
\leq_s \g^{-1} ( \| h \|_{s+\t} + \g^{-1} \| \tilde u \|_{s + \t + \s} \| h \|_{s_0} ). 
\end{equation}

In sections \ref{sec:linearized operator}-\ref{sec:semi-FIO conj} we have conjugated 
the restricted operator $\mL_2^2 + \mR$ (see Remark \ref{rem:reduced lin to invert}) to $\tilde \mD + \tilde \mR_6$,
\begin{equation}  \label{conj tot}
\ba
&\mL_2^2 + \mR = \Phi_1 (\tilde \mD + \tilde \mR_6) \Phi_2^{-1}, \\
&\Phi_1 := \tilde \mZ \tilde \mB \tilde \mA \tilde P \tilde S \tilde M \tilde A,  \quad
\Phi_2 := \tilde \mZ \tilde \mB \tilde \mA \tilde Q \tilde S \tilde M \tilde A.
\ea
\end{equation}
All these operators have been estimated in the previous sections (they are all bounded, except $\tilde S$, which is of order 1/2). 
Thus, by composition, we obtain the following result.

\begin{theorem}[Inversion of the restricted linearized operator] \label{thm:restricted inverse}
There are $\s, C > 0$ with the following property. 
Let $u = \bar u_\e + \tilde u$, 
with $\tilde u \in H^{s + \s}$, $\| \tilde u \|_{s_0 + \s} < C \e^{2+\d}$, 
$5 \leq s_0 \leq s$, $\d>0$, $\e < \e_0(s)$ for some $\e_0(s)$ depending on~$s$. 
Assume that the first Melnikov conditions \eqref{non-res} hold. 
Then the operator $\mL_2^2 + \mR : W_2 \to R_2$ is invertible, with 
\begin{equation}  \label{tame inverse linearized}
\| (\mL_2^2 + \mR)^{-1} h \|_s 
\leq_s \g^{-1} ( \| h \|_{s + 2} + \g^{-1} \| \tilde u \|_{s + \s} \| h \|_{s_0} ).
\end{equation}
\end{theorem}

By Theorem \ref{thm:restricted inverse} and Lemmata \ref{lemma:inv bif} and \ref{lemma:inv bif 2}, 
we deduce (with a larger $\s$ if necessary)

\begin{corollary}[Inversion of the linearized operator] \label{cor:entire inverse}
Assume the hypotheses of Theorem \ref{thm:restricted inverse}. 
Then the linearized operator $F'(u) : X \times Y \to Y \times X$ is invertible, with 
\begin{equation*} 
\| F'(u)^{-1} h \|_s 
\leq_s \e^{-2} ( \| h \|_{s + 2} + \g^{-1} \| \tilde u \|_{s+\s} \| h \|_{s_0} ).
\end{equation*}
\end{corollary}

For $u,\om,h$ depending on the parameters $(\e,\xi)$, 
using Corollary \ref{cor:entire inverse} we prove a tame estimate for $F'(u)^{-1}$ also in Lipschitz norms
\eqref{def norm Lipe} (with a larger $\s \geq 6$).

\begin{lemma}[Inversion in Lipschitz norms] 
\label{lemma:Lip inv}
Let $5 \leq s_0 \leq s$, $u = \bar u_\e + \tilde u$, where $\bar u_\e$ is defined in \eqref{approx sol}
and $\tilde u = \tilde u(\e,\xi)$ is defined for parameters $\e \in (0,\e_0)$, $\xi \in \mG$, 
with $\e_0 = \e_0(s) < 1$, $\mG \subseteq [1,2]$. 
Let $\| \tilde u \|_{s_0 + \s}^\Lipe < C \e^{2+\d}$, $\d>0$. 
Let $\om$ be given by \eqref{freq-ampl}.
Assume that \eqref{non-res} hold for all $\xi \in \mG$. 
Then 
\begin{equation}  \label{Lip inv}
\| F'(u)^{-1} h \|_s^\Lipe 
\leq_s \e^{-2} ( \| h \|_{s + 6}^\Lipe + \e^{-2} \| \tilde u \|_{s + \s}^\Lipe \| h \|_{s_0}^\Lipe ).
\end{equation}
\end{lemma}

\begin{proof}
Denote $A_i$ the linearized operator $F'(u)$ when $u = u_i := u(\xi_i)$, $\om = \om(\xi_i)$, $i=1,2$, 
and denote $h_i := h(\xi_i)$. Thus
\[
A_1^{-1} h_1 - A_2^{-1} h_2 
= A_1^{-1} [h_1 - h_2]
+ A_1^{-1} (A_2 - A_1) A_2^{-1} h_2.
\]
$A_1^{-1}, A_2^{-1}$ satisfy \eqref{tame inverse linearized}. 
The difference $A_2 - A_1$ is
\[
\ba
(A_2 - A_1)h &= \{ \bar \om_2 \e^2 (\xi_2 - \xi_1) + \bar \om_3 (\xi_2^{3/2} - \xi_1^{3/2}) \} \pa_t h\\
&\quad+ \int_0^1 \mN''(u_1 + \th (u_2 - u_1)) [u_2 - u_1, h] \, d \th,
\ea
\]
where $\mN(u)$ is the nonlinear part of $F(u)$.
Since $\| \bar u_\e \|_s^\Lipe \leq_s \e$, we get
\begin{align*}
\frac{ \| (A_2 - A_1)h \|_s }{| \xi_2 - \xi_1| } 
& \leq_s \e \| h \|_{s+2} 
+ ( \| \tilde u \|_{s+2}^{\lip} + \| \tilde u \|_{s_0+2}^{\lip} \| \tilde u \|_{s+2}^{\sup} ) \| h \|_{s_0 + 2}
\\ & 
\leq_s \e \| h \|_{s+2} + \e^{-1} \| \tilde u \|_{s+2}^\Lipe \| h \|_{s_0 + 2}\\
&\leq_s \e \| h \|_{s+2} + \e^{-1} \| \tilde u \|_{s+4}^\Lipe \| h \|_{s_0},
\end{align*}
and the thesis follows by composition.
\end{proof}

\subsection{Dependence of the eigenvalues on the parameters}

The constants $\lm_3, \lm_1, \lm_{-1}$ in the formula for the eigenvalues $\mu_j$ in \eqref{eigenvalues mu j} 
depend on the point $u = (\eta,\psi)$ where the linearization $F'(u)$ takes place, and on $\om$. 
In particular, $\lm_3$ depends only on $u$ by formula \eqref{mu}, namely
\begin{equation} \label{lm3 formula}
\lm_3 = \lm_3(u) = \sqrt{\mult} = \meanT \Big( \meanT \sqrt{1 + \eta_x^2}\, dx \Big)^{-3/2} \, dt .
\end{equation}
$\lm_1$ depends on $u,\om$ by the following formula (which is obtained using \eqref{formula lm 1} and going back with the changes of variable $\mA, B$ of section \ref{sec:changes} in the integral)
\begin{equation} \label{lm1 formula}
\lm_1 = \lm_1(u,\om) 
= - \frac{1}{8 \pi^2 \sqrt{\kappa} \lm_3(u)} 
\int_{\T^2} \frac{1 + \b_x}{1 + \a'(t)} [ \om \b_t + V (1+\b_x) ]^2 dt dx,
\end{equation}
where $V$ is defined in \eqref{def B V}, $\a,\b$ are defined in section \ref{sec:changes} and satisfy \eqref{formula beta alpha}.
To write an explicit formula for $\lm_{-1}$ is possible, but more involved and not necessary for our purposes. 
Anyway, $\lm_{-1}$ is the space-time-average of a polynomial function of $a_{14}, a_{29}, a_{31}, a_{32}$ (and their derivatives), $\om, \om^{-1}, \lm_3, \lm_3^{-1}$ with real coefficients. 
Hence this is a polynomial function of $a_7,\ldots,a_{13}$ (and their derivatives), $\om, \om^{-1}, \lm_3, \lm_3^{-1}$ and, going back with the changes of variable $\mA, B$ of section \ref{sec:changes} in the integral, 
one obtains for $\lm_{-1}$ a similar result as for $\lm_1$. Thus we have: 

\begin{lemma}  \label{lemma:eigenvalues der u}
$\lm_3(u), \lm_1(u,\om), \lm_{-1}(u,\om)$ are $C^2$ functions of $(u,\om)$ 
in the domain $\| u \|_{\s_0} < \d$, $ \om \in [ \frac12 \bar\om , \frac32 \bar\om]$, 
where $\s_0 > 0$ is a universal constant and $\d > 0$ depends only on $\kappa$. 

As a consequence, if $u_1, u_2$ are in the ball $\| u_i \|_{\s_0} < \d$, $i=1,2$, 
then $|\lm_k(u_1) - \lm_k(u_2)| \leq C \| u_1 - u_2 \|_{\s_0}$, $k = 3, 1, -1$, 
where $C>0$ depends only on $\kappa$.
\end{lemma} 

The number $\s_0$ in Lemma \ref{lemma:eigenvalues der u} can be explicitly computed by counting how many derivatives of $u$ are involved in the transformation procedure of 
sections \ref{sec:linearized operator}-\ref{sec:semi-FIO conj}.

For $(u,\om)$ depending on the parameters $(\e,\xi)$, we deduce the following expansion for 
$\lm_3, \lm_1, \lm_{-1}$:

\begin{lemma}  \label{lemma:eigenvalues parameter dep}
Let $u = \bar u_\e + \tilde u$, where $\bar u_\e$ is defined in \eqref{approx sol}
and $\tilde u = \tilde u(\e,\xi)$ is defined for parameters $\e \in (0,\e_0)$, $\xi \in \mG$, 
with $\e_0 < 1$, $\mG \subseteq [1,2]$. 
Let $\| \tilde u \|_{\s_0}^\Lipe < C \e^{2+\d}$, $\d>0$.
Let $\om$ be given by \eqref{freq-ampl}. 
Then $\lm_3, \lm_1, \lm_{-1}$ depend on $\xi$ in a Lipschitz way, with
\begin{equation}  \label{lm3 lm1 lm-1}
\lm_3 = 1 - \frac{3}{16} \, \e^2 \xi + r_3, \quad
| r_3 |^\Lipe \leq C \e^3, \quad 
| \lm_1 |^\Lipe + | \lm_{-1} |^\Lipe \leq C \e^2,
\end{equation}
where $r_3 := \lm_3 - 1 + \tfrac{3}{16} \e^2 \xi$.
\end{lemma}

\begin{proof} By \eqref{approx sol},\eqref{def kernel}, $\eta = \e \sqrt{\xi} \cos(t) \cos(x) + O(\e^2)$.
Therefore the inequality for $r_3$ follows easily from formula \eqref{lm3 formula}.
By Lemma \ref{lemma:eigenvalues der u}, $\lm_1, \lm_{-1}$ are functions of $(u,\om)$ of class $C^2$.
Since $u = \e \sqrt{\xi} \, v_0 + O(\e^2)$ ($v_0$ is defined in \eqref{def kernel}) and 
$\om = \bar \om + O(\e^2)$, one has
\[
|\lm_i(u,\om) - \lm_i(\e \sqrt{\xi} \, v_0, \bar \om)|^\Lipe \leq C \e^2,
\quad i=1,-1
\]
by the mean value theorem and standard analysis for composition of functions.
Thus the inequalities for $\lm_1, \lm_{-1}$ in \eqref{lm3 lm1 lm-1} hold if
\begin{equation} \label{lmi start}
\lm_i(\e \sqrt{\xi} \, v_0, \bar \om) = O(\e^2), \quad i=1,-1. 
\end{equation}
To prove \eqref{lmi start}, let $u = \e \sqrt{\xi} \, v_0$, $\om = \bar\om$. 
By \eqref{def B V}, $V = \e \sqrt{\xi} \bar\om \sin(t) \sin(x) + O(\e^2)$. 
By \eqref{formula beta alpha}, $\a, \b = O(\e^2)$. Therefore, by \eqref{lm1 formula}, 
we get \eqref{lmi start} for $\lm_1$. 

To prove \eqref{lmi start} for $\lm_{-1}$, we compute the order $\e$ of 
almost all the coefficients in sections \ref{sec:linearized operator}-\ref{sec:semi-FIO conj}, namely: 
$a,c,B,V$ in section \ref{sec:linearized operator}; 
$\b, \a, a_1, \ldots, a_{19}$ in section \ref{sec:changes};
$a_{25}, a_{27}, a_{28}$ in section \ref{sec:symm top};
$a_{29}, \ldots, a_{32}, v_2, v_4, g_3, g_5$ in section \ref{sec:symm lower};
and, in section \ref{sec:semi-FIO conj}, $\b, v^{(k)}_m$ (with $k=1,0,-1,-2$; $m=2,1,0$),
$p^{(0)}, p^{(-1)}, p^{(-2)}$, $f$, 
$b^{(0)}, b^{(-1)}$, 
$g^{(-1)}, g^{(-2)}$, and finally $r^{(-1)}$, which gives $\lm_{-1}$ by \eqref{lambda -1}.
All these coefficients are functions of the form 
\[
c_0 + c_1 \e \sqrt{\xi} \, \psi(t,x) + O(\e^2),
\]
where $c_0, c_1$ are real constants, and 
$\psi = \cos(t) \cos (x)$, or $\psi = \cos (t) \sin (x)$, 
or $\psi = \sin(t) \cos(x)$, or $\psi = \sin(t) \sin(x)$.
We calculate that the term of order $O(1)$ in $\lm_{-1}$ is zero, while its term of order $O(\e)$ is automatically zero because $\psi$ has zero mean. The proof of \eqref{lmi start} is complete.
\end{proof}

By \eqref{lm3 lm1 lm-1} we deduce that the eigenvalues $\mu_j$ in \eqref{eigenvalues mu j} satisfy
\begin{equation} \label{quasi-linear perturbation of eigenvalues}
|\mu_j - (j + \kappa j^3)^{1/2}| \leq C \e^2 j^{3/2}.
\end{equation}

\section{Nash-Moser iteration and measure of parameter set} \label{sec:NM}

Consider the finite-dimensional subspaces $E_n := \{ u : u = \Pi_n u \}$, $n \geq 0$, where 
\begin{equation}  \label{def Nn}
N_n := N_0^{\chi^n} = (N_0)^{\chi^n}, 
\quad \chi := \frac32, \quad 
N_0 := \e^{-\rho_0} := \e^{- 1/\rho_1}, \quad 
\rho_0 := \frac{1}{\rho_1} > 0,
\end{equation}
and $ \Pi_n $ are the projectors (Fourier truncation)
\begin{align*}
\Pi_n u(t,x) := \sum_{|l| + |j| < N_n} \hat u_{lj} e^{i(lt+jx)} \quad  
& \text{where} \ u(t,x) = \sum_{l,j \in \Z} \hat u_{lj} e^{i(lt+jx)}.
\end{align*}
We denote $ \Pi_n^\bot := I - \Pi_n $.  
The classical smoothing properties also hold for the Lipschitz norms \eqref{def norm Lipe}: 
for all $\alpha, \b, s \geq 0$, 
\begin{equation}\label{smoothing}
\| \Pi_n u \|_{s + \alpha}^\Lipe \leq N_n^{\alpha} \| u \|_{s}^\Lipe; 
\quad
\| \Pi_n^\bot u \|_s^\Lipe \leq N_n^{-\b} \| u \|_{s + \b}^\Lipe. 
\end{equation}
Define the following constants: 
\begin{equation} \label{NM const}
\ba
&\d = \frac12, \qquad 
\a_0 = 6+\s, \qquad
\a_1 = \rho_1 = 9 \a_0, \qquad 
\\
&\kappa_1 = 3(\s + 4 + 2\rho_1) + 1, \qquad 
\beta_1 = 3 + \s + \a_1 + \frac23 \kappa_1 + 4 \rho_1. 
\ea
\end{equation} 
All these constants depend only on $\s$, 
where $\s \geq 6$ is the loss of regularity in \eqref{Lip inv}.

\begin{theorem}[Nash-Moser iteration] \label{thm:NM} 
Let $s_0 \geq 5$. 
There exists $ \e_0 > 0$ such that, if $\e \in (0,\e_0]$, then, for all $ n \geq 0 $: 

\begin{itemize}
\item[$ (\mP 1)_{n}$] \emph{(Convergent sequence).}
There exists a function 
$u_n = \bar u_\e + \tilde u_n 
: \mG_n \subseteq [1,2] \to E_n$, 
$\xi \mapsto u_n(\xi) = (\eta_n(\xi), \psi_n(\xi))$,  
where $\bar u_\e$ is defined in \eqref{approx sol}, 
and $u_0 := \bar u_\e$, $\tilde u_0 = 0$,
such that
\begin{equation}\label{un etc}
\| u_n \|_{s_0 + \s }^{\Lipe} \leq C_* \e, \quad 
\| \tilde u_n \|_{s_0 + \s }^{\Lipe} \leq C_* \e^{2+\d}.
\end{equation}
The function $ u_n $ has parity $u_n \in X \times Y$. 
The sets $\mG_n $ are defined inductively by: $\mG_0 := [1,2]$, 
\begin{equation}\label{def:Gn+1}
\mG_{n+1} :=  
\Big\{ \xi \in \mG_n \, : \, |\om l + \mu_j(u_{n})| > \frac{\gamma}{j^\t} \quad \forall l \in \Z, \ j \geq 2 \Big\},
\end{equation}
where $\mu_j(u_n) = \mu_j(\xi, u_n(\xi))$ are defined in \eqref{eigenvalues mu j}.
The difference $ h_n := u_{n} - u_{n-1}
= \tilde u_n - \tilde u_{n-1}$, 
$n \geq 1$ (with $h_0 := 0$) is defined on $\mG_n$, and 
\begin{equation}  \label{Hn}
\ba
&\| h_n \|_{s_0 + \s}^{\Lipe} 
\leq C_* \e^{2+\d} N_{n}^{-\a_0} 
= C_* \e^{2+\d+\frac{1}{9} \chi^n}, \\
&\| F(u_n) \|_{s_0}^{\Lipe} 
\leq C_* \e^4 N_{n}^{- \alpha_1}
= C_* \e^{4+\chi^n}.
\ea
\end{equation}

\item[$ (\mP 2)_{n}$] \emph{(High norms).} 
$\| \tilde u_n \|_{s_0 + \beta_1}^\Lipe \leq C_* N_{n}^{\kappa_1}$ and 
$\| F(u_n) \|_{s_0 + \beta_1}^\Lipe \leq C_* N_{n}^{\kappa_1}$.

\item[$ (\mP 3)_{n}$] \emph{(Measure).} 
For all $n \geq 0$, the Lebesgue measure of the set $ \mG_n \setminus \mG_{n+1} $ satisfies
\begin{equation}  \label{Gn - Gn+1}
| \mG_n \setminus \mG_{n+1} | \leq C_* \e^{1/18} \, 2^{-n}.
\end{equation}
\end{itemize}
\end{theorem}

In sections \ref{sec:proof of NM iteration}, \ref{sec:measure} we prove Theorem \ref{thm:NM}.

\subsection{Proof of the nonlinear iteration} \label{sec:proof of NM iteration}

In this section we prove $({\mathcal P}1,2)_n$ by induction. 
The proof of $({\mathcal P}3)_n$ is in section \ref{sec:measure}. 
To shorten the notation, in this section we use the following abbreviations: $| \ |_s := \| \ \|_s^\Lipe$, 
$F_n := F(u_n)$, $L_n := F'(u_n)$, $s_1 := s_0 + \b_1$.

\medskip

\emph{Proof of} $({\mathcal P}1,2)_0$. 
$u_0 = \bar u_\e \in E_0$ if $N_0 > 5$, i.e. for $\e$ sufficiently small. 
By Lemma \ref{lemma:F(u0)}, 
the bounds \eqref{un etc}, \eqref{Hn} hold. 
To satisfy also $({\mathcal P}2)_0$, take $C_* = C_* (s_1) $ large enough.

\medskip

\emph{Assume that $({\mathcal P}1,2)_n$ hold for some $n \geq 0$, and prove $({\mathcal P}1,2)_{n+1}$.}
By \eqref{un etc} and Corollary \ref{cor:entire inverse}, $F'(u_n)$ is invertible 
for all $\xi \in \mG_{n+1}$, and the inverse satisfies \eqref{tame inverse linearized}, \eqref{Lip inv}. 
For $ \xi \in \mG_{n+1} $ we define 
\begin{equation}\label{def hn+1}
u_{n+1} := u_n + h_{n+1}, \quad 
h_{n+1} := - \Pi_{n+1} F'(u_n)^{-1} F(u_n).
\end{equation}
Let 
\begin{equation}\label{def Q}
Q(u_n, h) := F(u_n + h) - F(u_n) - F'(u_n) h, \quad 
Q_n := Q(u_n, h_{n+1}).
\end{equation}
By the definitions \eqref{def hn+1},\eqref{def Q}, and splitting $\Pi_{n+1} = I - \Pi_{n+1}^\bot$, 
\begin{equation}
\ba
&F(u_{n+1}) = F(u_n) + F'(u_n) h_{n+1} + Q_n 
= R_n + Q_n, 
\\
&R_n := F'(u_n) \Pi_{n+1}^\bot F'(u_n)^{-1} F(u_n).
\label{relazione algebrica induttiva}
\ea
\end{equation}

\emph{Estimate of $ Q_n $}. 
For all $h \in E_{n+1}$, by \eqref{smoothing} with $\a = 2$,
\begin{equation}\label{stima Q}
\ba
| Q(u_n, h)|_s 
&\leq_s | h |_{s+2} | h |_{s_0 + 2}
+ | \tilde u_n |_{s+2} | h |_{s_0 + 2}^2, 
\\
| Q(u_n, h) |_{s_0} &\leq_{s_0} N_{n+1}^4 | h |_{s_0}^2. 
\ea 
\end{equation}
By the definition \eqref{def hn+1} of $h_{n+1}$, \eqref{Lip inv} and \eqref{smoothing} with $\a = \s$, $\a = 6$,
\begin{equation} \label{H n+1}
\ba
| h_{n+1} |_{s_1} 
&\leq_{s_1}  \e^{-2} N_{n+1}^{\s} 
\big( | F_n |_{s_1} + \e^{-2} | \tilde u_n |_{s_1} | F_n |_{s_0} \big),
\\ 
| h_{n+1} |_{s_0} 
&\leq_{s_0} \e^{-2} N_{n+1}^{6} | F_n |_{s_0}.
\ea
\end{equation}
Then $Q_n$ in \eqref{def Q} satisfies 
\begin{equation}  \label{Qn}
\ba
| Q_n |_{s_1} 
&\leq_{s_1} \e^{-4} N_{n+1}^{\s + 10} | F_n |_{s_0}
( | F_n |_{s_1} + \e^{-2} | \tilde u_n |_{s_1} | F_n |_{s_0} ),
\\
| Q_n |_{s_0} 
&\leq_{s_0} \e^{-4}  N_{n+1}^{16} | F_n |_{s_0}^2. 
\ea
\end{equation}

\emph{Estimate of $ R_n $.} 
The linearized operator $L_n = F'(u_n)$ satisfies
\begin{equation}  \label{tame lin}
| L_n h |_s \leq_s | h |_{s+2} + | \tilde u_n |_{s+2} | h |_{s_0 + 2}, 
\quad
| L_n h |_{s_0} \leq_{s_0} | h |_{s_0+2}
\quad \forall h.
\end{equation}
Then, by \eqref{smoothing} with $\b = \b_1 - 2 - \s$ and \eqref{Lip inv},
\begin{align} 
| R_n |_{s_0} 
& \leq_{s_0} | \Pi_{n+1}^\bot L_n^{-1} F_n |_{s_0 + 2}
\leq_{s_0} N_{n+1}^{-\b} | L_n^{-1} F_n |_{s_0 + 2 + \b}
\notag \\ & 
\leq_{s_1} \e^{-2} N_{n+1}^{-\b_1 + 2 +\s} (| F_n |_{s_1}
+ \e^{-2} | \tilde u_n |_{s_1} | F_n |_{s_0} ).
\label{Rn low norm} 
\end{align}
For the high norm, since $\Pi_{n+1}^\bot = I - \Pi_{n+1}$, 
we split 
$R_n = F_n - L_n \Pi_{n+1} L_n^{-1} F_n $. 
By \eqref{tame lin} and \eqref{smoothing} with $\a = 2 + \s$, $\a = 2$,
\begin{align}
| L_n \Pi_{n+1} L_n^{-1} F_n |_{s_1} 
& \leq_{s_1} 
| \Pi_{n+1} L_n^{-1} F_n |_{s_1 + 2}
+ | \tilde u_n |_{s_1 + 2} | \Pi_{n+1} L_n^{-1} F_n |_{s_0 + 2}
\notag \\ & 
\leq_{s_1}  N_{n+1}^{2+\s} | L_n^{-1} F_n |_{s_1 - \s}
+ N_n^2 | \tilde u_n |_{s_1} N_{n+1}^{2} | L_n^{-1} F_n |_{s_0}
\notag \\ & 
\leq_{s_1}  N_{n+1}^{2+\s} \e^{-2} ( | F_n |_{s_1 - \s + 6}
+ \e^{-2} | \tilde u_n |_{s_1} | F_n |_{s_0}) \notag
\\ & \quad \quad 
+ N_{n+1}^4 | \tilde u_n |_{s_1} \e^{-2} | F_n |_{s_0 + 2}
\notag \\ & 
\leq_{s_1}  \e^{-2} N_{n+1}^{4+\s} ( | F_n |_{s_1} + \e^{-2} | \tilde u_n |_{s_1} | F_n |_{s_0}).
\label{Rn high norm}
\end{align}
In the last inequality we have used the interpolation estimate
$$
| \tilde u_n |_{s_1} | F_n |_{s_0 + 2}
\leq | \tilde u_n |_{s_1 + 2} | F_n |_{s_0}
+ | \tilde u_n |_{s_0 + 2} | F_n |_{s_1}
$$
and then \eqref{smoothing} with $\a = 2$ for $\tilde u_n \in E_n$.

\emph{Estimate of $F_{n+1}$}. 
Since $F_{n+1} = R_n + Q_n$, by \eqref{Rn high norm},\eqref{Qn},
\begin{align}
| F_{n+1} |_{s_1} 
& \leq_{s_1} \e^{-2} N_{n+1}^{\s+4} \{ 1 + \e^{-2} N_{n+1}^6 |F_n|_{s_0} \}
\big( | F_n |_{s_1} + \e^{-2} | \tilde u_n |_{s_1} | F_n |_{s_0} \big)
\notag \\ 
& \leq_{s_1} \e^{-2} N_{n+1}^{\s+4} ( | F_n |_{s_1} + \e^{-2} | \tilde u_n |_{s_1} | F_n |_{s_0} ).
\label{F n+1 high norm}
\end{align}
Note that $\e^{-2} N_{n+1}^6 |F_n|_{s_0} \leq 1$ for $\e$ sufficiently small, 
because $\a_1 > 6\chi$. 
Also, by \eqref{Rn low norm}, \eqref{Qn},
\begin{equation}  \label{F n+1 low norm}
| F_{n+1} |_{s_0} 
\leq_{s_1} \e^{-2} N_{n+1}^{-\b_1 + 2 + \s} 
(| F_n |_{s_1} + \e^{-2} | \tilde u_n |_{s_1} | F_n |_{s_0} )
+ \e^{-4} N_{n+1}^{16} | F_n |_{s_0}^2 .
\end{equation}

\emph{Estimate of $ \tilde u_{n+1} $}.  
By \eqref{H n+1}, and using that $ \tilde u_{n+1} = \tilde u_n + h_{n+1} $, 
$\e^{-2} | F_n |_{s_0} \leq 1$, we get
\begin{equation}\label{U n+1 alta}
| \tilde u_{n+1} |_{s_1} 
\leq_{s_1} \e^{-2} N_{n+1}^{\s} ( | \tilde u_n |_{s_1} + | F_n |_{s_1} ).
\end{equation}
Let $B_n := | \tilde u_n |_{s_1} + | F_n |_{s_1}$. 
From \eqref{F n+1 high norm},\eqref{U n+1 alta}, using that $\e^{-2} | F_n |_{s_0} \leq 1$, we get
\[
B_{n+1} \leq C \e^{-2} N_{n+1}^{\s+4} B_n 
\leq C N_{n+1}^{\s+4+2\rho_1} B_n \quad \forall n \geq 0,
\] 
for some $C = C(s_1)$ independent of $n$, because $\e^{-2} = N_0^{2 \rho_1} < N_{n+1}^{2 \rho_1}$.  
Hence, by induction, $B_n \leq C' N_n^{\kappa_1}$ for all $n \geq 0$, for some $C'$, 
because $\kappa_1 > 3( \s + 4 + 2\rho_1)$. 
Thus $({\mathcal P}2)_{n + 1}$ is proved.

\emph{Proof of $({\mathcal P}1)_{n + 1}$}. 
Using \eqref{F n+1 low norm}, \eqref{Hn}, $(\mP 2)_n$, 
\begin{equation}  \label{last 1}
| F_{n+1} |_{s_0} 
\leq C_1 \{ \e^{-2} N_{n+1}^{- \beta_1 + 2 + \s} C_* N_{n}^{\kappa_1} 
+ \e^{-4} N_{n+1}^{16} ( C_* \e^4 N_{n}^{-\alpha_1} )^2 \}, 
\end{equation}
for some $C_1 = C_1(s_1)$.
The right-hand side term of \eqref{last 1} is $ \leq C_* \e^4 N_{n+1}^{-\alpha_1} $ if
\begin{equation} \label{provided 2}
2 C_1 \e^{-6} N_{n+1}^{- \beta_1 + 2 +\s + \a_1} N_{n}^{\kappa_1} \leq 1,
\quad
2 C_1 C_* N_{n+1}^{16+\a_1} N_{n}^{-2\alpha_1} \leq 1
\quad \forall n \geq 0.
\end{equation}
Recalling \eqref{def Nn},\eqref{NM const}, 
the inequalities in \eqref{provided 2} hold taking $\e$ small enough.
This gives $| F_{n+1} |_{s_0} \leq C_* \e^4 N_{n+1}^{-\alpha_1}$. 
The bound \eqref{Hn} for $h_{n+1}$ follows by \eqref{smoothing} (with $\a = \s$), 
\eqref{H n+1}, and the bound \eqref{Hn} for $F_n$, 
using \eqref{NM const}, and taking $\e$ small enough.

Finally, using \eqref{Hn}, the bound \eqref{un etc} for $\tilde u_{n+1}$ holds because 
$\tilde u_{n+1} = h_1 + \ldots + h_{n+1}$ and $\sum_{k=1}^\infty N_k^{-\a_0} < 1$ for $ \e $ small. 
The proof of $(\mP 1,2)_n$ is concluded.

\subsection{Measure estimates} \label{sec:measure}

In this section we prove $({\mathcal P}3)_n$ for all $n \geq 0$. 
Let us estimate $[1,2] \setminus \mG_1$ first. For $l \in \Z$, $j \geq 2$, define
\[
\mA_{lj} := \{ \xi \in [1,2] : | \om l + \mu_j | < \g j^{-\t} \}
\]
where $\g := \e^{5/6} $ and the eigenvalues $\mu_j = \mu_j(u_0)$.
If $\mA_{lj} \neq \emptyset$, then there exists $\xi \in [1,2]$ for which
\[
- \frac{\mu_j}{\om} - \frac{\g}{\om j^\t} < l < - \frac{\mu_j}{\om} + \frac{\g}{\om j^\t}
\] 
(where $\mu_j, \om$ depend on $\xi$). 
By the inequality $|\om^{-1} - \bar\om^{-1}| \leq C \e^2$, 
and using \eqref{quasi-linear perturbation of eigenvalues}, 
we deduce that $\mu_j \om^{-1} = (j + \kappa j^3)^{1/2} \bar\om^{-1} + O(\e^2 j^{3/2})$, and 
\begin{equation} \label{l interval 0}
- \frac{(j + \kappa j^3)^{1/2} }{\bar \om} - C \e^2 j^{3/2} - \frac{2 \g}{\bar\om j^\t}
< l < - \frac{(j + \kappa j^3)^{1/2} }{\bar \om} + C \e^2 j^{3/2} + \frac{2 \g}{\bar\om j^\t} 
\end{equation}
because $\om > \bar\om / 2$ for $\e$ sufficiently small. 
Note that all the terms in the inequality \eqref{l interval 0} are independent of $\xi$. 
As a consequence, for each fixed $j \geq 2$, 
\begin{equation} \label{number of l}
\sharp \{ l \in \Z : \mA_{lj} \neq \emptyset \} 
< C \e^2 j^{3/2} + 2
\end{equation}
for $\e$ sufficiently small, simply because the number of integers in an interval $(a,b)$ is $< b-a+1$.

Now we study the variation of the eigenvalues with respect to the parameter $\xi$. 
By \eqref{l interval 0}, 
\begin{equation} \label{l like j 32}
l = - \frac{(j + \kappa j^3)^{1/2} }{\bar \om} + O(\e^2 j^{3/2}) + O(\g j^{-\t}).
\end{equation}
Let $f_{lj}(\xi) := \om l + \mu_j$, where the dependence on $\xi$ of the eigenvalue is put into evidence.  Replacing $l$ by \eqref{l like j 32}, 
and using \eqref{eigenvalues mu j}, \eqref{lm3 lm1 lm-1},
\begin{align*}
\frac{f_{lj} (\xi_2) - f_{lj} (\xi_1)}{\xi_2 - \xi_1} \, 
& = \Big( \e^2 \bar\om_2 + \e^3 \bar\om_3 \frac{ \xi_2^{3/2} - \xi_1^{3/2} }{\xi_2 - \xi_1} \Big) l \\
&\quad+ \Big( - \frac{3}{16} \, \e^2 + \frac{ r_3(\xi_2) - r_3(\xi_1) }{\xi_2 - \xi_1} \Big) (j + \kappa j^3)^{1/2} 
\\ & \quad \ 
+ \frac{ \lm_1(\xi_2) - \lm_1(\xi_1) }{\xi_2 - \xi_1} \, j^{1/2} 
+ \frac{ \lm_{-1}(\xi_2) - \lm_{-1}(\xi_1) }{ \xi_2 - \xi_1 } \, j^{-1/2} 
\\
& = 
\e^2 \Big( - \frac{3}{16} \, - \frac{\bar\om_2}{\bar\om} + O(\e) \Big) (j + \kappa j^3)^{1/2}
+ O(\e^2 j^{1/2}).
\end{align*}
Now $- \frac{3}{16} - \frac{\bar\om_2}{\bar\om}$ is nonzero for all $\kappa \geq 0$ 
(using \eqref{def om 2}, one can check that $|\frac{3}{16} + \frac{\bar\om_2}{\bar\om}| \geq 2$ for all $\kappa \geq 0$). Hence 
\begin{equation} \label{der flj}
| f_{lj}(\xi_2) - f_{lj}(\xi_1) | \geq \e^2 c j^{3/2} |\xi_2 - \xi_1| \quad \forall j \geq C, 
\end{equation}
where $C,c>0$ are constants depending only on $\kappa$. 

\begin{remark} 
For $ 2 \leq j < C $ one could impose a finite list of inequalities for $\kappa$, and obtain, as a consequence, that \eqref{der flj} holds for all $j \geq 2$. However, there is no need of doing in that way: using the cut-off \eqref{cut-off} below, the low frequencies $j < C$ have not to be studied if $\e$ is small enough. 
\end{remark}

By \eqref{der flj}, the measure of the set $\mA_{lj}$ is 
\begin{equation} \label{measure mA lj 0}
| \mA_{lj} | \leq \frac{2 \g}{j^\t} \frac{1}{\e^2 c j^{3/2}} \,
= \frac{C \g \e^{-2}}{j^{\t + (3/2)}}\,.
\end{equation}
We impose a Diophantine condition on the surface tension coefficient $\kappa$: we assume \eqref{kappa dioph}, namely 
\[
| \bar\om l + (j+\kappa j^3)^{1/2}| 
= | \sqrt{1 + \kappa}\, l + \sqrt{j + \kappa j^3} | > \frac{ \g_*}{j^{\t_*}} \quad 
\forall l \in \Z, \ j \geq 2,
\]
for some constant $\g_* \in (0,1/2)$, where we fix $\t_* = 3/2$. 
By \eqref{kappa dioph}, if $\mA_{lj} \neq \emptyset$, then 
\[ 
\frac{\g}{j^\t} > | \om l + \mu_j | 
\geq | \bar\om l + (j+\kappa j^3)^{1/2}| - C \e^2 j^{3/2}
> \frac{\g_*}{j^{\t_*}} - C \e^2 j^{3/2}
\]
for some $C \geq 1$, whence 
$C \e^2 j^{3/2} > \g_* j^{-\t_*} - \g j^{-\t}
\geq \g_* j^{-\t_*} / 2$ if $\g \leq \g_* / 2$ (i.e. $\e$ small enough) and $\t \geq \t_*$ (we have fixed $\t = \t_* = 3/2$). 
Thus we have found the following ``cut-off'': 
$\mA_{lj}$ can be nonempty only for 
\begin{equation} \label{cut-off}
j > \Big( \frac{\g_*}{2C \e^2} \Big)^{\frac{1}{\t_* + (3/2)}} =: C_0 \e^{-\a}, \quad 
\a := \frac{2}{\t_* + (3/2)} = \frac23. 
\end{equation}
Thus, by \eqref{number of l}, \eqref{measure mA lj 0}, 
\begin{align} \label{measure sum}
\Big| \bigcup_{l \in \Z, \, j \geq 2} \mA_{lj} \Big| 
& \leq \sum_{j > C_0 \e^{-\a}} (C \e^2 j^{3/2} + 2) \frac{C \g \e^{-2}}{j^{\t + (3/2)}} \\
&\leq C \g \sum_{j > C_0 \e^{-\a}} \frac{1}{j^\t} 
+ C \g \e^{-2} \sum_{j > C_0 \e^{-\a}} \frac{1}{ j^{\t + (3/2)} }\notag
\\ & 
\leq C \g (\e^{-\a})^{-\t+1} + C \g \e^{-2} (\e^{-\a})^{-\t - \frac12 }
\leq C \e^{1/18}.  
\notag 
\end{align} 
We have proved that $[1,2] \setminus \mG_1$ has Lebesgue measure $\leq C \e^{1/18}$, 
which is \eqref{Gn - Gn+1} for $n = 0$.

\begin{remark} \label{rem:why imposing kappa dioph}
The condition \eqref{kappa dioph} allows to get a positive measure estimates even if $\g = \e^{5/6} \gg \e$. The advantage of imposing \eqref{kappa dioph} is that, regarding size, $\mD^{-1} \mR = O(\e \g^{-1}) = O(\e^{1/6}) \ll 1$, so that $\mD + \mR$ can be inverted simply by Neumann series. 

Without \eqref{kappa dioph}, in the sum \eqref{measure sum} the cut-off $j > C \e^{-2/3}$ disappears, and the second sum becomes $\leq C \g \e^{-2}$. 
Therefore, to get a parameter set of asymptotically full measure, it should be $\g = o(\e^2)$ as $\e \to 0$. But then $\mD^{-1} \mR = O(\g^{-1}) \mR$ is small only if $\mR = o(\g)$. 
This means that one has to expand $\mR = \e \mR_1 + \e^2 \mR_2 + o(\e^2)$, to calculate the precise formula of $\mR_1, \mR_2$, to invert $\mD + \e \mR_1 + \e^2 \mR_2$ in a non-perturbative way, 
and then to invert $\mD + \mR$ as a perturbation of $\mD + \e \mR_1 + \e^2 \mR_2$. 
This means, in fact, that one has to calculate the normal form of order 2.
\end{remark}

\begin{remark} \label{rem:why gamma depends on epsilon}
We could also fix $\g$ to be independent of $\e$, taking a larger value of $\t$. 
However, the larger is $\t$, the larger is the number of steps we have to make in Section \ref{sec:semi-FIO conj} to reach a sufficiently regularizing remainder $\mR$ (it should be 
$\mR = O(|D_x|^{-\t})$ to obtain $\mR \mD^{-1}$ bounded in Section \ref{sec:inv lin op}).
Hence it is convenient to keep $\t$ as lowest as possible, but still sufficiently large to get a  positive measure set of parameters.
\end{remark}

Now we prove \eqref{Gn - Gn+1} for $n \geq 1$. Let $J_n := \e^{-11/18} 4^n$. 
Let $\mA^{n+1}_{lj} := \{ \xi \in \mG_n : | \om l + \mu_j(u_n) | < \g j^{-\t} \}$.
For $j > J_n$ we follow exactly the same argument above, and we find 
\begin{align*} 
\Big| \bigcup_{l \in \Z, \, j > J_n} \mA^{n+1}_{lj} \Big| 
& \leq \sum_{j > J_n} (C \e^2 j^{3/2} + 2) \frac{C \g \e^{-2}}{j^3} \\
&\leq C \g \sum_{j > J_n} \frac{1}{j^{3/2}} 
+ C \g \e^{-2} \sum_{j > J_n} \frac{1}{ j^3 }
\\ & 
\leq C \g J_n^{-1/2} + C \g \e^{-2} J_n^{-2}
\leq C \e^{1/18} \, 2^{-n}.
\end{align*} 
For $j \leq J_n$ we use Lemma \ref{lemma:eigenvalues der u}, \eqref{Hn}, the Lipschitz estimate
\[
|\mu_j(u_n) - \mu_j(u_{n-1})| 
\leq C \| u_n - u_{n-1} \|_{\s_0} j^{3/2} 
= C \| h_n \|_{\s_0} j^{3/2} 
\leq C \e^{2 + \d + \frac19 \chi^n} j^{3/2}
\]
and the triangular inequality to deduce that, if $\xi \in \mG_n$, then 
\[
|\om l + \mu_j(u_n)| 
\geq |\om l + \mu_j(u_{n-1})| - |\mu_j(u_n) - \mu_j(u_{n-1})| 
\geq \g j^{-\t} - C \e^{2+\d + \frac19 \chi^n} j^{3/2}. 
\]
On the other hand, if $\xi \in \mA_{lj}^{n+1}$, then $|\om l + \mu_j(u_n)| < \g j^{-\t}$, 
and therefore $f_{lj}(\xi) := \om l + \mu_j(u_n)$ is in a region of Lebesgue measure  
$\leq C \e^{2+\d + (1/9) \chi^n} j^{3/2}$. 
Thus we follow the same argument as above, but with $C \e^{2+\d + (1/9) \chi^n} j^{3/2}$ instead of 
$2 \g j^{-\t}$. We get
\begin{align*} \label{measure sum}
\Big| \bigcup_{l \in \Z, \, j \leq J_n} \mA^{n+1}_{lj} \Big| 
& \leq \sum_{j \leq J_n} (C \e^2 j^{3/2} + 2) C \e^{2 + \d + \frac19 \chi^n} j^{3/2} \frac{1}{c \e^2 j^{3/2}}\\
&\leq C \e^{2 + \d + \frac19 \chi^n} J_n^{5/2} + C \e^{\d + \frac19 \chi^n} J_n ,
\end{align*} 
which is $\leq C \e^{1/18} \, 2^{-n}$ because $2 + \d + \frac19 \chi - \frac52 \frac{11}{18} \geq \frac{1}{18}$ 
and $\d + \frac19 \chi - \frac{11}{18} \geq \frac{1}{18}$.
$(\mP 3)_n$ is proved.
\qed

\bigskip

\textbf{Proof of Theorem \ref{thm:main} concluded.}
Theorem \ref{thm:NM} implies that the sequence $u_n$ is well-defined for 
$\xi \in \mG_\infty := \cap_{n \geq 0} \, \mG_n \subset [1,2]$. 
By \eqref{Gn - Gn+1}, the set $\mG_\infty$ has positive Lebesgue measure 
$| \mG_\infty | \geq 1 - C \e^{1/18}$, asymptotically full $|\mG_\infty| \to 1$ as $\e \to 0$.
By \eqref{Hn}, $u_n$ is a Cauchy sequence in $\| \ \|_{s_0 + \s}$, and therefore it converges to a limit $u_\infty$ in $H^{s_0 + \s}(\T^2)$. 
By \eqref{Hn}, for all $\xi \in \mG_\infty$, $u_\infty$ is a solution of $F(u_\infty, \om) = 0$, 
with $\| u_\infty - \bar u_\e \|_{s_0 + \s} \leq C \e^{2+\d}$,
where $\om = \om(\xi)$ is given by \eqref{freq-ampl}.
Renaming $u := u_\infty$, $\mG_\e := \mG_\infty$, the proof of Theorem \ref{thm:main} is complete.
\qed

\bigskip

\textbf{Proof of Lemma \ref{lemma:kappa dioph}.}
Let $\kappa_0 > 0$, $\t_* > 1$. 
Let $\g_* \in (0,1/2)$, $l \in \Z$, $j \geq 2$, and define
\[
\mA_{l,j}(\g_*) := \Big\{ \kappa \in [0,\kappa_0] : 
| f_{lj}(\kappa) | \leq \frac{\g_*}{ j^{\t_*} } \Big\}, \qquad 
f_{lj}(\kappa) := l + \frac{ \sqrt{j + \kappa j^3}}{ \sqrt{1+\kappa} } \,.
\]
If $\mA_{lj}(\g_*) \neq \emptyset$, then $ |l| < C j^{3/2}$ for some constant $C > 0$ 
depending on $\kappa_0$ and independent of $\t_*, \g_*, l, j$.
Therefore for each $j$ there are at most $C j^{3/2}$ indices $l$ such that $\mA_{lj}(\g_*) \neq \emptyset$. 
Moreover the derivative of $f_{lj}(\kappa)$ with respect to $\kappa$ is
\[
f_{lj}'(\kappa) = \frac{j^3 - j}{2 \sqrt{j + \kappa j^3} (1 + \kappa)^{3/2} } 
\geq \frac{j^3 - j}{2 \sqrt{j + \kappa_0 j^3} (1 + \kappa_0)^{3/2} } 
\geq c j^{3/2}
\]
for some $c > 0 $ (depending on $\kappa_0$ and independent of $\t_*, \g_*, l, j$).
Hence the Lebesgue measure of $\mA_{lj}(\g_*)$ is 
\[
|\mA_{lj}(\g_*)| \leq \frac{2 \g_*}{j^{\t_*}} \, \frac{1}{c j^{3/2}} \, 
= \frac{C \g_*}{j^{\t_* + (3/2)}}
\]
for some $C > 0$.
Since $\t_* > 1$, 
\[
\Big| \bigcup_{l \in \Z, \, j \geq 2} \mA_{lj}(\g_*) \Big| 
\leq \sum_{j \geq 2} \frac{C \g_*}{j^{\t_* + (3/2)}} j^{3/2} \leq C \g_*
\]
for some $C$ depending on $\kappa_0, \t_*$. 
As a consequence, the set $\tilde \mK(\g_*) := \{ \kappa \in [0,\kappa_0] : |f_{lj}(\kappa)| > \g_* j^{-\t_*} \}$ has Lebesgue measure $|\tilde \mK(\g_*)| \geq \kappa_0 - C \g_*$. 
Therefore $\tilde \mK := \bigcup_{\g_* \in (0,1/2)} \tilde \mK(\g_*)$ has full measure 
$| \tilde \mK| = \kappa_0$. 
Finally note that $\tilde \mK \subset \mK \subset [0,\kappa_0]$. 
The proof of Lemma \ref{lemma:kappa dioph} is complete.
\qed

\section{Pseudo-differential operators in the class $S^0_{1/2,1/2}$}
\label{sec:semi-FIO}

In this section we prove some results on pseudo-differential operators in the class $S^0_{1/2,1/2}$
on the 1-dimensional torus that are used in our existence proof for the water waves problem. 
These results also hold for a more general class of Fourier integral operators.
In Sections \ref{sec:invertibility}-\ref{sec:composition} we prove the invertibility, composition formulae and tame estimates for operators depending on the space variable $x \in \T$ only, 
then in section \ref{sec:with time dep} we explain how to include the dependence on the time variable $t \in \T$.

\subsection{Invertibility} \label{sec:invertibility}

We consider Fourier integral operators that change $e^{ikx}$ into $e^{i \phi(x,k)}$ for some 
phase function $\phi$. Namely, 
let $L>0$ and let $f \colon \R \to \R$ be a $C^\infty$ function with
\begin{equation} \label{f basic} 
f(0) = 0, \qquad 
\|f'\|_{L^\infty} \leq L,
\end{equation}
so that $|f(\xi) - f(\eta)| \leq L |\xi-\eta|$ and $|f(\xi)| \leq L |\xi|$ for all $\xi,\eta \in \R$. 
Let $\b(x)$ be a real-valued periodic function and let
\[
\phi(x,\xi) := \xi x + f(\xi) \beta(x), \quad x,\xi \in \R.
\]
Denote 
\[
w_\xi(x) := e^{i \phi(x,\xi)}, \quad 
e_\xi(x) := e^{i \xi x}, \quad 
x,\xi \in \R.
\]
When $\xi = k$ is an integer, both $e_k$ and $w_k$ are $2\p$-periodic functions of $x$. 
We define the operator $A$ by setting 
$A e_\xi = w_\xi$ for $\xi \in \R$. 
Thus 
\begin{equation}  \label{def A real line}
Ag(x) = \int_\R \hat g(\xi) \, w_\xi(x)\,d\xi,
\quad 
g(x) = \int_\R \hat g(\xi) \, e_\xi(x)\,d\xi
\end{equation}
for functions $g : \R \to \C$, where $\hat g$ is the Fourier transform of $g$ on the real line, and, on the torus,
\begin{equation}  \label{def A torus}
Au(x) = \sum_{k \in \Z} \hat u_k\,w_k(x), 
\quad
u(x) = \sum_{k \in \Z} \hat u_k\,e_k(x)
\end{equation}
for periodic functions $u:\T \to \C$, where $\hat u_k$ are the Fourier coefficients of $u$. 

\medskip

\noindent 
\textbf{Adjoint operators.} 
Quantitative estimates for $A$ and its inverse are the goal of this section. 
To obtain these bounds, we shall study $A^* A$ and $A A^*$.

Consider the scalar product of $L^2(\T)$ and that one of $L^2(\R)$, 
\[
(u,v)_{L^2(\T)} = \int_\T u(x)\,\overline{v(x)}\,dx, \quad
(g,h)_{L^2(\R)} = \int_\R g(x)\,\overline{h(x)}\,dx, 
\]
where $u,v \in L^2(\T)$ and $g,h \in L^2(\R)$. 
Denote $A^*_\T$, $A^*_\R$ the adjoint of $A$ with respect to the scalar product of $L^2(\T)$ and $L^2(\R)$ respectively, namely 
\[
A^*_\T u(x) = \sum_{k \in \Z} (u,w_k)_{L^2(\T)}\, e_k(x), \quad x \in \T, 
\]
and
\[
A^*_\R g(x) = \int_\R (g, w_\xi)_{L^2(\R)}\, e_\xi(x) \,d\xi, 
\quad x \in \R.
\]
Hence 
\[
A^*_\T\,A u(x) = \sum_{k \in \Z} (Au,w_k)_{L^2(\T)} \, e_k(x) 
= \sum_{k,j \in \Z} (w_j,w_k)_{L^2(\T)}\,\hat u_j \,  e_k(x)\,,
\]
namely the operator $M := A^*_\T\,A$ is represented by the matrix $(M_k^j)_{k,j \in \Z}$ with respect to the exponentials basis $\{ e_k \}_{k \in \Z}$, where 
\begin{equation}\label{n104}
M_k^j := (w_j,w_k)_{L^2(\T)}\,, \quad k,j \in \Z.
\end{equation}
On the other hand, 
\begin{equation} \label{AA*T}
A A^*_\T\,u(x) = \sum_{k \in \Z} (u,w_k)_{L^2(\T)}\, w_k(x) \,.
\end{equation}
We shall see that, to prove the invertibility of $A A^*_\T$, 
instead of writing a matrix representation like $M$ above, it is convenient to study 
\begin{equation} \label{AA*R}
A A^*_\R\,g(x) = \int_\R (g,w_\xi)_{L^2(\R)} \, w_\xi(x) \, d\xi\,,
\quad x \in \R
\end{equation}
and pass from the real line to the torus in a further step. 

Let us begin with estimates on $A^*_\T A$. Notation: Sobolev norms on the torus are denoted by 
\[
\| u \|_s = \| u \|_{H^s(\T)}, \quad \| u \|_0 = \| u \|_{L^2(\T)};
\]
other norms are indicated explicitly.

\begin{lemma}[Estimates for $A^*_\T A$]  \label{lemma:nuovo M} 
There exist universal constants $C,\d > 0$, with $C \d < 1/4$, with the following properties. 

($i$) If $L \| \b \|_3 \leq \d$, then $M = A^*_\T A : L^2(\T) \to L^2(\T)$ is bounded and invertible, with 
\begin{equation}\label{n107}
\| (M-I) u \|_0 + \| (M\inv -I) u \|_0 \leq  C L \| \b \|_3  \| u \|_0.
\end{equation}
As a consequence, for $\delta$ is small enough,
\[
\| M u \|_0 \leq 2 \| u \|_0 \,, \quad
\| M\inv u \|_0 \leq 2 \| u \|_0 \,.
\]
($ii$) Let $s \geq 1$. If $L \| \b \|_3 \leq \d$ and $\b \in H^{s+2}(\T)$, then both $M$ and $M\inv$ are bounded and invertible from $H^s(\T)$ onto $H^s(\T)$, with 
\begin{equation} \label{tame Ms}
\| (M-I)u \|_s + \| (M\inv-I)u \|_s \,
\leq C L \|\b\|_3 \|u\|_s + C(s) L \|\b\|_{s+2} \|u\|_1
\end{equation}
where $C$ is the universal constant of part ($i$), 
and $C(s)>0$ depends only on $s$. 
As a consequence, 
\begin{equation} \label{tame M}
\| Mu \|_s  \,, \, 
\| M\inv u \|_s \, 
\leq 2 \| u \|_s + C(s) L \|\b\|_{s+2} \|u\|_1 \,.
\end{equation}
\end{lemma}

\begin{proof}
($i$) Fix a universal constant $\d_0 > 0$ such that if $\| u \|_3 \leq \d_0$, then $\| u' \|_{L^\infty} \leq 1/2$ and $\| u \|_2 \leq 1$. 
Thus we can assume that $L \| \b' \|_{L^\infty} \leq 1/2$ and $L \| \b \|_2 \leq 1$. 

Using the notation \eqref{n104}, on the diagonal $j=k$, one has
\[
M_k^k = (w_k, w_k)_{L^2(\T)} = 2\pi 
\]
because $\overline{w_k} = w_k^{-1}$. 
For $j \neq k$,  
\[
M_k^j = \intp e^{i \om (x+p(x))} \, dx \,, \quad
\om := j-k, \quad 
p(x) := \frac{f(j) - f(k)}{j-k} \, \b(x)\,.
\]
By \eqref{f basic}, $|p'(x)| \leq L |\b'(x)| \leq 1/2$ for all $x$, 
and $\| p \|_s \leq L \|\b\|_s$ for every $s \geq 0$. 
In particular, $\|p\|_2 \leq 1$. 
By Lemma \ref{lemma:non-stationary phase}, 
\begin{equation} \label{stima Mkj}
|M_k^j| \leq \frac{C(\a) L \|\b\|_{\a+1}}{|k-j|^\a}\,.
\end{equation}
Split $M = I + R$, where $I$ is the identity map and $R_k^j = M_k^j$ for $k \neq j$ and $R_k^k = 0$. 
By H\"older's inequality and \eqref{stima Mkj} applied with $\alpha=2$,
\begin{align}
\| R u \|_0^2 
& = \sum_k \Big| \sum_{j \neq k} M_k^j \hat u_j \Big|^2 
\leq \sum_k \Big( \sum_{j \neq k} |M_k^j| |\hat u_j| \Big)^2 
\notag
\\ & 
\leq \sum_k \Big( \sum_{j \neq k} \frac{C L \| \b \|_3}{|k-j|} \,
\, \frac{|\hat u_j|}{|k-j|} \Big)^2
\notag
\\ & 
\leq \sum_k \Big( \sum_{j \neq k} \frac{C^2 L^2 \| \b \|_3^2}{|k-j|^2}\,\Big)
\Big( \sum_{j \neq k} \frac{|\hat u_j|^2}{|k-j|^2} \Big)
\notag
\\ & 
= C^2 L^2 \| \b \|_3^2 \,C_2  \sum_{j} \Big( \sum_{k \neq j} \frac{1}{|k-j|^2} \, \Big) |u_j|^2 
\notag
\\ & 
\leq C^2 L^2 \, C_2^2 \, \| \b \|_3^2 \,  \|u\|_0^2 \label{Ru0}
\end{align}
where $C_2 = \sum_{k \neq 0} |k|^{-2}\, < \infty$. 
Thus $\| Ru \|_0 \leq C_0 L \| \b \|_3 \, \| u \|_0$ 
for some universal constant $C_0 > 0$. 
This is the desired estimate \eqref{n107} for $R=M-I$. 
By Neumann series, if $C_0 L \|\b\|_3 \leq 1/2$, 
then $M  : L^2(\T) \to L^2(\T)$ is invertible, with 
\begin{align*}
\|(M\inv - I)u \|_0 
\leq \sum_{n=1}^\infty \|R^n u\|_0 
\leq  2 C_0 L \| \b \|_3 \, \|u\|_0\,.
\end{align*}

\medskip

($ii$) For $s \geq 1$, $k \in \Z$, split $\Z_k := \Z \setminus \{k\}$ into two components,
\[
\Z_k = A \cup B, 
\quad 
A = \{ j \in \Z_k : \la k \ra^s \leq 2 \la j \ra^s \}, 
\quad
B = \Z_k \setminus A,
\]
and write
\[ 
\| Ru \|_s^2 
\leq \sum_k \Big\vert \sum_{j \neq k} M_k^j \hat u_j \Big\vert^2 \la k \ra^{2s} 
\leq 2 (S_A + S_B),
\]
where 
\[
S_A = \sum_k \Big( \sum_{j \in A} |M_k^j| |\hat u_j| \la k \ra^{s} \Big)^2 
\]
and similarly $S_B$.
For $S_A$ we use estimate \eqref{stima Mkj} with $\a = 2$,
\[
S_A \leq \sum_k \Big( \sum_{j \in A} \frac{C L \|\b\|_3}{|k-j|^2}\, |\hat u_j|\, 2 \la j \ra^{s} \Big)^2 \,,
\]
then repeat the same calculations as \eqref{Ru0} with $|\hat u_j| \la j \ra^s$ instead of $|\hat u_j|$, whence
\[
S_A \leq (C L \|\b\|_3 \|u\|_s)^2
\]
where $C$ is a universal constant. 
The estimate for $S_B$ is similar, 
applying \eqref{stima Mkj} with $\a = s+1$ and 
noticing that 
$\la k \ra \leq c_s |k-j|$ for $j \in B$, we obtain
\begin{align}
S_B & 
\leq \sum_k \Big( \sum_{j \in B} 
\frac{C(s+1) L \| \b \|_{s+2}}{|k-j|^{s+1}}\, 
|\hat u_j|\, c_s^s |k-j|^{s} \Big)^2
\notag
\\ & 
\leq C(s) \sum_k \Big( \sum_{j \in B} 
\frac{L \| \b \|_{s+2}}{|k-j|}\, 
|\hat u_j|\, \la j \ra \, \frac{1}{\la j \ra} \Big)^2
\notag
\\ & 
\leq C(s) \sum_k \Big( \sum_{j \in B} 
\frac{L^2 \| \b \|_{s+2}^2}{|k-j|^2}\, |\hat u_j|^2\,\la j \ra^2 \Big) 
\Big( \sum_{j \in B} \frac{1}{\la j \ra^2} \Big)
\notag
\\ & 
\leq C(s) L^2 \| \b\|_{\a+2}^2 \| u \|_1^2 .
\end{align}
This yields
 \[
\| Ru \|_s \leq C_1 L \| \b \|_3 \| u \|_s + C(s) L \| \b\|_{s+2} \| u \|_1\,, \quad 
\| Ru \|_1 \leq C_1 L \| \b \|_3 \| u \|_1 
\]
where $C_1$ is a universal constant and $C(s)$ depends on $s$. 
Hence $Mu \in H^s(\T)$ for all $u \in H^s(\T)$ together with 
the estimate for $M-I$ given by \eqref{tame Ms}. Now, by induction, 
\[
\| R^n u\|_s \leq (C_1 L \| \b \|_3)^n \| u \|_s 
+ n (C_1 L \| \b \|_3)^{n-1} C(s) L \| \b\|_{s+2} \| u \|_1 \quad \forall n \geq 1,
\]
and the desired estimate \eqref{tame Ms} for $M^{-1}-I$ follows from Neumann series. 
\end{proof}

As an immediate corollary of the operator norm estimate for 
$A^*_\T A$, we have a bound for $A$:

\begin{lemma}[$L^2$-bound for $A$] 
\label{lemma:L2 bound A} 
Let $\d$ be the universal constant of Lemma \ref{lemma:nuovo M}. 
If $L \| \b \|_3 \leq \d$, then both $A$ and $A^*_\T : L^2(\T) \to L^2(\T)$ are bounded, with
\[
\| Au \|_0 \leq 2 \|u\|_0\,, \quad 
\| A^*_\T u \|_0 \leq 2 \|u\|_0\, \qquad  \forall u \in L^2(\T).
\]
\end{lemma}

Now we consider an operator $E$ with the same phase $\phi(x,\xi)$ as $A$ and, in addition, an amplitude $a(x,\xi)$, namely
\[
E : e_\xi \mapsto q_\xi, \quad q_\xi(x) := a(x,\xi)\,w_\xi(x), 
\]
where $a(x,\xi)$ is a $2\p$-periodic function of $x$ for every $\xi \in \R$ (or, at least, for every $\xi = k \in \Z$). 
If $u(x)$ is a periodic function with Fourier coefficients $\hat u_k$, then 
\[
Eu(x) = \sum_{k \in \Z} \hat u_k q_k(x) 
= \sum_{k \in \Z} \hat u_k a(x,k) w_k(x).
\]
Analogous definition for functions $g$ on the real line. 

\begin{remark}  \label{remark:E*E}
The adjoint operator $E^*_\T$ of $E$ with respect to the $L^2$-scalar product on the torus is 
\[
E^*_\T u(x) = \sum_{k \in \Z} (u,q_k)_{L^2(\T)}\, e_k(x), \quad x \in \T, 
\]
therefore
\[
E^*_\T\,E u(x) = \sum_{k \in \Z} (Eu,q_k)_{L^2(\T)} \, e_k(x) 
= \sum_{k,j \in \Z} (q_j,q_k)_{L^2(\T)}\,\hat u_j \,  e_k(x)\,,
\]
namely the operator $G := E^*_\T\,E$ is represented by the matrix $(G_k^j)_{k,j \in \Z}$ with respect to the exponentials basis $\{ e_k \}_{k \in \Z}$, where 
\[
G_k^j := (q_j,q_k)_{L^2(\T)}\,, \quad k,j \in \Z. 
\qedhere
\]
\end{remark}

\begin{lemma}[$L^2$-bound with amplitude] 
\label{lemma:L2 bound E}
Let $\d$ be the universal constant of Lemma \ref{lemma:nuovo M}, and let $L \| \b \|_3 \leq \d$. 
Let 
\[
\s,\t,K \in \R, \quad \s > 1, \quad \t,K \geq 0.
\]
If $a(\cdot\,,k) \in H^\s(\T)$ for all $k \in \Z$, with 
\[
\| a(\cdot\,,k) \|_\s \leq K \la k \ra^\t \quad \forall k \in \Z,
\]
then $E : H^\t(\T) \to L^2(\T)$ is bounded, with 
\[
\| Eu \|_0 \leq 2 C_\s K \|u\|_\t \quad \forall u \in H^\t(\T),
\]
where $C_\s := \sum_{j \in \Z} \la j \ra^{-\s} < \infty$. 
\end{lemma}

\begin{proof} Develop $a(x,k)$ in Fourier series in $x$, 
$a(x,k) = \sum_{j \in \Z} \hat a_j(k)\, e_j(x)$, with 
$|\hat a_j(k)| \la j \ra^\s$ $\leq K \la k \ra^\t$. 
Write $Eu$ as 
\begin{equation}
Eu(x)  = \sum_{k} \hat u_k \, \Big( \sum_j \hat a_j(k) \, e_j(x) \Big) w_k(x) 
 = \sum_{j} (A F_j u)(x) \, e_j(x) \label{E trucco}
\end{equation}
where $F_j$ is the Fourier multiplier 
$F_j : e_k \mapsto \hat a_j(k) \, e_k$, satisfying 
\[
\| F_j u \|_0^2 = 
\sum_k |\hat u_k|^2 |\hat a_j(k)|^2 
\leq \sum_k |\hat u_k|^2 \frac{K^2 \la k \ra^{2\t}}{\la j \ra^{2\s}}
= \frac{K^2}{\la j \ra^{2\s}}\, \| u \|_\t^2 \,.
\]
Remembering $\| Au \|_0 \leq 2 \|u\|_0$ 
(see Lemma \ref{lemma:L2 bound A}), we obtain
$$
\| Eu \|_0  \leq \sum_j \| (A F_j u) e_j \|_0 
\leq \sum_j \| A F_j u \|_0 
\leq 2 \sum_j \| F_j u \|_0 
\leq 2 K C_\s \| u \|_\t 
$$
where $C_\s = \sum_j \la j \ra^{-\s}$.
\end{proof}

Now go back to the study of $A:e_k \mapsto w_k$. 

\begin{lemma}[Sobolev bounds] 
\label{lemma:Hs bound A} 
Let $L \| \b \|_3 \leq \d$, where $\d \leq 1/2$ is the universal constant of the previous lemmas. 
Let $\a \geq 1$ be an integer. 
If $\b \in H^{\a+2}(\T)$, then 
$A$ and $A^*_\T$ $: H^\a(\T) \to H^\a(\T)$ are bounded, with 
\[
\| A u \|_\a + \| A^*_\T\, u \|_\a 
\leq C(\a) \Big( \| u \|_{\a} + L \| \b \|_{\a + 2} \| u \|_1 \Big)
\]
for some constant $C(\a) > 0$. 
\end{lemma}

\begin{proof} 
Let us prove the estimate for $A$ first. Since we already proved that 
$\| Au \|_0 \leq 2 \| u \|_0 \leq 2 \|u\|_\a$ and since $\| Au \|_\a \leq C(\a) (\|Au\|_0 + \|\pa_x^\a Au \|_0)$, 
it is sufficient to estimate the $L^2$-norm of $\pa_x^\alpha Au$. 

\smallbreak

The derivatives of $w_k(x) = e^{i \phi(x,k)}$ satisfy 
$\pa_x^\a (e^{i \phi(x,k)}) =  P_\a(x,k)\,e^{i \phi(x,k)}$ 
with
\begin{equation} \label{formula Palfa 1}
P_\a(x,k) = \sum_{n=1}^{\a} \ \sum_{\nu \in S_{\a,n}} 
C(\nu) \, (\pa_x^{\nu_1} \phi)(x,k) \cdots (\pa_x^{\nu_n} \phi)(x,k) \,,
\end{equation}
where $\nu= (\nu_1, \ldots, \nu_n)  \in S_{\a,n}$ means 
$1 \leq \nu_1 \leq \ldots \leq \nu_n$ and 
$\nu_1 + \ldots + \nu_n = \a$.
Therefore 
\begin{align*}
\pa_x^\a Au(x) & = \sum_{k \in \Z} \hat u_k \, P_\a(x,k)\,w_k(x) 
= 
\sum_{n=1}^{\a} \ \sum_{\nu \in S_{\a,n}} 
C(\nu) \, E_\nu u(x), 
\end{align*}
where
\[
E_\nu u(x) := \sum_{k \in \Z} \hat u_k \, a_\nu(x,k)\, w_k(x), \quad 
a_\nu(x,k) := (\pa_x^{\nu_1} \phi)(x,k) \cdots (\pa_x^{\nu_n} \phi)(x,k).
\]
Write $\phi$ as 
\[
\phi(x,k) = k \, h(x,k), \quad 
h(x,k) = x + \frac{f(k)}{k}\,\b(x), \quad k \neq 0.
\]
Since $\nu_i \geq 1$, one may write
$\| \pa_x^{\nu_i} \phi(\cdot\,,k) \|_2 \leq |k| \, \| h'\|_{\nu_i +1}$ 
($h'=\pa_x h$). Therefore
\[
\| a_\nu(\cdot\,,k)\|_2 
\leq |k|^n C^{n-1} \| h'\|_{\nu_1 +1} \cdots \| h'\|_{\nu_n +1}
\]
where $C$ is the algebra constant of $H^2(\T)$ 
so that $\| uv \|_2 \leq C \| u \|_2 \| v \|_2$. 
By Lemma \ref{lemma:L2 bound E} (here $\s=2$)
\[
\| E_\nu u \|_0 \leq C(\a) \, \| h'\|_{\nu_1 +1} \cdots \| h'\|_{\nu_n +1} \, \| u \|_n 
\]
for some constant $C(\a)$ depending on $\a$. 
By interpolation in Sobolev class, since $1\leq n \leq \a$, 
\[
\| h'\|_{\nu_i +1} \leq 
2\, \| h'\|_{2}^{\th_i} \, \| h'\|_{\a+1}^{1-\th_i}\,, \quad 
\| u \|_n \leq 2\, \| u \|_1^{\th_0} \| u \|_{\a}^{1-\th_0},
\]
with $\th_0,\th_i \in [0,1]$, $i=1,\ldots,n$, and  
\[
\nu_i + 1 = 2 \th_i + (\a+1) (1-\th_i), \quad
n = 1\, \th_0 + \a (1-\th_0).
\]
Hence
\[
\prod_{i=1}^n \| h'\|_{\nu_i +1} 
\| u \|_n 
\leq 2^{n+1} \| h'\|_{2}^{\th_1 + \ldots + \th_n} 
\| h'\|_{\a +1}^{n-(\th_1 + \ldots + \th_n)}  \, 
\| u \|_1^{\th_0} \| u \|_{\a}^{1-\th_0}.
\]
Since $\nu_1 + \ldots + \nu_n = \a$, 
\[
\th_0 + \th_1 + \ldots + \th_n = n, \quad 
(\th_1 + \ldots + \th_n) = (n-1) + (1-\th_0),
\] 
and 
\begin{align*}
\prod_{i=1}^n \| h'\|_{\nu_i +1} 
\| u \|_n 
& \leq 2^{n+1} \| h'\|_{2}^{n-1} 
\big( \| h'\|_{2} \| u \|_{\a} \big)^{1-\th_0}
\big( \| h'\|_{\a +1} \| u \|_1 \big)^{\th_0} \\
& \leq 2^{n+1} \| h' \|_2^{n-1} \Big( \| h'\|_{2} \| u \|_{\a} + \| h'\|_{\a +1} \| u \|_1 \Big).
\end{align*}
By assumption, $L \| \b \|_3 \leq \d \leq 1$, and, by \eqref{f basic}, $|f(\xi)| \leq L|\xi|$ for all $\xi$, therefore 
\[
\| h' \|_2 = \| 1 + f(k) k\inv \b' \|_2 
\leq 1 + L \| \b \|_3 \leq 2, 
\quad 
\| h' \|_{\a+1} \leq 1 + L \| \b \|_{\a+2},
\]
and 
\[ 
\| E_\nu u \|_0 
\leq C(\a) \Big( \| u \|_{\a} + (1 + L \| \b \|_{\a + 2}) \| u \|_1 \Big) 
\leq C(\a) \Big( \| u \|_{\a} + L \| \b \|_{\a + 2} \| u \|_1 \Big).
\]
Taking the sum over $n=1,\ldots, \a$, $\nu \in S_{\a,n}$, gives the desired estimate 
for the $L^2$-norm of $\pa_x^\a Au$, which completes the proof of the estimate for $A$.

\medskip

Now we prove the same estimate for $A^*_\T$. Remember that 
\[
A^*_\T u(x) 
= \sum_{k \in \Z} \Big( \int_\T u(y)\, e^{-i \phi(y,k)}\, dy \Big) \, e^{ikx}\,.
\]
Write $-\phi(y,k)$ as 
\[
-\phi(y,k) = (-k) \big( y + p(y,k) \big), \quad 
p(y,k) = \frac{f(k)}{k}\,\b(y).
\]
Using Lemma \ref{lemma:non-stationary phase} 
together with the notations introduced in its proof, with $\om=-k$, 
\begin{align*}
\pa_x^\a (A^*_\T u)(x) 
& = \sum_{k \in \Z} \Big( \int_\T u(y)\, e^{-i \phi(y,k)}\, dy \Big) (ik)^\a \, e^{ikx} 
\\ & 
= \sum_{k \in \Z} \Big( \frac{i^\a}{(-k)^\a}\, \int_\T Q_\a(y) \, e^{-i\phi(y,k)}\, dy \Big) \, (ik)^\a \, e^{ikx} 
\\ & 
= \sum_{k \in \Z} \Big( \int_\T Q_\a(y) \, e^{-i\phi(y,k)}\, dy \Big)  e^{ikx} 
\\ & 
= \sum_{n=0}^\a \sum_{k \in \Z} \Big( \int_\T (\pa_y^n u)(y)\, Q_\a^{(n)}(y,k) \, e^{-i\phi(y,k)}\, dy \Big)  e^{ikx} 
\\ & 
= \sum_{n=0}^\a E_n^* \, (\pa_x^n u)(x)
\end{align*}
where $E_n^*$ is the $L^2(\T)$-adjoint operator of the \textsc{fio} $E_n$ having phase $\phi$ and amplitude $a_n$,
\[
E_n v(x) := \sum_{k \in \Z} \hat v_k \, a_n(x,k)\, e^{i\phi(x,k)}\,,
\]
\[
a_n(x,k) := Q_\a^{(n)}(x,k) 
= \frac{1}{(h')^{2\a}}\, \sum_{\nu \in \mV_{\a,n}} C(\nu) (\pa^{\nu_1} h)(x,k) \ldots (\pa^{\nu_\a} h)(x,k),
\]
and $\nu \in \mV_{\a,n}$ means 
\[
\nu = (\nu_1, \ldots, \nu_\a) \in \Z^\a, \quad \nu_i \geq 1, \quad 
\nu_1+\ldots +\nu_\a = 2\a-n.
\]
Here, as above, $h'(x,k) = 1 + f(k) k\inv \b'(x)$ so 
$1/h'-1=F(\b')$ for some smooth function $F$ vanishing at the origin
and hence  
$\| 1/h' \|_2 \leq 2 $, provided that $\|\b\|_{3}$ is small enough.
Then, with similar calculations as above, one proves that 
\[
\| a_n(\cdot\,,k) \|_2 \leq C(\a) \sum_{\nu \in \mV_{\a,n}} \| h' \|_{\nu_1+1} \ldots \| h' \|_{\nu_\a+1} \,.
\]
Therefore, by Lemma \ref{lemma:L2 bound E}, the 
operator norm 
$\| E_n \|_{0,0} := \sup \{ \| E_n u\|_0 : \| u \|_0 = 1 \} $
satisfies 
\[
\| E_n \|_{0,0} \leq C(\a) \sum_{\nu \in \mV_{\a,n}} \| h' \|_{\nu_1+1} \ldots \| h' \|_{\nu_\a+1} \,.
\]
Since $\| E_n^* \|_{0,0} = \| E_n \|_{0,0}$, 
\[
\| E_n^* \, \pa_x^n u \|_0 \leq C(\a) \sum_{\nu \in \mV_{\a,n}} \| h' \|_{\nu_1+1} \ldots \| h' \|_{\nu_\a+1} \,\| u \|_n
\]
and we conclude the proof using 
the interpolation like above and summing for $n=0,\ldots,\a$.
\end{proof}

We have proved that $A^*_\T A$ is invertible, therefore $A^*_\T$ is surjective and $A$ is injective. To prove that $A$ is invertible, we need the invertibility of $A A^*_\T$. We prove it by studying $A A^*_\R$. 

Recall that we consider a phase function $\phi(x,\xi)=x\xi+f(\xi)\beta(x)$ where 
$\b(x)$ is a smooth real-valued periodic function and  
$f$ is a $C^\infty$ function such that 
$f(0) = 0$ and $\|f'\|_{L^\infty} <+\infty$. 
Hereafter, we make a further assumption.

\begin{assu}
Assume that $f\colon \xR\rightarrow \xR$ is a $C^\infty$ function such that 
\[
f(\xi) = 0 \quad \forall |\xi| \leq 1/4; \qquad  
f(\xi) = |\xi|^r \quad \forall |\xi|\geq 1,
\]
where $0 < r < 1$ is a real number.
\end{assu}

We shall apply the following results with $r = 1/2$. 

\begin{lemma}
\label{lemma:AA*T}
Assume that $f$ satisfies the above assumption. 
There exist constants $C_1, \d_1 > 0$, with $C_1 \d_1 \leq 1/4$, such that, if $\| \b \|_3 \leq \d_1$, then $A A^*_\T : L^2(\T) \to L^2(\T)$ is invertible, with operator norm 
\[
\| A A^*_\T u \|_0 \leq 2 \| u \|_0, \quad 
\| (A A^*_\T)^{-1} u \|_0 \leq 2 \| u \|_0\,.
\]
More precisely, 
\[
\| (A A^*_\T - I)u \|_0 \leq C_1 \|\b\|_3 \|u\|_0\,\leq \frac14\, \| u \|_0\,.
\]
\end{lemma}

\begin{proof} The proof is split in several steps. 

\medskip

\textsc{Step 1.} Observe the following fact.
Let $\b \in H^{m+1}(\T)$ for some integer $m\geq  2$, with $L \|\b'\|_{L^\infty} \leq \d$ and $\| \b \|_{m+1} \leq K$, $K>0$. Then, for every $\psi \in C^\infty_0(\R)$,
\begin{equation} \label{basic integrability}
|(\psi, w_\xi)_{L^2(\R)}| = \Big| \int_\R \psi(y)\,e^{-i \phi(y,\xi)}\,dy \big| \leq \frac{C(\psi,K)}{1+|\xi|^m}\, \quad \forall \xi \in \R
\end{equation}
for some constant $C(\psi,K)$ which depends on $\|\psi\|_{W^{m,\infty}}$ and $K$.
Indeed, integrating by parts gives 
\[
\int \psi(y)\,e^{-i \phi(y,\xi)}\,dy 
= \frac{1}{i\xi}\, \int e^{-i \phi(y,\xi)}\,\mL\psi(y)\,dy,
\]
where 
\[
\mL\psi = \pa_y (v \psi), \quad 
v = \Big( 1+\frac{f(\xi)}{\xi}\,\b'(y) \Big)^{-1}\,.
\]
To gain a factor $\xi^m$ at the denominator, integrate by parts $m$ times.

\medskip

\textsc{Step 2.} To prove the invertibility of $A A^*_\T$, it is 
convenient to study 
$A A^*_\R$ and pass from the real line to the torus in a further step. 
$A A^*_\R$ is given by \eqref{AA*R}, namely
\begin{equation} \label{AA*R estesa}
A A^*_\R\,g(x) = \int_\R \Big( \int_\R g(y)\,e^{-i \phi(y,\xi)}\, dy \Big) \, e^{i \phi(x,\xi)} \,d\xi.
\end{equation}
For $g \in C^\infty_0(\R)$ and $\b \in H^3(\T)$, with $L \| \b' \|_{L^\infty} \leq 1/2$, the integral is finite by \eqref{basic integrability}.
Now we want to change the integration variable $\xi$: this is the reason for which we consider real frequencies $\xi \in \R$ and not only integers $k \in \Z$. 

Fix $\d_0$ small enough 
so that $|c f'(\xi)| \leq 1/4$ for all $\xi \in \R$, provided that $|c| \leq \d_0$. 
For each $|c| \leq \d_0$, the map 
\[
\xi \mapsto \g(\xi,c) := \xi + c f(\xi) 
\]
is a diffeomorphism of $\R$ because $3/4 \leq \pa_\xi \g(\xi,c) \leq 5/4$, 
and $\g$ is $C^\infty$ in both the variables $(\xi,c)$. 
Therefore, by the implicit function theorem, the inverse map 
$\mu(\th,c)$,
\[
\xi = \mu(\th,c) \quad \Leftrightarrow \quad \th = \g(\xi,c),
\]
satisfies $4/5 \leq \pa_\th \mu(\th,c) \leq 4/3$ and is $C^\infty$ in both $(\th,c)$. Let 
\[
h(\th,c) := \mu(\th,c) - \th.
\]
Thus $h \in C^\infty$, $|\pa_\th h(\th,c)| \leq 1/3$ for all $\th \in \R$, $|c| \leq \d_0$,
\[
h(\th,c) + c f(\xi) = 0, \quad 
\pa_\th h(\th,c) = -\frac{c f'(\xi)}{1+c f'(\xi)} 
\quad \text{where} \ \xi = \th + h(\th,c), 
\]
\[
\pa_c h(\th,c) = -\frac{f(\xi)}{1+c f'(\xi)}\,, \quad 
\pa_{\th c} h(\th,c) = \frac{c f(\xi) f''(\xi)}{[1+c f'(\xi)]^3}\,  - \frac{f'(\xi)}{[1+c f'(\xi)]^2}\,.
\]
Then one proves by induction that, for any $m\ge 2$, 
$\partial^{m}_\th\partial_c h(\th,c)$ is a linear combination of terms
$$
\frac{1}{[1+c f'(\xi)]^N}\prod_{j=1}^pf^{(n_j)}(\xi)
$$
with the property that the largest exponent $\max_{1\le j\le p} n_j$ is greater or equal to $2$. 
As a result, 
\begin{equation} \label{stimine derivate hcc ecc}
\begin{aligned}
& \| \pa_{\th c} h(\cdot,c)\|_{L^\infty(\R,d\th)}
\leq C\,, \\
&\|\pa_{\th}^m\pa_c h(\cdot,c)\|_{L^1(\R,d\th)}\leq C, \quad m=2,3,4,
\end{aligned}
\end{equation}
for all $|c| \leq \d_0$, where the constant $C>0$ depends only on $f$. 
Now 
let 
\[
\tilde \b(x,y) := \frac{\b(x) - \b(y)}{x-y} \, \quad \text{for } \ x \neq y, \qquad 
\tilde \b(x,x) := \b'(x),
\]
and let $\th$ be the new frequency variable, 
\begin{equation} \label{change}
\th = \xi + f(\xi) \tilde \b(x,y) , 
\quad 
\xi = \th + h(\th, \tilde \b(x,y)) = \mu(\th, \tilde \b(x,y)).
\end{equation}
The order of integration in \eqref{AA*R estesa} cannot be changed because the double integral does not converge absolutely. We overcome this problem as usual, 
fixing $\psi \in C^\infty_0(\R)$ with $\psi(0) = 1$ and noting that
\[
A A^*_\R \, g (x) = \lim_{\e \to 0} I_\e(x), \quad 
I_\e(x) := \int_\R \psi(\e \xi) 
\Big( \int_\R g(y)\,e^{-i \phi(y,\xi)}\, dy \Big) \, e^{i \phi(x,\xi)} \,d\xi
\]
by the dominated convergence theorem. 
It is found that
\[
A A^*_\R \, g (x) = \int_\R 
\Big( \int_\R g(y) \, e^{-i \th y}\, (1+q(x,y,\th)) \,dy \Big) 
e^{i \th x} \, d\th 
\]
with
\[ 
q(x,y,\th) := \pa_\th h(\th,\tilde \b(x,y)).
\] 
Namely $A A^*_\R$ is the sum $(I+Q)$ of the identity map and the pseudo-differential operator $Q$ of compound symbol $q(x,y,\th)$.

\medskip

\textsc{Step 3.} First order Taylor's formula in the $y$ variable at $y=x$ gives 
\begin{align*}
q(x,y,\th) & = q_0(x,\th) + q_1(x,y,\th), \\
q_0(x,\th) & := q(x,x,\th) = (\pa_\th h)(\th,\b'(x)), \\
q_1(x,y,\th) & := \int_0^1 (\pa_y q)\big( x, x+s(y-x), \th \big) \, ds\, (y-x)\,.
\end{align*}
Split $Q = Q_0 + Q_1$ accordingly. 
Since $q_0(x,\th)$ does not depend on $y$, 
\[
Q_0 g(x) 
= \int_\R \hat g(\th) \, q_0(x,\th) \, e^{i \th x}\, d\th.
\]
$q_0(x,\th)$ is $2\p$-periodic in $x$, for every $\th \in \R$, because $\b'(x)$ is periodic. 
By \eqref{stimine derivate hcc ecc}, 
\[
|q_0(x,\th)| \leq C |\b'(x)| \,, \quad 
|\pa_x q_0(x,\th)| \leq C |\b''(x)|\,,
\]
whence 
\begin{equation} \label{stima q0}
\| q_0(\cdot\,,\th) \|_1  \leq C \|\b\|_2\, \quad 
\forall \th \in \R
\end{equation}
for some constant $C>0$ 
(where, remember, $\| \cdot \|_m$ is the $H^m(\T)$ norm).

\smallbreak
We next study $Q_1$. 
If the order of integration in $Q_1$ can be changed, then an integration by parts in the $\th$ variable gives
\[
Q_1 g(x) = -i \int_0^1 \int_\R \int_\R (\pa_{\th y} q)\big( x, x+s(y-x), \th \big) \, e^{i\th(x-y)}\, g(y)\,d\th\,dy \,ds \,.
\]
Since 
$$
|\pa_{\th z} q (x,z,\th)| 
= |(\pa_{\th\th c} h)(\th,\tilde \b(x,z))|\,|\pa_z \tilde \b(x,z)| \, ,
$$
by \eqref{stimine derivate hcc ecc}, 
\begin{equation} \label{under treshold}
\|\pa_{\th z} q (x,z,\cdot)\|_{L^1(\R,d\th)} 
\leq C \| \b''\|_{L^\infty} 
\end{equation}
for all $x,z$ in $\R$. Therefore the triple integral converges absolutely ($g \in C^\infty_0(\R)$ by assumption), and one can prove that the order of integration can actually be changed (introduce a cut-off function $\psi(\e \th)$, with $\psi(0)=1$, $\psi \in C^\infty_0(\R)$, and pass to the limit as $\e \to 0$ like above). 
Denote
\[
a(x,y,\th) := -i \int_0^1 (\pa_{\th z} q)\big( x, x+s(y-x), \th \big) \, ds \,,
\]
so that 
\[
Q_1 g(x) = \int_{\R^2} a(x,y,\th)\, g(y)\, e^{i\th(x-y)}\, dy\, d\th\,.
\]
By \eqref{under treshold}, for all $x$ one has 
\begin{equation}  \label{2-5-2014}
| Q_1 g(x) | \leq \int_\R |g(y)| \Big( \int_\R |a(x,y,\th)| \, d\th \Big) dy 
\leq C \| \b \|_3 \| g \|_{L^1(\R)}.
\end{equation}
Denote $Tg(x) := x g(x)$. The commutator 
$[T,Q_1] = T Q_1 - Q_1 T$ is the same integral as $Q_1$ with an additional factor $(x-y)$,
\[
[T,Q_1]g(x) = \int_{\R^2} a(x,y,\th)\, (x-y) \, e^{i\th(x-y)}\, g(y)\,dy\,d\th\,.
\]
Integrating by parts in $\th$ again, 
\[
[T,Q_1] g(x) = i \int_{\R^2} \pa_\th a(x,y,\th)\, e^{i\th(x-y)}\, g(y)\,dy\,d\th\,,
\]
and $\pa_{\th \th z} q (x,z,\th)$ 
satisfies the same estimate \eqref{under treshold} as $\pa_{\th z} q$. 
Note that no other derivatives in $y$ are involved in this argument, therefore $\b$ does not increase its derivation order. 
Repeat the same integration by parts twice: write 
\[
x^2 = [(x-y)+y]^2
= (x-y)^2 + 2 (x-y) y + y^2 \,,
\]
so that 
\[
x^2 Q_1 g(x) = \int_{\R^2} \Big( i^2 (\pa_\th^2 a) \, g + 2i(\pa_\th a) 
(T g) + a (T^2 g) \Big) \, e^{i\th(x-y)} \, dy\, d\th\,.
\]
Every $\pa_\th^m a(x,y,\th)$, $m=0,1,2$, satisfies an estimate like \eqref{under treshold}, namely 
\[
\|\pa_\th^m a(x,y,\cdot)\|_{L^1(\R,d\th)} \leq C \| \b \|_3\,, \quad  m=0,1,2,
\]
for some constant $C>0$. Now assume that $g(y) = 0$ for all $|y| > 2\p$. Then, by H\"older's inequality,
\[
\int_\R |y^m g(y)|\,dy \leq C \| g \|_{L^2(\R)} \quad \forall m=0,1,2.
\]
Thus we have 
$|x^2 Q_1 g(x)| \leq C \|\b\|_3 \, \| g \|_{L^2(\R)}$ 
and, using also \eqref{2-5-2014}, 
\[
|Q_1 g(x)| \leq \frac{C \|\b\|_3 \, \| g \|_{L^2(\R)}}{1+|x|^2}\,.
\]
Hence, provided that $g(y) = 0$ for all $|y| > 2\p$, 
both $Q_1 g$ and $TQ_1 g$ are in $L^2(\T)$, with 
\begin{equation} \label{stima Q1 TQ1}
\| Q_1 g \|_{L^2(\R)} + \| T Q_1 g \|_{L^2(\R)} 
\leq 2 \| (1+x^2)^{1/2} Q_1 g(x) \|_{L^2(\R)}  
\leq C \|\b\|_3 \, \| g \|_{L^2(\R)} \,.
\end{equation}

\textsc{Step 4.} Let $\mP$ be the ``periodization'' map defined in the Appendix. We observe that 
\begin{equation}  \label{P AA*R = AA*T P}
\mP (AA^*_\R \,\psi) = AA^*_\T (\mP \psi) 
\quad \forall \psi \in C^\infty_0(\R). 
\end{equation}
To prove \eqref{P AA*R = AA*T P}, fix $\psi \in C^\infty_0(\R)$ and calculate
\begin{align*}
\mP (AA^*_\R \,\psi)(x) 
& = \sum_{j \in \Z} (AA^*_\R \,\psi)(x + 2\p j)  \\
& = \sum_{j \in \Z} \int_\R (\psi,w_\xi)_{L^2(\R)} \, w_\xi(x+2\p j) d\xi \\
& = \sum_{j \in \Z} \int_\R (\psi,w_\xi)_{L^2(\R)} \, w_\xi(x) \, e^{i 2\p j \xi} \, d\xi  
\end{align*}
because $\phi(x+2\p j,\xi) = \phi(x,\xi) + 2\p j \xi$. 
For each fixed $x \in \R$, by \eqref{basic integrability}, the map $\xi \mapsto g(\xi) := (\psi,w_\xi)_{L^2(\R)} \, w_\xi(x)$ satisfies 
\[
(1+|\xi|^2)(|g(\xi)| + |g'(\xi)|) \leq C \quad \forall \xi \in \R, 
\]
for some constant $C$. Then, by Lemma \ref{lemma:PSF} and \eqref{scalar products}, 
\begin{align*}
\sum_{j \in \Z} \int_\R g(\xi) \, e^{i 2\p j \xi} \, d\xi 
& = \sum_{k \in \Z} g(k)  \\ 
& = \sum_{k \in \Z} (\psi,w_k)_{L^2(\R)} \, w_k(x) \\
& = \sum_{k \in \Z} (\mP \psi,w_k)_{L^2(\T)} \, w_k(x) \\
& = A A^*_\T \, (\mP \psi)(x). 
\end{align*}

\medskip

\textsc{Step 5.} From the two previous steps, since $\mP \mS = I$ on $L^2(\T)$, 
\[
A A^*_\T 
= A A^*_\T \, \mP \, \mS 
= \mP \, A A^*_\R \, \mS 
= \mP \, (I + Q_0 + Q_1) \, \mS 
= I + \mP \,Q_0\, \mS + \mP \,Q_1\, \mS \,.
\]
For $u \in C^\infty(\T)$, by Lemma \ref{lemma:PSF},
\begin{align*}
\mP \,Q_0\, \mS u(x) 
& = \sum_{k \in \Z} (Q_0 \mS u)(x+2\p k) \\
& = \sum_{k \in \Z} \int_\R \widehat{(\mS u)}(\xi)\, q_0(x,\xi)\,e^{i\xi(x+2\p k)} \, d\xi \\
& \stackrel{\eqref{eq:delta Z}}{=} 
\sum_{k \in \Z} \widehat{(\mS u)}(k)\, q_0(x,k) \, e^{ikx} \\
& \stackrel{\eqref{n1110}}{=} \sum_{k \in \Z} \hat u_k \, q_0(x,k) \, e^{ikx} \,.
\end{align*}
It is possible to use \eqref{eq:delta Z} here because $|q_0(x,\xi)|+|\pa_\xi q_0(x,\xi)| \leq C$ for all $x,\xi$, for some $C>0$, and $\widehat{\mS u}(\xi)$ rapidly decreases as $\mS u$ is compactly supported. 
Therefore
\[
\| \mP \,Q_0\, \mS u \|_0 \leq C \| \b \|_2 \, \| u \|_0 
\]
by \eqref{stima q0} and Lemma \ref{lemma:reg pseudo-diff}($ii$), and, by density, this holds for all $u \in L^2(\T)$.

By Lemma \ref{lemma:T} and \eqref{stima Q1 TQ1},
\begin{align*}
\| \mP \,Q_1 \mS u \|_0 
& \leq \| Q_1 \mS u \|_{L^2(\R)} + \| T Q_1 \mS u \|_{L^2(\R)} 
\leq C \| \b \|_3 \, \| \mS u \|_{L^2(\R)}  
\leq C \| \b \|_3 \, \| u \|_0
\end{align*}
for some constant $C>0$.  
We have proved that 
\[
A A^*_\T = I + B, \quad 
\| Bu \|_0 \leq C \|\b\|_3 \|u\|_0,
\]
where $B := \mP \,Q_0\, \mS + \mP \,Q_1\, \mS$. 
Therefore, by Neumann series, $A A^*_\T : L^2(\T) \to L^2(\T)$ is invertible, with operator norm $\leq 2$, for $\| \b \|_3 \leq \d_1$, for some constant $\d_1$.
\end{proof}

Collecting the previous estimates, and taking the worst $\|\b\|$ among all, we have the following

\begin{lemma}
\label{lemma:Ainv}
There exist universal constants $C,\d_1>0$ such that, 
if $\b \in H^3(\T)$, $\| \b \|_3 \leq \d_1$, then $A, A^*_\T : L^2(\T) \to L^2(\T)$ are invertible operators, with 
\[
\| A u \|_0 + \| A\inv u \|_0 + \| A^*_\T\, u \|_0 + 
\| (A^*_\T)\inv u \|_0 
\leq C \, \| u \|_0 \,.
\]
If, in addition, $\b \in H^{\a+2}(\T)$, $\a \geq 1$ integer, then 
$A, A^*_\T : H^\a(\T) \to H^\a(\T)$ are invertible, with 
\[
\| A u \|_\a + \| A\inv u \|_\a + \| A^*_\T\, u \|_\a + 
\| (A^*_\T)\inv u \|_\a 
\leq C(\a) \, \Big( \| u \|_\a + \| \b \|_{\a+2} \, \| u \|_1 \Big),
\]
where $C(\a)>0$ is a constant that depends only on $\a$. 
\end{lemma}

\begin{proof} 
Both $A A^*_\T$ and $A^*_\T\,A$ are invertible on $L^2(\T)$, therefore $A$ and $A^*$ are also invertible. The estimates for $A\inv$ and $(A^*_\T\,)^{-1}$ come from the equalities
\[
A\inv = (A^*_\T A)\inv A^*_\T \,, \quad 
(A^*_\T)\inv = A (A^*_\T A)\inv \,. \qedhere
\]
\end{proof}

\subsection{With amplitude}

Now let $E$ be the operator with phase $\phi(x,k)$ and amplitude $a(x,k)$, namely
\[
E u(x) = \sum_{k \in \Z} \hat u_k \, a(x,k) e^{i\phi(x,k)}\,,
\]
where $a(x,k)$ is $2\p$-periodic in $x$, and $\phi$ is like above. 
We are interested to the case when the amplitude is of order zero in $k$ and is a perturbation of 1,  
\[
a(x,k) = 1 + b(x,k)\,.
\]
Denote $|b|_s := \sup_{k \in \Z} \| b(\cdot\,,k) \|_s$.

\begin{lemma}
\label{lemma:Einv} 
There exist a universal constant $\d > 0$ with the following properties. 
Let $\b \in H^3(\T)$ and $b(\cdot\,,k) \in H^3(\T)$ for all $k \in \Z$. 
If 
\[
\| \b \|_3 + |b|_3 \leq \d\,,
\]
then $E$ and $E^*_\T$ are invertible from $L^2(\T)$ onto itself, with 
\[
\| E u \|_0 + \| E\inv u \|_0 + \| E^*_\T\, u \|_0 + 
\| (E^*_\T)\inv u \|_0 \leq C \, \| u \|_0 \,,
\]
where $C>0$ is a universal constant. 

If, in addition, $\a \geq 1$ is an integer, $\b \in H^{\a+2}(\T)$ and $b(\cdot\,,k) \in H^{\a+2}(\T)$ for all $k \in \Z$, then 
\begin{multline*}
\| E u \|_\a + \| E\inv u \|_\a + \| E^*_\T\, u \|_\a + 
\| (E^*_\T)\inv u \|_\a \\
\leq C(\a) \, \Big( \| u \|_\a + (|b|_{\a+2} + \| \b\|_{\a+2}) \| u \|_1 \Big) 
\end{multline*}
where $C(\a)>0$ depends only on $\a$.
\end{lemma}

\begin{proof} 
\textsc{Step 1.} Let $B$ be the operator with amplitude $b$ and phase $\phi$, so that $E = A+B$. By Lemma \ref{lemma:L2 bound E}, $\| Bu \|_0 \leq C |b|_2 \|u\|_0$,
therefore, using Lemma~\ref{lemma:Ainv},
\[
\| A\inv B u \|_0 \leq C |b|_2 \| u \|_0
\]
for some universal constant $C>0$, provided that $\delta$ is small enough. Then 
$E=A(I+A\inv B)$ is invertible in $L^2(\T)$ by Neumann series. 
Analogous proof for $E^*$.

\medskip

\textsc{Step 2.} The matrix $L := E^*_\T E$ is given by Remark \ref{remark:E*E}, and it is
\[
E^* E = A^* A + A^* B + B^* A + B^* B.
\]
On the diagonal, 
\[
L_k^k = \int_\T |1+b(x,k)|^2 \, dx 
\geq \frac12
\]
if $|b(x,k)| \leq 1/2$ for all $x \in \T$, $k \in \Z$. 
Off-diagonal, 
\begin{align*}
L_j^k & = M_j^k 
+ \int_\T \Big( \overline{b(x,k)} 
+ b(x,j) 
+ b(x,j) \, \overline{b(x,k)} \Big) \, e^{i[\phi(x,j) - \phi(x,k)]}\,dx ,
\end{align*}
where the matrix $M_j^k$ is defined in \eqref{n104}.
Using \eqref{stima Mkj} for the first term and Lemma \ref{lemma:non-stationary phase} for the other three terms, 
\begin{equation} \label{L off-diagonal}
|L_j^k| \leq \frac{C(\a,K)}{|k-j|^{\a}}\, \big( |b|_\a + \| \b \|_{\a+1} \big), \quad k \neq j,
\end{equation}
for $\| \b \|_2 \leq K$ and $|b|_1 \leq K$. 
Let $D$ be the Fourier multiplier $e_k \mapsto L_k^k e_k$ and $R$ the off-diagonal part $R=L-D$. For $\a=2$ in \eqref{L off-diagonal}, 
\[
\| D\inv R u \|_0 \leq C(K)\, (\|\b\|_3 + |b|_2) \, \| u \|_0,
\]
therefore $L$ is invertible in $L^2(\T)$ if 
\begin{equation} \label{smallness beta b}
\| \b \|_3 + |b|_2 \leq \d
\end{equation}
for some universal $\d>0$ (for example, fix $K=1$ first, then fix $\d$ sufficiently small). 
For $s \geq 1$ integer, if \eqref{smallness beta b} holds,
\[
\| Ru \|_s \leq C \| u \|_s + C(s) (\| \b \|_{s+2} + |b|_{s+1}) \| u \|_1  
\]
by \eqref{L off-diagonal} and usual calculations for off-diagonal matrices. This gives the tame estimate for $L$, and, by Neumann series, also for $L\inv$, namely: if \eqref{smallness beta b} holds, then
\begin{equation} \label{tame E*E}
\| E^*_\T\,E u \|_s + 
\| (E^*_\T\,E)\inv u \|_s \leq 
C \| u \|_s + C(s) (\| \b \|_{s+2} + |b|_{s+1}) \| u \|_1 
\end{equation}
where $C>0$ is a universal constant and $C(s)>0$ depends on $s$. 

\medskip

\textsc{Step 3.} $Bu(x)$ is given by \eqref{E trucco} where $F_j$ is the Fourier multiplier $F_j e_k = \hat b_j(k) \, e_k$.  
Integrating by parts, for $j \neq 0$,
\[
|\hat b_j(k)| = \Big| \int_\T b(x,k) \, e^{-ijx}\,dx \Big| 
= \frac{1}{|j|^m} \, \Big| \int_\T \pa_x^m b(x,k) \, e^{-ijx}\,dx \Big| 
\leq \frac{\| \pa_x^m b(\cdot\,,k) \|_0}{|j|^m} \, 
\]
for all $k \in \Z$. Hence
\begin{equation}\label{n1026}
\| F_j u \|_\a \leq \frac{|b|_m}{\la j \ra^m}\, \| u \|_\a \,, \qquad \forall j \in \Z\,.
\end{equation}
Now use the tame product rule in Sobolev spaces 
$\| uv\|_\a \le K \| u\|_{L^\infty}\| v\|_\a+K\| v\|_{L^\infty}\| v\|_\a$ 
together with the Sobolev embedding $H^1(\T)\subset L^\infty(\T)$ to deduce, 
\begin{align*} 
\| Bu \|_\a  \leq \sum_{j \in \Z} \| (A F_j u) \, e_j \|_\a 
\leq C(\a) \sum_{j \in \Z} \| A F_j u \|_\a  + \| A F_j u \|_1  \, \la j \ra^\a.
\end{align*}
Thus, applying Lemma~\ref{lemma:Ainv} and 
using \eqref{n1026} with either $m = 2$ or $m = \a + 3$,
\begin{align*} 
\| Bu \|_\a 
& \leq C(\a) \sum_{j \in \Z} \Big( \| F_j u \|_\a + \|\b\|_{\a+2} \|F_j u \|_1 \Big)
+ \| F_j u \|_1  \, \la j\ra^\a \\
& \leq C(\a) \sum_{j \in \Z} \Big( \| u \|_\a + \|\b\|_{\a+2} \| u \|_1 \Big)\, \frac{|b|_{2}}{\la j\ra^2} \,  
+ \| u \|_1  \, \frac{|b|_{\a+2}}{\la j\ra^2} \\
& \leq C(\a) \Big( \| u \|_\a + \|\b\|_{\a+2}\| u \|_1 \Big)\, |b|_{2} \,  + \| u \|_1  \, |b|_{\a+2} 
\end{align*}
if \eqref{smallness beta b} holds. The sum with the analogous estimate for $A$ gives
\[
\| Eu \|_\a \leq C(\a,K) \, \Big( \| u \|_\a + (|b|_{\a+2} + \| \b\|_{\a+2}) \| u \|_1 \Big)
\]
for $\| \b \|_3 \leq K$, $|b|_2 \leq K$. 
For $E^*_\T$, note that 
\[
E^*_\T u = \sum_{j \in \Z} F^*_j A^*_\T (e_{-j} \, u) 
\]
and repeat the same argument. 
\end{proof}

\subsection{Composition formula}  \label{sec:composition}

Consider the periodic FIO of amplitude $a$ and phase function $\phi$, 
\begin{equation} \label{def A FIO with amplitude}
Au(x) = \sum_{k \in \Z} \hat u_k \, a(x,k) \, e^{i \phi(x,k)}, \qquad 
\phi(x,k) = kx + f(k) \b(x),
\end{equation}
where $f(k) = |k|^{1/2}$ for all $k \in \Z$.

\begin{lemma} \label{lemma:composition formula} 
Let $A$ be the operator \eqref{def A FIO with amplitude}, 
with $\|\b\|_{W^{1,\infty}} \leq 1/4$ and $\| \b \|_2 \leq 1/2$. 
Let 
\[
r,m,s_0 \in \R, \quad 
m \geq 0, \quad
s_0 > 1/2, \quad 
N \in \N, \quad 
N \geq 2(m + r + 1) + s_0. 
\] 
Then 
\[
|D_x|^r A u = \sum_{\a=0}^{N-1} B_\a u + R_N u,
\]
where 
\[
B_\a u(x) 
= \binom{r}{\a} \,  \sum_{k \in \Z^*} |k|^{r-\a} \, (- i \, \sgn \, k)^\a \, \hat u_k \,  e^{ikx} \, \pa_x^\a \big\{ a(x,k) e^{i f(k) \b(x)} \big\},
\]
namely 
\[
B_\a u = F_\a \Dx^{r-\a} \mH^\a u,
\quad 
F_\a v(x) := 
\binom{r}{\a} \, 
\sum_{k \in \Z} \hat v_k \, e^{ikx} \, \pa_x^\a \big\{ a(x,k) e^{i f(k) \b(x)} \big\},
\]
and 
$\binom{r}{\a} \, := \frac{r (r-1) \ldots (r-\a+1)}{\a!}$. 

For every $s \geq s_0$, the remainder satisfies
\begin{equation} \label{final tame estimate}
\| R_N \Dx^m u \|_s 
\leq 
C(s) \Big\{ \mK_{2(m+r+s_0+1)} \, \| u \|_s 
+ \mK_{s + N + m + 2} \, \| u \|_{s_0} \Big\},
\end{equation}
where $\mK_\mu := \| a \|_{\mu} + \| a \|_1  \| \b \|_{\mu + 1}$ for $\mu \geq 0$ 
and $\| a \|_\mu := \sup_{k \in \Z} \| a(\cdot,k) \|_\mu$. 

Moreover,
\[
\mH |D_x|^r A u = \sum_{\a=0}^{N-1} B_\a \mH u + \tilde R_N u,
\]
where $\tilde R_N$ satisfies the same estimate \eqref{final tame estimate} as $R_N$.
\end{lemma}

\begin{remark} 
In particular, $B_0 u = A |D_x|^r u$.  $B_\a$ is of order $r - (\a/2)$.
\end{remark}

\begin{proof}
Denote by $(\hat z_j(k))_{j\in \Z}$ the Fourier coefficients of the periodic function 
$x \mapsto a(x,k) \, e^{i f(k) \b(x)}$ and consider a $C^{\infty}$ function $g \colon \R \to \R$ such that 
$g(\xi)= |\xi|^r$ for all $|\xi| \geq 2/3$, 
and $g(\xi)= 0$ for $|\xi| \leq 1/3$. 
If $\hat u_k$ are the Fourier coefficients of a periodic function $u(x)$, then
\begin{equation} \label{DAu}
\Dx^r A u(x) 
= \sum_{k,j \in \Z} \hat u_k \, \hat z_j(k) \, g(k+j)\, e^{i(k+j)x}.
\end{equation}
 Taylor's formula gives, for some $t \in [0,1]$,
\[
g(k+j) = \sum_{\a=0}^{N-1} \frac{1}{\a!} \, g^{(\a)}(k) \, j^\a + r_N(k,j), 
\qquad
 r_N(k,j) = \frac{1}{N!} \,g^{(N)}(k+tj) \, j^N. 
\]
Accordingly, \eqref{DAu} is split into
$
\Dx^r A u = \sum_{\a = 0}^{N-1} B_\a u + R_N u$. Hence
\begin{align*} 
B_\a u(x) 
& := \sum_{k,j \in \Z} \hat u_k \, \hat z_j(k) \, \frac{1}{\a!} \, g^{(\a)}(k)\, j^\a \, e^{i(k+j)x} 
\\ 
& = \sum_{k \in \Z} \frac{1}{i^\a \a!} \, g^{(\a)}(k)\, \hat u_k \,  e^{ikx} 
\Big( \sum_{j \in \Z} \hat z_j(k) \, (ij)^\a \, e^{ijx} \Big)
\\  
& = \sum_{k \in \Z^*} \frac{1}{i^\a \a!} \, r (r-1) \ldots (r-\a+1) \, |k|^{r-\a} \, (\sgn \, k)^\a \, \hat u_k \,  e^{ikx} \, \pa_x^\a \big\{ a(x,k) e^{i f(k) \b(x)} \big\} 
\\  
& = \sum_{k \in \Z^*} \binom{r}{\a} \, |k|^{r-\a} \, (- i \, \sgn \, k)^\a \, \hat u_k \,  e^{ikx} \, \pa_x^\a \big\{ a(x,k) e^{i f(k) \b(x)} \big\}.
\end{align*}
Since $|k|^{r-\a} \, (- i \, \sgn \, k)^\a \, \hat u_k$ is the $k$-th Fourier coefficient of $\Dx^{r-\a} \mH^\a u(x)$, one can also write 
\[
B_\a u = F_\a \Dx^{r-\a} \mH^\a u,
\] 
with $F_\a$ defined in the statement of the Lemma. 
$B_\a$ has order $r - \a/2$ since, as power of $k$, the maximum order 
of $\pa_x^\a \big\{ a(x,k) e^{i f(k) \b(x)} \big\}$ is $|k|^{\a/2}$.

\medskip

It remains to estimate the remainder 
\[
R_N u(x) := \sum_{k,j \in \Z} \hat u_k \, \hat z_j(k) \, r_N(k,j)\, e^{i(k+j)x}.
\]
For $N \geq r$, the $N$-th derivative of $g$ satisfies 
$|g^{(N)}(\xi)| \leq C_{r,N} \la \xi \ra^{r-N} \leq C_{r,N}$ for all $\xi$ in $\R$. 
In particular, for $|j| \leq \frac12|k|$ and $t \in [0,1]$, 
we have $|k+tj| \geq \frac12\,|k|$ and hence 
\begin{equation} \label{stima rN(k,j) j small}
|r_N(k,j)| 
\leq C_{r,N} \, \la k \ra^{r-N} |j|^N 
\qquad \forall |j| \leq \frac12\, |k|,
\end{equation}
and, in general, 
\begin{equation} \label{stima rN(k,j) generica}
|r_N(k,j)| \leq C_{r,N} \, |j|^N \quad \forall j \in \Z.
\end{equation}

We split $R_N$ into 2 components, 
$R_N = \mR_1 + \mR_2$, 
$\mR_1$ for low frequencies and $\mR_2$ for high frequencies:
\[
\ba
\mR_1 u(x) 
&:= \sum_{|j| \leq \frac12 |k|} \hat u_k \, \hat z_j(k) \, r_N(k,j)\, e^{i(k+j)x},
\\
\mR_2 u(x) 
&:= \sum_{|j| > \frac12 |k|} \hat u_k \, \hat z_j(k) \, r_N(k,j)\, e^{i(k+j)x}.
\ea
\]

Let us analyze the coefficients $\hat z_j(k)$. 
First of all, for $j \neq 0$, $\a \geq 0$,
\[
|\hat z_j(k)| 
\leq \frac{1}{|j|^\a} \, \big\| a(\cdot,k) \, e^{i f(k) \b(\cdot)} \big\|_\a . 
\]
Now, for $\a \in \N$, 
\[
\pa_x^\a \Big( a(x,k) e^{i f(k) \b(x)} \Big) 
= e^{i f(k) \b(x)} \, \sum_{\a_1 + \a_2 = \a} C(\a_1, \a_2) \, (\pa_x^{\a_1} a)(x,k) \, P_{\a_2}(x,k),
\]
where $P_{\a}$ is defined, as usual, by 
$\pa_x^{\a} \big( e^{i f(k) \b(x)} \big) =  e^{i f(k) \b(x)} \, P_\a(x,k)$. Hence 
it follows from the estimate \eqref{stima Palpha k beta} in the Appendix and an interpolation argument that,
\begin{equation}  \label{stima zjk generica}
|\hat z_j(k)| 
\leq \frac{C_\a |k|^{\a/2} }{|j|^\a} \, 
\big( \| a \|_\a + \| a \|_0 | \b |_\a \big)
\quad \forall j \neq 0,
\end{equation}
where $\| a \|_\a := \sup_{k \in \Z} \| a(\cdot,k) \|_\a$. For $|j| > \frac12 |k|$, this estimate can be improved: 
\[
\hat z_j(k) 
= \int_\T a(y,k) \, e^{i f(k) \b(y)} e^{-ijy} \, dy 
= \int_\T u(y) e^{i \om (y + p(y))} \, dy, 
\]
with
\[
\om = -j, \quad 
p(y) = - \frac{f(k)}{j} \, \b(y), \quad
u(y) = a(y,k).
\]
Therefore, applying the non-stationary phase argument of Lemma \ref{lemma:non-stationary phase}, 
\begin{equation} \label{stima qjk j high}
|\hat z_j(k)| 
\leq \frac{C_\a}{|j|^\a} \, \big( \| a \|_\a + \| \b \|_{\a+1} \| a \|_1 \big) 
\quad \forall |j| > \frac12|k|
\end{equation}
provided that 
$\| \b \|_2 \leq 1/2$ and 
$| \b |_1 \leq 1/4$.

\medskip

$\bullet$ \emph{Estimate for $\mR_1$. --- } 
We study the composition $\mR_1 \Dx^m$, 
which is the pseudo-differential operator with symbol
\[
\rho_1(x,k) := \sum_{|j| \leq \frac12 |k|} |k|^m \, \hat z_j(k) \, r_N(k,j) \, e^{ijx}.
\]
By \eqref{stima rN(k,j) j small} and \eqref{stima zjk generica}, for any $\a \in \N$,
\begin{align*}
\| \rho_1(\cdot,k) \|_{s_0}^2 
& =  \sum_{|j| \leq \frac12 |k|}  |k|^{2m} \, |\hat z_j(k)|^2 \, |r_N(k,j)|^2 \, \la j \ra^{2s_0} 
\\
& \leq  \sum_{|j| \leq \frac12 |k|} C \la k \ra^{2(m + r - N + (\a/2))} \, \la j \ra^{2(s_0 + N -\a)} 
\big( \| a \|_\a + \| a \|_0 | \b |_\a \big)^2.
\end{align*}
Now, assume that 
\begin{equation} \label{provv 1}
s_0 + N -\a \geq 0.
\end{equation}
Then
\[
\sum_{|j| \leq \frac12 |k|} \la j \ra^{2(s_0 + N -\a)} 
\leq C \la k \ra^{2(s_0 + N -\a) + 1}
\]
because $\la j \ra \leq \la k \ra$ and the number of terms in the sum is $\leq C |k|$. 
Hence
\[
\| \rho_1(\cdot,k) \|_{s_0}^2 
\leq  C \la k \ra^{2(m + r - N + (\a/2) + s_0 + N -\a) + 1} \, \big( \| a \|_\a + \| a \|_0 | \b |_\a \big)^2.
\]
The exponent of $\la k \ra$ is $\leq 0$ if 
$\a \geq 2(m + r + s_0) + 1$. 
Hence fix $\a$ to be the integer 
\[
\a_0 := \min \{ \a \in \N : \, \a \geq 2(m + r + s_0) + 1 \} 
= 2(m + r + s_0) + 1 + \d_0, \quad 0 \leq \d_0 < 1.
\] 
By assumption, 
\[
N \geq 2(m + r + 1) + s_0 > 2(m + r + s_0) + 1 + \d_0 - s_0 = \a_0 - s_0,
\]
therefore $s_0 + N -\a_0 \geq 0$, and \eqref{provv 1} is satisfied.
We get 
$\| \rho_1(\cdot,k) \|_{s_0}
\leq 
C  \big( \| a \|_{\a_0} + \| a \|_0 | \b |_{\a_0} \big)$ so, 
by Lemma \ref{lemma:low freq pseudo-diff}, 
\[ 
\| \mR_1 \Dx^m u \|_s 
\leq C(s) \,\big( \| a \|_{\a_0} + \| a \|_0 | \b |_{\a_0} \big) \, \| u \|_s
\] 
for all $s_0 > 1/2$, and $s \geq 0$. 
Moreover, $|\b|_{\a_0} \leq C \| \b \|_{\a_0 + 1}$, and 
$\a_0 \leq 2(m + r + s_0) + 2$. Therefore 
\begin{equation} \label{stima mR1}
\| \mR_1 \Dx^m u \|_s 
\leq C(s) \,\big( \| a \|_{2(m+r+s_0+1)} + \| a \|_0 \| \b \|_{2(m+r+s_0)+3} \big) \, \| u \|_s.
\end{equation}

\medskip

$\bullet$ \emph{Estimate for $\mR_2$. --- }
Now we study the composition $\mR_2 \Dx^m$, 
which is the pseudo-differential operator with symbol
\[
\rho_2(x,k) := \sum_{|j| > \frac12 |k|} |k|^m \, \hat z_j(k) \, r_N(k,j) \, e^{ijx}.
\]
By \eqref{stima rN(k,j) generica} and \eqref{stima qjk j high}, 
for any $\a \in \N$,
\begin{align*}
\| \rho_2(\cdot,k) \|_{s}^2 
& = \sum_{|j| > \frac12 |k|}  |k|^{2m} \, |\hat z_j(k)|^2 \, |r_N(k,j)|^2 \, \la j \ra^{2s} 
\\
& \leq  \sum_{|j| > \frac12 |k|} 
C |k|^{2m} \, \la j \ra^{2(s + N - \a)} 
\big( \| a \|_\a + \| \b \|_{\a+1} \| a \|_1 \big)^2 
\\
& \leq  C |k|^{2(m + s + N - \a) + 1} \,
\big( \| a \|_\a + \| \b \|_{\a+1} \| a \|_1 \big)^2 
\end{align*}
because
\[
\sum_{|j| > \frac12 |k|} \la j \ra^{2(s + N - \a)} 
\leq C \int_{\frac12 |k|}^{+\infty} t^{2(s + N - \a)} \, dt 
\leq C |k|^{2(s + N - \a)+1}
\]
for $2(s + N - \a)+1 < 0$, namely for 
$\a > s + N + \tfrac12$. 

The exponent of $|k|$ is $\leq 0$ for 
$\a \geq s + N + m + \tfrac12\,$. 
Fix $\a_1 := \min \{ n \in \N : \ n \geq s + N + m + 1\}$. 
Thus 
\[
\| \rho_2(\cdot,k) \|_{s}
\leq 
C \big( \| a \|_{\a_1} + \| \b \|_{\a_1 + 1} \| a \|_1 \big).
\]
By Lemma \ref{lemma:high freq pseudo-diff}, 
\[
\| \mR_2 \Dx^m u \|_s 
\leq C(s) \, 
\big( \| a \|_{\a_1} + \| \b \|_{\a_1 + 1} \| a \|_1 \big) \, \| u \|_{s_0}
\]
for all $s_0 > 1/2$, and $s \geq s_0$.
Moreover, since $\a_1 \leq s + N + m + 2$,
\begin{equation} \label{stima mR2}
\| \mR_2 \Dx^m u \|_s 
\leq C(s) \, 
\big( \| a \|_{s + N + m + 2} + \| \b \|_{s + N + m + 3} \| a \|_1 \big) \, \| u \|_{s_0}.
\end{equation}

\medskip

$\bullet$ \emph{Estimate for $R_N$. --- } 
The sum of \eqref{stima mR1} and \eqref{stima mR2} gives
\begin{multline*}
\| R_N \Dx^m u \|_s 
\leq 
C(s) \Big\{ \Big( \| a \|_{2(m+r+s_0+1)} + \| a \|_0 \| \b \|_{2(m+r+s_0)+3} \Big) \, \| u \|_s
\\
+ \Big( \| a \|_{s + N + m + 2} + \| \b \|_{s + N + m + 3} \| a \|_1 \Big) \, \| u \|_{s_0} \Big\},
\end{multline*}
which is \eqref{final tame estimate}.

\medskip

For $\Dx^r \mH$, the Fourier multiplier is $g(\xi) = - i \, \sgn(\xi) |\xi|^r$ instead of 
$g(\xi) = |\xi|^r$. Therefore there is an additional factor $(- i \, \sgn(k))$ in the formula for $B_\a$. The estimates for the remainder are the same as above. 
\end{proof}

\subsection{With dependence on time}  \label{sec:with time dep}

Assume that the operator $A$ in \eqref{def A torus} depends on time, namely 
the phase space is
$$
\phi(t,x,k) = kx + |k|^{1/2} \b(t,x),
$$ 
where $\b$ is periodic in the time variable $t \in \T$. 
Then the inequality of the previous sections also hold (with minor changes) in spaces $H^s(\T^2)$.
For example: 

\begin{lemma}  \label{lemma:est A with time}
If $\| \b \|_4 \leq \d$ (where $\d \in (0,1)$ is a universal constant), then for all integers 
$s$ 
\begin{equation}  \label{est A with time}
\| A h \|_s \leq C \| h \|_s \quad \forall s = 0,1; 
\qquad 
\| A h \|_s \leq C(s) ( \| h \|_s + \| \b \|_{s+3} \| h \|_1 ) 
\quad \forall s \geq 2,
\end{equation}
for all $h = h(t,x)$, where $\| \ \|_s$ is the norm of $H^s(\T^2)$.
\end{lemma}

\begin{proof} We have already proved that, without dependence on time (i.e. $h = h(x)$, $\b = \b(x)$), 
\begin{equation}  \label{only space}
\| A h \|_{L^2_x} \leq C \| h \|_{L^2_x}, \quad 
\| A h \|_{H^s_x} \leq C(s) (\| h \|_{H^s_x} + \| \b \|_{H^{s+2}_x} \| h \|_{H^1_x})
\end{equation}
provided $\| \b \|_{H^3_x} \leq \d$. 
Now let $h,\b$ depend also on time. For each fixed $t$, 
$\| \b(t) \|_{H^3_x} \leq \| \b \|_{L^\infty_t H^3_x}$ $\leq \| \b \|_{H^1_t H^3_x} \leq \| \b \|_4 \leq \d$, and then \eqref{only space} holds at each $t$. Therefore 
\[
\| A h \|_{L^2_t L^2_x}^2 = \int_\T \| A(t) h(t) \|_{L^2_x}^2 dt 
\leq \int_\T C^2 \| h(t) \|_{L^2_x}^2 dt = C^2 \| h \|_{L^2_t L^2_x}^2, 
\]
i.e. $\| Ah \|_0 \leq C \| h \|_0$. Similarly, for $s \geq 1$, using $\| \b \|_{L^\infty_t H^{s+2}_x} \leq C \| \b \|_{H^1_t H^{s+2}_x}$,  
\begin{align*}
\| A h \|_{L^2_t H^s_x}^2 
& = \int_\T \| A(t) h(t) \|_{H^s_x}^2 \, dt \\
&\leq C(s) \int_\T \{ \| h(t) \|_{H^s_x}^2 + \| \b(t) \|_{H^{s+2}_x}^2 \| h(t) \|_{H^1_x}^2 \} \, dt 
\\ & 
\leq C(s) (\| h \|_{L^2_t H^s_x}^2 + \| \b \|_{H^1_t H^{s+2}_x}^2 \| h \|_{L^2_t H^1_x}^2), 
\end{align*}
whence
\begin{equation} \label{1/2}
\| A h \|_{L^2_t H^s_x} 
\leq C(s) (\| h \|_s + \| \b \|_{s+3} \| h \|_1).
\end{equation}
The norm $\| u \|_s$ of $H^s(\T^2)$ is equivalent to the norm $\| u \|_{L^2_t H^s_x} + \| u \|_{H^s_t L^2_x}$. Then it remains to prove that also 
\begin{equation} \label{2/2}
\| A h \|_{H^s_t L^2_x} \leq C(s) (\| h \|_s + \| \b \|_{s+3} \| h \|_1).
\end{equation}
The time derivative of $A h$ is $\pa_t (Ah) = A( h_t ) + \b_t A (i |D_x|^{1/2} h)$.

$\bullet$ For $s=1$, using the inequalities $\| u \|_{L^\infty(\T^2)} \leq C \| u \|_2$ and $\| A u \|_0 \leq C \| u \|_0$, we have
\begin{align*}
\| \pa_t(Ah) \|_0 
& \leq \| A h_t \|_0 + \| \b_t A (i |D_x|^{1/2} h) \|_0  \\
&\leq C (\| h_t \|_0 + \| \b_t \|_{L^\infty(\T^2)} \| A (i |D_x|^{1/2} h) \|_0)
\\ & 
\leq C (\| h \|_1 + \| \b_t \|_2 \| |D_x|^{1/2} h \|_0)
\leq C (\| h \|_1 + \| \b \|_3 \| h \|_1) 
\end{align*}
(we have rudely worsened $\| h \|_{1/2} \leq \| h \|_1$). 
Therefore \eqref{2/2} holds for $s=1$, and, as a consequence, \eqref{est A with time} holds for $s = 1$, namely $\| A h \|_1 \leq C \| h \|_1$ (because $\| \b \|_4 \leq 1$).   

$\bullet$ For $s=2$, we use the product estimate $\| uv \|_1 \leq C \| u \|_1 \| v \|_2$ 
to deduce that 
\begin{align*}
\| \pa_t (A h) \|_{H^1_t L^2_x} 
& \leq \| \pa_t (A h) \|_1 
\leq \| A h_t \|_1 + \| \b_t A |D_x|^{1/2} h \|_1 
\\ & 
\leq C (\| h_t \|_1 + \| \b_t \|_2 \| A |D_x|^{1/2} h \|_1) \\
&\leq C (\| h \|_2 + \| \b \|_3 \| |D_x|^{1/2} h \|_1) 
\leq C \| h \|_2.
\end{align*}
Therefore \eqref{2/2} holds for $s=2$, and, as a consequence, \eqref{est A with time} holds for $s = 2$. 

$\bullet$ Now assume that \eqref{2/2} holds for some $s \geq 2$; we prove it for $s+1$.
The sum of $\eqref{1/2}$ and $\eqref{2/2}$ implies that \eqref{est A with time} holds at that $s$. 
We estimate 
\[
\| \pa_t (Ah) \|_{H^s_t L^2_x} 
\leq \| A h_t \|_{H^s_t L^2_x} + \| \b_t A (i |D_x|^{1/2} h) \|_{H^s_t L^2_x}.
\]
By \eqref{2/2}, 
$\| A h_t \|_{H^s_t L^2_x} \leq C(s) (\| h_t \|_s + \| \b \|_{s+3} \| h_t \|_1) 
\leq C(s) (\| h \|_{s+1} + \| \b \|_{s+3} \| h \|_2)$ and, by interpolation, this is  
$\leq C(s) (\| h \|_{s+1} + \| \b \|_{s+4} \| h \|_1)$ because $\| \b \|_4 \leq 1$. 
For the other term, we use the product estimate 
$\| uv \|_s \leq C(s) (\| u \|_s \| v \|_2 + \| u \|_2 \| v \|_s)$ 
and \eqref{est A with time} at $s$ to deduce that 
\begin{align*}
& \| \b_t A (|D_x|^{1/2} h) \|_{H^s_t L^2_x}
\leq \| \b_t A (|D_x|^{1/2} h) \|_s
\\ & 
\leq C(s) ( \| \b_t \|_s \| A (|D_x|^{1/2} h) \|_2 + \| \b_t \|_2 \| A ( |D_x|^{1/2} h) \|_s )
\\ & 
\leq C(s) ( \| \b \|_{s+1} \| h \|_3 + \| \b \|_{s+1} \| \b \|_5 \| h \|_2 
+ \| \b \|_3 \| h \|_{s+1} + \| \b \|_3 \| \b \|_{s+3} \| h \|_2 )
\end{align*}
which is $\leq C(s) (\| h \|_{s+1} + \| \b \|_{s+4} \| h \|_1)$ by interpolation. 
Hence \eqref{2/2} holds for $s+1$. 

By induction, we have proved \eqref{2/2} for all $s \geq 1$. 
The sum of \eqref{1/2} and \eqref{2/2} gives the thesis.
\end{proof}

To prove a time-dependent version of Lemma \ref{lemma:composition formula}, 
note that the time derivative of the operator $R_N |D_x|^m$ in \eqref{final tame estimate} 
is
$$
\pa_t (R_N |D_x|^m h) = R_N |D_x|^m h_t + i \b_t R_N |D_x|^{m+(1/2)} h,
$$
namely $R_N |D_x|^m$ satisfies the same formula for the time-derivative as the operator $A$ of Lemma \ref{lemma:est A with time}.

Also note that all these proofs can be easily adapted to include the amplitude function $a(t,x,k)$.

\section{Appendix.}
\label{sec:utilities}

In this appendix we gather some classical facts that are used in the proof.

\subsection{Classical tame estimates for pseudo-differential operators on the torus}
\label{S:tame}

Let 
\begin{equation}
\label{A pseudo-diff}
A u(x) = \sum_{k \in \Z} \hat u_k \, a(x,k)\, e^{ikx}, \quad x \in \R,
\end{equation}
$a(x,k)$ periodic in $x$, for all $k \in \Z$. 

\begin{lemma}[Bounded or regularizing pseudo-differential operators on the torus] 
\label{lemma:reg pseudo-diff}
1) Let $s,\s,\t,K$ be real numbers with $s,\s,K \geq 0$ and $\t > 1/2$. 
Assume that 
\begin{align} 
\label{condiz pseudo reg 1} 
\| a(\cdot\,, k) \|_{s+\s} \, \la k \ra^{\t-s} 
& \leq K \quad \forall k \in \Z,
\\
\label{condiz pseudo reg 2} 
\| a(\cdot\,, k) \|_{\t} \, \la k \ra^\s  
& \leq K \quad \forall k \in \Z.
\end{align}
Then $A$ in \eqref{A pseudo-diff} maps $H^s(\T)$ into $H^{s+\s}(\T)$, with 
\[
\| Au \|_{s+\s} \leq C K \| u \|_s \quad \forall u \in H^s(\T),
\]
where $C>0$ depends on $s+\s$ and $\t$.
2) The same conclusion holds if \eqref{condiz pseudo reg 1} is replaced by 
\begin{equation} \label{condiz pseudo reg 3}
\| a(\cdot\,, k) \|_{s+\s+\t} \, \la k \ra^{-s} 
\leq K \quad \forall k \in \Z.
\end{equation}
\end{lemma}
\begin{proof} 
Develop $a(x,k)$ in Fourier series, 
$a(x,k) = \sum_{j \in \Z} \hat a_j(k) \, e^{ijx}$ so that 
\[
Au(x) = \sum_{n,k} \hat u_k \, \hat a_{n-k}(k) \, e^{inx}.
\]
One has 
\[
\la n \ra^{2(s+\s)} 
\leq C_1 \Big( \la n-k \ra^{2(s+\s)} + \la k \ra^{2(s+\s)} \Big).
\]
Therefore 
$\| Au \|_{s+\s}^2 \leq C_1 \,(S_1 + S_2)$ 
where 
\begin{align*}
S_1 & := 
\sum_{n} \Big( \sum_k |\hat u_k| \, |\hat a_{n-k}(k)| \Big)^2 
\la n-k \ra^{2(s+\s)} \,,
\\ 
S_2 & := 
\sum_{n} \Big( \sum_k |\hat u_k| \, |\hat a_{n-k}(k)| \Big)^2 
\la k \ra^{2(s+\s)} \,.
\end{align*}
By H\"older's inequality, \eqref{condiz pseudo reg 1} implies that
\begin{align*}
S_1 & \leq \sum_{n} \Big( \sum_k |\hat u_k|^2 \, |\hat a_{n-k}(k)|^2 \,
\la k \ra^{2\t} \Big) \Big( \sum_k \frac{1}{\la k \ra^{2\t}} \Big) \, 
\la n-k \ra^{2(s+\s)} 
\\
& \leq C_2 \sum_{k} |\hat u_k|^2 \la k \ra^{2\t} \| a(\cdot\,,k) \|_{s+\s}^2 
\\
& \leq C_2 K^2 \| u \|_s^2 
\end{align*} 
where $C_2 := \sum_j \la j \ra^{-2\t}$ 
is finite because $\t > 1/2$. Similarly, 
one estimates $S_2$ using \eqref{condiz pseudo reg 2} and one obtains that 
$\| Au \|_{s+\s}^2 \leq 2C_1 C_2 K^2 \|u\|_s^2$. 

\smallbreak
To prove part 2) of the lemma, it is sufficient to replace $\la k \ra^\t$ with $\la n-k \ra^\t$ when H\"older's inequality is applied to estimate $S_1$.
\end{proof} 

We now consider paradifferential operators, which are 
pseudo-differential operators with spectrally localized symbols $a(x,k)$ (see \cite{Bony}). 
Namely, develop $a(x,k)$ in Fourier series, 
$a(x,k) = \sum_{j \in \Z} \hat a_j(k) \, e^{ijx}$. We shall consider low (resp.\ high) frequencies 
symbols such that $\hat a_j(k) =0$ for $|j|\geq C |k|$ (resp.\ $|j|\leq C|k|$).

\begin{lemma}[Low frequencies symbol] \label{lemma:low freq pseudo-diff}
Let $A$ be the pseudo-differential operator 
\[
A u(x) = \sum_{k \in \Z} \hat u_k \, a(x,k)\, e^{ikx}, \quad 
a(x,k) = \sum_{|j| \leq C |k|} \hat a_j(k) \, e^{ijx},
\]
where the symbol $a$ is Fourier supported on $\{ j \in \Z : |j| \leq C|k| \}$ for some constant $C$.
Then 
\[
\| A u \|_s \leq C(s) \, \| a \|_{s_0} \| u \|_s, \quad 
\| a \|_{s_0} := \sup_{k \in \Z} \| a(\cdot, k) \|_{s_0},
\]
for all $s_0 > 1/2$, and $s \geq 0$.
\end{lemma}
\begin{proof}
Notice that \eqref{condiz pseudo reg 2} holds with $\tau=s_0>1/2$ and $\sigma=0$. 
Also, for $s=0$, \eqref{condiz pseudo reg 3} holds with $\tau=s_0>1/2$ and $\sigma=0$; which 
in turn implies that  \eqref{condiz pseudo reg 3} holds for any $s\geq 0$ since 
$\| a(\cdot\,, k) \|_{s+s_0}\leq K (s) \| a(\cdot\,, k) \|_{s_0} \, \la k \ra^{s}$ 
in view of the spectral localization. 
\end{proof}

\begin{lemma}[High frequencies symbol] \label{lemma:high freq pseudo-diff}
Let $A$ be the pseudo-differential operator 
\[
A u(x) = \sum_{k \in \Z} \hat u_k \, a(x,k)\, e^{ikx}, \quad 
a(x,k) = \sum_{|j| > C |k|} \hat a_j(k) \, e^{ijx},
\]
where the symbol $a$ is Fourier supported on $\{ j \in \Z : |j| > C|k| \}$ for some constant $C$.
Then 
\[
\| A u \|_s \leq C(s) \, \| a \|_{s} \| u \|_{s_0}, \quad 
\| a \|_{s} := \sup_{k \in \Z} \| a(\cdot, k) \|_{s},
\]
for all $s_0 > 1/2$, and $s \geq s_0$.
\end{lemma}
\begin{proof}
Notice that \eqref{condiz pseudo reg 1} and \eqref{condiz pseudo reg 2} hold 
with $(\tau,\sigma,s)$ replaced by $(s_0,s-s_0,s_0)$. To see this, 
notice that the assumption that $a$ is Fourier supported on $\{ j \in \Z : |j| > C|k| \}$ 
implies that $\| a(\cdot\,, k) \|_{s_1} \la k \ra^{s_2-s_1} \leq K (s_1,s_2) \| a(\cdot\,, k) \|_{s_2} \,$ 
for $s_2\geq s_1$.
\end{proof}

We also recall the following estimates for the Hilbert transform 
(see \cite{IPT} or Lemma~B.5 in \cite{Baldi-Benj-Ono}).

\begin{lemma}\label{benj}
1) Let $s,m_1,m_2$ in $\N$ with $s\ge 2$, $m_1,m_2\ge 0$, $m=m_1+m_2$. 
Let $f\in H^{s+m}(\T)$. Then $[f,\mathcal{H}]u=f\mathcal{H}u-\mathcal{H}(fu)$ 
satisfies
$$
\big\| \partial_x^{m_1}[f,\mathcal{H}]\partial_x^{m_2} u\big\|_s
\le C(s)\big( \| u\|_s\| f\|_{m+2}+\| u\|_2\| f\|_{m+s}\big).
$$

2) There exists a universal constant $\delta$ in $(0,1)$ with the following property. 
Let $s,m_1,m_2$ in $\N$ and set $m=m_1+m_2$, 
$\beta\in W^{s+m+1,\infty}(\T,\R)$ with $|\beta|_1\le \delta$. 
Let $Bh(x)=h(x+\beta(x))$ for $h\in H^s(\T)$. Then
$$
\big\| \partial_x^{m_1}\big( B^{-1}\mathcal{H}B-\mathcal{H}\big)
\partial_x^{m_2} u\big\|_s
\le C(s)\big( |\beta |_{m+1} \| u\|_{s}+|\beta|_{s+m+1} \| u\|_0\big).
$$ 
\end{lemma}

\subsection{Interpolation estimates}
\label{subsec:estimates for exponentials}

Recall the interpolation inequality: if $s_1 \leq s \leq s_2$, then 
\begin{equation}
\label{basic interpolation inequality}
| f |_s \leq C(s_1, s_2) | f |_{s_1}^{\lm} | f |_{s_2}^{1-\lm}, 
\quad 
s = \lm s_1 + (1-\lm) s_2, \quad 
\lm \in [0,1],
\end{equation}
where $| \ |_s$ is either the $C^s$-norm the $H^s$-norm (or other norms of a scale with interpolation).
As a consequence, one has the following
\begin{lemma}[Interpolation]\label{lemma:funda}
Let $n \geq 1$ be an integer, let $\d \geq 0$, and let $\nu_1, \ldots, \nu_n$ be real numbers with 
\[
\nu_j \geq \d \quad \forall j = 1, \ldots, n, \qquad 
\sum_{j=1}^n \nu_j = \a.
\]
Then 
\begin{equation} \label{funda n}	
\prod_{j=1}^n |f|_{\nu_j} \leq C(\a) |f|_{\d}^{n-1} |f|_{\a - \d n + \d}.
\end{equation}
\end{lemma}

\begin{proof}
Since $\sum_{j=1}^n \nu_j = \a$, and each $\nu_j$ is $\geq 1$,
\[
\nu_j = \a - \sum_{i \neq j} \nu_i 
\leq \a - \d(n-1).
\]
So $\nu_j \in [\d, \a - \d n + \d]$. 
Apply \eqref{basic interpolation inequality} with 
$s_1 = 1$, $s = \nu_j$, $s_2 = \a - \d n + \d$, and define $\th_j \in [0,1]$ by
\[
\nu_j = \d \th_j + (\a - \d n + \d) (1-\th_j),  \quad j=1,\ldots,n.
\]
Since $\nu_1 + \ldots + \nu_n = \a$, we find
\[
\Big( \sum_{j=1}^n \th_j \Big) (\a - \d n) 
= n(\a - \d n + \d) - \a 
= (\a - \d n)(n-1).
\]
If $\a - \d n \neq 0$, then 
$
\sum_{j=1}^n \th_j = n-1$ and $\sum_{j=1}^n (1 - \th_j) = 1$. 
Thus, using \eqref{basic interpolation inequality},
\[
\ba
\prod_{j=1}^n |f|_{\nu_j} 
&\leq C(\a) \prod_{j=1}^n |f|_\d^{\th_j} |f|_{\a - \d n + \d}^{1-\th_j} 
= C(\a) |f|_\d^{(\sum \th_j)} |f|_{\a - \d n + \d}^{(\sum(1-\th_j))} \\
&\leq C(\a) |f|_\d^{n-1} \, |f|_{\a- \d n + \d}.
\ea
\]
If, instead, $\a - \d n = 0$, then $\nu_j = \d$ for all $j$, and 
the conclusion still holds.
\end{proof}

The previous lemma, which has an interest per se, can be used to estimate the exponentials.
Let $v(x)$ be a function. The derivatives of $e^v$ are 
\[
\pa_x^\a (e^{v(x)}) = P_\a(x) \, e^{v(x)},
\]
where $P_\a(x)$ satisfies
$P_0(x) = 1$ and $P_{\a+1}(x) = \pa_x P_\a(x) + P_\a(x) \pa_x v(x)$. 
Thus, by induction,
\begin{equation} \label{sum Palpha}
P_\a(x) = \sum_{n=1}^\a \, \sum_{\nu \in S_{\a,n}} C(\nu) (\pa_x^{\nu_1}v)(x) \cdots (\pa_x^{\nu_n}v)(x), \qquad \a \geq 1,
\end{equation} 
where $\nu= (\nu_1, \ldots, \nu_n) \in S_{\a,n}$ means $\nu_j\ge 1$ and 
$\nu_1+\dots+\nu_n = \a$. The previous lemma implies that
\begin{equation*}
|  (\pa_x^{\nu_1}v)(x) \cdots (\pa_x^{\nu_n}v)(x)| 
\leq \prod_{j=1}^n | v|_{\nu_j}\leq C(\a)| v|_1^{n-1}|v|_{\a-n+1}.
\end{equation*} 
Then use \eqref{basic interpolation inequality} with 
$s_1 = 1$, $s = \a + 1 - n$, $s_2 = \a$, namely
$|v|_{\a+1-n} \leq C(\a) |v|_1^{1-\mu} |v|_\a^\mu$ 
with $\mu$ defined by
\[
\a+1-n = (1-\mu) + \a\mu,
\]
which is $\mu = (\a-n)/(\a-1)$.
Since $n-\mu = \a(1-\mu)$, we get
\begin{align*}
|v|_1^{n-1} \, |v|_{\a+1-n} 
&\leq C(\a) |v|_1^{n-1} \, |v|_1^{1-\mu} |v|_\a^\mu 
= C(\a) |v|_1^{n-\mu} |v|_\a^\mu \\
&\leq C(\a) (|v|_1^\a)^{1-\mu} |v|_\a^\mu \\
&\leq C(\a) \big( |v|_1^\a + |v|_\a \big).
\end{align*}
As a consequence, 
\begin{equation} \label{stima Palpha}
|P_\a(x)| \leq C(\a) ( |v|_1^\a + |v|_\a ).	
\end{equation}
In the case $v(x) = i |k|^{1/2} \b(x)$, this gives 
\begin{equation} \label{stima Palpha k beta}
|P_\a(x)| \leq C(\a) \big( |k|^{\a/2} |\b|_1^\a + |k|^{1/2} |\b|_\a \big).	
\end{equation}

\subsection{Non-stationary phase}

The following lemma is the classical fast oscillation estimate, based on 
repeated integrations by parts on the torus, in a tame version.

\begin{lemma}[Non-stationary phase] \label{lemma:non-stationary phase}
Let $p \in H^2(\T,\R)$,  
\begin{equation} \label{b'<1/2}
\| p \|_2 \leq K, \quad 
|p'(x)| \leq \frac12 \quad \forall x \in \R,
\end{equation}
for some constant $K>0$. 
Let $\om$ be an integer, $\om \neq 0$, and let $\a \geq 1$ be an integer. 
If $p \in H^{\a+1}(\T)$, $u \in H^\a(\T)$, then 
\[
\int_\T u(x)\, e^{i\om(x+p(x))}\,dx 
= \Big( \frac{i}{\om} \Big)^\a \int_\T Q_\a(x)\, e^{i\om(x+p(x))}\,dx\,,
\]
where $Q_\a \in L^2(\T)$, 
\[
\| Q_\a \|_0 \leq C(\a,K) \big( \| u \|_\a + \| p \|_{\a+1}\, \|u\|_1 \big),
\]
and $C(\a,K)$ is a positive constant that depends only on $\a$ and $K$. 
If $u=1$, then 
\[
\| Q_\a \|_0 \leq C(\a,K) \| p \|_{\a+1} \,.
\]
If $p=p(x,\om)$ and $u=u(x,\om)$ depend on $\om$, the estimate still holds if $K \geq \|p(\cdot,\om)\|_2$ and $|\pa_x p(x,\om)| \leq 1/2$. 
\end{lemma}

\begin{proof} Put $h(x) := x+p(x)$. 
By induction, integrating by parts $\a$ times gives
\[
\int_\T u(x)\, e^{i\om h(x)}\,dx 
= \Big( \frac{i}{\om} \Big)^{\a} \int_\T 
Q_\a(x) \, e^{i \om h(x)}\, dx,
\]
where, for $\a\geq 1$, $Q_\a$ is of the form
\begin{equation} \label{Q formula 1}
Q_\a = \frac{1}{(h')^{2\a}} \, \sum_{\nu \in S_\a} 
C(\nu) (\pa^{\nu_0}u) (\pa^{\nu_1} h) \ldots (\pa^{\nu_\a} h),
\end{equation}
where $\nu = (\nu_0, \nu_1, \ldots, \nu_\a) \in S_\a$ means 
\begin{equation} \label{nu in Salfa}
 0 \leq \nu_0 \leq \a, \quad 
1 \leq \nu_1 \leq \ldots \leq \nu_\a\,,  \quad
\nu_0 + \nu_1 + \ldots + \nu_\a = 2\a.
\end{equation}
Formula \eqref{Q formula 1} is proved by induction starting from 
$Q_{\a} = \pa_x(Q_{\a-1}/h')$. 
If we organize the sum in \eqref{Q formula 1} according to the number of indices among $\nu_1, \ldots, \nu_\a$ that are equal to 1,  we obtain 
\begin{equation} \label{Q alfa p}
Q_\a = \sum_{n=0}^\a \frac{1}{(h')^{\a+n}} \, \sum_{\mu \in T_{\a,n}}  C(\mu) (\pa^{\mu_0} u)(\pa^{\mu_1} p) \ldots (\pa^{\mu_n} p),
\end{equation}
where $\mu = (\mu_0,\mu_1, \ldots, \mu_n) \in T_{\a,n}$ means 
\[
 0 \leq \mu_0 \leq \a-n, \quad 
2 \leq \mu_1 \leq \ldots \leq \mu_n \,, \quad \mu_0 + \mu_1 + \ldots + \mu_n = \a + n.
\]
To estimate the products in \eqref{Q alfa p}, we distinguish three cases.

Case 1: $n=0$. Then $\mu_0 = \a$ and 
$\| (\pa^{\mu_0} u)(\pa^{\mu_1} p) \ldots (\pa^{\mu_n} p)\|_0 
= \| \pa^\a u \|_0 
\leq \| u \|_\a \,$.

Case 2: $n \geq 1$ and $\mu_0 = 0$. Then
\begin{align*} 
\| (\pa^{\mu_0} u)(\pa^{\mu_1} p) \ldots (\pa^{\mu_n} p) \|_0 
& \leq \| u \|_{L^\infty}  \| \pa^{\mu_1} p \|_{L^\infty} \ldots 
\| \pa^{\mu_{n-1}} p \|_{L^\infty}  \| \pa^{\mu_n} p \|_0 
\\ & 
\leq C^{n-1} \| u \|_1 \| p \|_{\mu_1 +1} \ldots 
\| p \|_{\mu_{n-1} +1}  \| p \|_{\mu_n}
\end{align*}
where $C$ is the universal constant of the embedding $\| u \|_{L^\infty} \leq C \| u \|_1$. 
Now it follows from 
Lemma~\ref{lemma:funda} 
applied with $\delta=2$ and $|\cdot|_s$ replaced with the Sobolev norms 
$\|\cdot\|_s$ (which satisfies the interpolation estimate \eqref{basic interpolation inequality}) 
that
\[
\| p \|_{\mu_1 +1} \ldots 
\| p \|_{\mu_{n-1} +1}  \| p \|_{\mu_n}
\leq C(\alpha) \| p \|_{2}^{n-1} \| p \|_{\a+1}\,.
\]

Case 3: $n, \mu_0 \geq 1$. For any $i\ge 1$, one has
$2 \leq \mu_i \leq \mu_n \leq \a$, because
$$
2(n-1) + \mu_n \leq \mu_1+\ldots + \mu_n = \a + n-\mu_0\le \a+ n-1.
$$ 
Therefore $\mu_i + 1 \leq \a+1$ for all $i\ge 1$, and one can write
\begin{align*} \label{aux}
&\| (\pa^{\mu_0} u)(\pa^{\mu_1} p) \ldots (\pa^{\mu_n} p) \|_0 \\
&\qquad \leq \| \pa^{\mu_0} u \|_0 \| \pa^{\mu_1} p \|_{L^\infty} \ldots 
\| \pa^{\mu_n} p \|_{L^\infty} 
\\ & \qquad
\leq C(n) \| u \|_{\mu_0} \| p \|_{\mu_1+1} \ldots \| p \|_{\mu_n + 1}  \notag
\\ & \qquad
\leq C(n) \| u \|_1^{\th_0} \| u \|_\a^{1-\th_0}\, 
\| p \|_{2}^{\th_1+\ldots+\th_n} 
\| p\|_{\a+1}^{n-(\th_1+\ldots+\th_n)} 
\end{align*}
where $\th_0, \th_1, \ldots, \th_n \in [0,1]$ are defined by 
\begin{align*}
\mu_i + 1 & = 2 \,\th_i + (\a+1)(1-\th_i) \quad \forall i=1,\ldots,n, \\
\mu_0 & = 1 \,\th_0 + \a (1-\th_0).
\end{align*}
Taking the sum gives $\th_0 + \th_1 + \ldots + \th_n = n$ 
because $\mu_0 + \mu_1 + \ldots + \mu_n = \a + n$. One deduces that
\begin{multline*}
\| u \|_1^{\th_0} \| u \|_\a^{1-\th_0}\, \| p \|_{2}^{\th_1+\ldots+\th_n} 
\| p\|_{\a+1}^{n-(\th_1+\ldots+\th_n)} \\
 \leq 
\| p \|_2^{n-1} (\|p\|_2 \|u\|_\a)^{1-\th_0} (\|p\|_{\a+1} \|u\|_1)^{\th_0} 
\end{multline*}
which in turn is smaller than 
\[
\| p \|_{2}^{n-1} (\|p\|_2 \|u\|_\a + \|p\|_{\a+1} \|u\|_1) 
\]
because $a^{1-\th} b^\th \leq (1-\th) a + \th b \leq a+b$ for all $a,b >0$, $\th \in [0,1]$. 

Since $\|1/h'\|_{L^\infty} \leq 2$, collecting all the above cases gives 
\[
\| Q_\a \|_0 \leq C(\a,K) (\|u\|_\a + \|p\|_{\a+1} \|u\|_1)
\]
for some constant $C(\a,K)>0$, because $\| p \|_2 \leq \| \b \|_2 \leq K$.  
Note that all the calculations above are not affected by a possible dependence of $p$ on $\om$. 

When $u=1$, only the case 2 gives a nonzero contribution to the sum.
\end{proof}

\subsection{$L^2(\T)$ as a subspace of $L^2(\R)$}

For the sake of completeness, we recall here how to define 
an isomorphism of $L^2(\T)$ onto a subspace of $L^2(\R)$, transforming Fourier series $\sum_{k \in \Z}$ into Fourier integrals $\int_\R d\xi$.

Consider a $C^\infty$ function $\rho : \R \to \R$ with compact support such that 
\begin{align*}
(i) & \ \ 0 \leq \rho(x) \leq 1 \quad \forall x \in \R, &&&
(iii) & \ \ \rho(x) = 0 \quad \forall |x| \geq 3\p/2, \\
(ii) & \ \ \rho(x) \geq 1/2 \quad \forall x \in [-\p,\p], &&&
(iv) & \ \ \rho(x) = 1 \quad \forall |x| \leq \p/2, \\
(v) & \ \ \sum_{k \in \Z} \rho(x+2\p k) = 1 \quad \forall x \in \R. &&& &
\end{align*}
Let $\mS$ be the multiplication operator $(\mS u)(x) := \rho(x) u(x)$, $u \in L^2(\T)$, and $X$ its range,
\[
X := \{ \mS u : u \in L^2(\T) \} \subset L^2(\R).
\]
The following properties of $\mS$ 
follow directly from the properties of $\rho$.
\begin{lemma}[The isomorphism $\mS$] \label{lemma:cutoff S}
The map $\mS : L^2(\T) \to X$ is bijective, and 
\[
\frac12\, \| u \|_{L^2(\T)} 
\leq \| \mS u \|_{L^2(\R)} 
\leq 2\, \| u \|_{L^2(\T)} 
\quad \forall u \in L^2(\T).
\]
If $u,v \in L^2(\T)$, then 
\begin{equation} \label{magic S}
(u,v)_{L^2(\T)} = (\mS u,v)_{L^2(\R)} 
\end{equation}
(the integral $(\mS u,v)_{L^2(\R)}$ is well-defined because $\mS u$ has compact support). 
In particular, for $v = e_k$, 
\begin{equation}\label{n1110}
\hat u_k = (u,e_k)_{L^2(\T)} = (\mS u,e_k)_{L^2(\R)} = \widehat{(\mS u)}(k) \quad \forall k \in \Z.
\end{equation}
Therefore 
\begin{equation} \label{from series to integral}
\mS \ : \ u(x) = \sum_{k\in\Z} \widehat{(\mS u)}(k) \, e_k(x) 
\ \  \mapsto \ \ 
\mS u(x) = \int_\R \widehat{(\mS u)}(\xi) \, e_\xi(x) \,d\xi.
\end{equation} 
\end{lemma}

We also recall the following version of the Poisson summation formula.

\begin{lemma}[Poisson summation formula]  
\label{lemma:PSF} 
Let $g : \R \to \C$ of class $C^1$, with 
\[
(1+\xi^2)(|g(\xi)| + |g'(\xi)|) \leq C \quad \forall \xi \in \R,
\]
for some constant $C>0$. Then for every $x \in \R$ the following two convergent numerical series coincide:
\[
\mP g(x) := \sum_{k \in \Z} g(x+2\p k) = \sum_{k \in \Z} \hat g(k)\,e^{ikx}\,.
\]
At $x=0$ this is the Poisson summation formula 
$\sum_{k \in \Z} g(2\p k) = \sum_{k \in \Z} \hat g(k)$. 
Moreover, 
\begin{equation} \label{eq:delta Z}
\sum_{k \in \Z} \int_\R e^{i2\p k \xi}\, g(\xi)\,d\xi 
= \sum_{k \in \Z} g(k) \,.
\end{equation}
\end{lemma}

We next show that the operator $\mP$ admits a unique continuous extension to the space of functions 
$g$ such that $(1+|x|^2)^{1/2} g(x) \in L^2(\R)$ (which is equivalent to $g=\hat h$ for some $h \in H^1(\R)$). 

\begin{lemma}  
\label{lemma:T} Let $g \in L^2(\R)$ with $Tg \in L^2(\R)$, where $(Tg)(x) := x g(x)$. 
Then the sequence $\{\hat g(k)\}_{k \in \Z}$ is in $\ell^2(\Z)$, the series $u(x) := \sum_{k \in \Z} \hat g(k)\,e^{ikx}$ belongs to $L^2(\T)$, with
\[
\| u \|_{L^2(\T)} \leq 2 ( \| g \|_{L^2(\R)} +  \|T g \|_{L^2(\R)}).
\]
If, in addition, $g$ satisfies the hypotheses of the previous lemma, then $u(x) = \mP g(x)$ for every $x$, whence
\[
\| \mP g \|_{L^2(\T)} \leq 2 ( \| g \|_{L^2(\R)} +  \|T g \|_{L^2(\R)}).
\]
\end{lemma}
\begin{proof} $\| u \|_{L^2(\T)}^2 = \sum_{k \in \Z} |\hat g(k)|^2$.
For every $\xi\in \R$, 
\[
|\hat g(k)|^2 
\leq \big( |\hat g(\xi)| + |\hat g(k)-\hat g(\xi)| \big)^2 
\leq 2 \big( |\hat g(\xi)|^2 + |\hat g(k)-\hat g(\xi)|^2 \big).
\]
Let $I_k := [k-1/2, k+1/2]$. Then 
\[
|\hat g(k)|^2 = \int_{I_k} |\hat g(k)|^2 \, d\xi 
\leq 2 \int_{I_k} |\hat g(\xi)|^2 \, d\xi 
+ 2 \int_{I_k} |\hat g(k)-\hat g(\xi)|^2 \, d\xi \,.
\]
By H\"older's inequality, 
\begin{align*}
\int_{I_k} |\hat g(k)-\hat g(\xi)|^2 \, d\xi 
& \leq \int_{I_k} \Big| \int_k^\xi |\hat g'(t)| \, dt \Big|^2 \, d\xi 
\\
& \leq \int_{I_k} \Big| \int_k^\xi |\hat g'(t)|^2 \, dt \Big| \, |k-\xi|\, d\xi 
\\
& \leq \int_{I_k} \Big( \int_{I_k} |\hat g'(t)|^2 \, dt \Big)\, d\xi 
= \int_{I_k} |\hat g'(\xi)|^2 \, d\xi.
\end{align*}
Therefore 
\begin{align*}
\sum_{k \in \Z} |\hat g(k)|^2
& \leq 2 \sum_{k \in \Z} \int_{I_k} (|\hat g(\xi)|^2 + |\hat g'(\xi)|^2) \, d\xi \\
& = 2 \int_\R (|\hat g(\xi)|^2 + |\hat g'(\xi)|^2) \, d\xi 
= 2 (\| g \|_{L^2(\R)}^2 + \| Tg \|_{L^2(\R)}^2 ). \qedhere
\end{align*}
\end{proof}

\begin{remark}
Denote by $\Lambda_s$ is the Fourier multiplier with symbol 
$\la \xi\ra^s=(1 + |\xi|^2)^{s/2}$. 
Clearly
\[
\Lambda_s \mP f(x) = \mP \Lambda_s f(x) = \sum_{k \in \Z} \hat f(k) \la k \ra^s e^{ikx}
\]
for all test function $f \in C^\infty_0(\R)$, namely 
$\Lambda_s \mP = \mP \Lambda_s$. 
The previous lemma thus implies that
\[
\| \mP g \|_{H^s(\T)} 
\leq 2 ( \| g \|_{H^s(\R)} +  \|T \Lambda_s g \|_{L^2(\R)}).
\]
Now one has 
$T \Lambda_s  
= \Lambda_s T + s \Lambda_{s-2} \pa_x$, as can be checked directly using the Fourier transform, 
so that
\begin{equation} \label{estimate mP Hs}
\| \mP g \|_{H^s(\T)} 
\leq 
C(s) (\| g \|_{H^s(\R)} + \| T g \|_{H^s(\R)}),
\end{equation}
where $C(s) = 2(1 + |s|)$. 

Note that, for $s$ integer, $\pa_x^s T = T \pa_x^s + s \pa_x^{s-1}$,
and this is useful to calculate $\| T g \|_{H^s(\R)}$ using $\| \pa_x^s T g \|_{L^2(\R)}$.
\end{remark}

Observe that $\mP \mS = I$ because, for $u$ periodic, 
\[
\mP \mS u(x) = \sum_{k \in \Z} (\mS u)(x+2\p k) 
= \sum_{k \in \Z} \rho(x+2\p k) u(x) = u(x).
\]
Moreover, 
\begin{equation} \label{scalar products} 
(\psi,v)_{L^2(\R)} = (\mP \psi,v)_{L^2(\T)} \quad 
\forall \psi \in C^\infty_0(\R), \ v \in L^2(\T).
\end{equation}
Indeed, 
\[ 
(\psi,v)_{L^2(\R)} 
= \int_\R \psi(x) \Big( \sum_{k \in \Z} \overline{\hat v_k}\,e^{-ikx} \Big)\, dx 
= \sum_{k \in \Z} \hat\psi(k)\,\overline{\hat v_k} \,,
\] 
and also 
\[ 
(\mP \psi,v)_{L^2(\T)} 
= \int_\T \Big( \sum_{k \in \Z} \hat \psi(k)\,e^{ikx} \Big) 
\Big( \sum_{j \in \Z} \overline{\hat v_j} e^{-ijx} \Big)\, dx 
= \sum_{k \in \Z} \hat\psi(k)\,\overline{\hat v_k} \,. \qedhere
\]

{\footnotesize 
\bibliographystyle{abbrv}

\def\cprime{$'$}

%\bibliography{bib_AB}
}

\bigskip

\begin{flushright}

Thomas Alazard

D\'epartement de Math\'ematiques et Applications 

Ecole Normale Sup\'erieure 

45 rue d'Ulm, Paris F-75005, France 

\emph{E-mail:} Thomas.Alazard@ens.fr

\bigskip

\bigskip

Pietro Baldi

Dipartimento di Matematica e Applicazioni ``R. Caccioppoli''

Universit\`a di Napoli Federico II 

Via Cintia, 80126 Napoli, Italy

\emph{E-mail:} pietro.baldi@unina.it

\end{flushright}

\end{document}